\documentclass[11pt,leqno,oneside]{amsart}
\usepackage{amsmath,amsfonts,amssymb,amsthm,enumerate,mathtools}
\usepackage{color,framed,mathscinet}
\usepackage[colorlinks, linkcolor=teal, citecolor=violet]{hyperref}
\usepackage[a4paper, left=1.7cm, right=1.7cm, top=3cm, bottom=2cm]{geometry}
\usepackage[numbers,sort&compress]{natbib}
\usepackage{colorprofiles}
\usepackage[a-2u,mathxmp]{pdfx}
\usepackage{doi}
\usepackage{comment}
\usepackage{multirow}
\usepackage{enumitem}
\setlist[enumerate]{label={\rm(\roman*)},leftmargin=6ex}
\reversemarginpar
\newcommand{\R}{\mathbb{R}}
\newcommand{\rn}{{\mathbb{R}^n}}
\newcommand{\N}{\mathbb{N}}
\renewcommand{\AA}{\mathcal{A}}
\newcommand{\CC}{\mathcal{C}}

\newcommand{\EE}{\mathcal{E}}
\newcommand{\FF}{\mathcal{F}}
\newcommand{\GG}{\mathcal{G}}
\newcommand{\HH}{\mathcal{H}}
\newcommand{\MM}{\mathcal{M}}
\newcommand{\RR}{\mathcal{R}}

\newcommand{\med}{\operatorname{med}}
\newcommand{\mv}{\operatorname{mv}}
\newcommand{\m}{\operatorname{m}}
\newcommand{\sgn}{\operatorname{sgn}}

\newcommand{\lgn}{\mathcal L}

\newcommand{\ib}{\frac{1}{\beta}}
\newcommand{\iib}{\frac{2}{\beta}}
\newcommand{\is}{\frac{1}{s}}
\newcommand{\tis}{\tfrac{1}{s}}
\newcommand{\ir}{\frac{1}{r}}
\newcommand{\tir}{\tfrac{1}{r}}

\newcommand{\RG}{(\rn,\gamma_n)}
\newcommand{\Hg}{\HH^{n-1}_{\gamma_n}}
\renewcommand{\d}{{\mathrm d}}
\newcommand{\dgn}{\d\gamma_n}
\newcommand{\dHg}{\d\Hg}
\newcommand{\weight}{\Theta}
\newcommand{\stoo}{\quad\text{as $s\to 0^+$}}
\newcommand{\ttoinf}{\quad\text{as $t\to\infty$}}

\newcommand{\thetab}{\theta_\beta}
\newcommand{\EF}{\varphi}
\newcommand{\expb}{\exp^\beta}
\DeclareMathOperator{\Exp}{\operatorname{Exp}}
\newcommand{\Expb}{\Exp^\beta}
\newcommand{\expLb}{\exp L^\beta\RG}
\newcommand{\Wlgn}{W_{\kern-1.2pt\lgn}\kern0.7pt}
\newcommand{\WexpLb}{\Wlgn\expLb}

\newcommand{\ltk}{\lambda_{t_k}}
\newcommand{\iltk}{\frac{1}{\ltk}}
\newcommand{\lt}{\lambda_t}
\newcommand{\ilt}{\frac{1}{\lt}}
\newcommand{\tilt}{\tfrac{1}{\lt}}
\newcommand{\tolt}{\frac{t_0}{\lt}}

\newcommand{\fn}{h}
\newcommand{\lambdanew}{\xi}
\newcommand{\fno}{f_0}
\newcommand{\fni}{f_1}
\newcommand{\ho}{h_0}
\newcommand{\so}{s_0}
\newcommand{\si}{s_1}
\newcommand{\iso}{\frac{1}{\so}}
\newcommand{\tiso}{\tfrac{1}{\so}}
\newcommand{\isi}{\frac{1}{\si}}

\newcommand{\phto}{\Phi(t_0)}
\newcommand{\phti}{\Phi(t_1)}
\newcommand{\phxi}{\Phi(x_1)}
\newcommand{\phaxi}{\Phi(|x_1|)}
\newcommand{\phk}{\Phi(k)}
\newcommand{\phtk}{\Phi(t_k)}
\newcommand{\AU}{\underline{A}}
\newcommand{\aU}{\underline{a}}

\newcommand{\NV}{\xi}

\let\bar\overline
\let\tilde\widetilde

\usepackage{xspace}
\makeatletter
\DeclareRobustCommand\onedot{\futurelet\@let@token\@onedot}
\def\@onedot{\ifx\@let@token.\else.\null\fi\xspace}
\def\eg{e.g\onedot}

\def\ae{a.e\onedot} 
\makeatother

\newtheoremstyle{MyPlain}{}{}{\itshape}{}{\bfseries}{.}{5pt plus 4pt minus 3pt}{\thmname{#1}\thmnumber{ #2}\thmnote{ \textbf{[#3]}}}
\theoremstyle{MyPlain}
\newtheorem{theorem}{Theorem}[section]
\newtheorem{theoremalph}{Theorem}

\newtheorem{lemma}[theorem]{Lemma}
\newtheorem{proposition}[theorem]{Proposition}
\newtheorem{corollary}[theorem]{Corollary}
\newtheoremstyle{MyRemark}{}{}{\upshape}{}{\bfseries}{.}{5pt plus 1pt minus 1pt}{}
\theoremstyle{MyRemark}

\numberwithin{equation}{section}

\expandafter\let\expandafter\oldproof\csname\string\proof\endcsname
\let\oldendproof\endproof
\renewenvironment{proof}[1][\proofname]{%
  \oldproof[{{\bf #1.}}]%
}{\oldendproof}

\makeatletter
\newcommand{\onorm}{\@ifstar\@onorms\@onorm}
\newcommand{\@onorms}[1]{%
	\left|\mkern-1.5mu\left|\mkern-1.5mu\left|
	#1
	\right|\mkern-1.5mu\right|\mkern-1.5mu\right|
}
\newcommand{\@onorm}[2][]{%
  \mathopen{#1|\mkern-1.5mu#1|\mkern-1.5mu#1|}
  #2
  \mathclose{#1|\mkern-1.5mu#1|\mkern-1.5mu#1|}
}
\makeatother
\newcommand{\normxY}[2]{\onorm*{#1}_{{#2}\left(s,\frac12\right)}}
\newcommand{\normxB}[1]{\normxY{#1}{L^{\tilde B}}}
\newcommand{\normwB}{\normxY{\weight}{L^{\tilde B}}}

\makeatletter
\def\paragraph{\bigskip\@startsection{paragraph}{4}%
  \z@\z@{-\fontdimen2\font}%
  {\normalfont\bfseries}}
\makeatother

\begin{document}

\title{Sharp exponential inequalities for the Ornstein-Uhlenbeck operator}

\begin{abstract}
The optimal constants in a class of exponential type inequalities for the
Ornstein-Uhlenbeck operator in the Gauss space are detected. The existence of
extremal functions in the relevant inequalities is also established. Our
results disclose analogies and dissimilarities in comparison with Adams'
inequality for the Laplace operator, a companion of our inequalities in the
Euclidean space.
\end{abstract}

\author{%
	Andrea Cianchi\textsuperscript{1}
	\and
	V\'\i t Musil\textsuperscript{1}
	\and
	Lubo\v s Pick\textsuperscript{2}
}
\address{%
	\textsuperscript{1}Dipartimento di Matematica e Informatica ``Ulisse Dini'',
	University of Florence,
	Italy
}
\address{
	\textsuperscript{2}Department of Mathematical Analysis,
	Faculty of Mathematics and Physics,
	Charles University,
	Czech Republic
}

\email{andrea.cianchi@unifi.it}



\urladdr{%
	0000-0002-1198-8718 {\rm(Cianchi)},
	0000-0001-6083-227X {\rm(Musil)},
	0000-0002-3584-1454 {\rm(Pick)}
}



\maketitle


\section*{How to cite this paper}
\noindent
This paper has been accepted for publication in \emph{Journal of Functional
Analysis} and is available on
\begin{center}
	\url{https://doi.org/10.1016/j.jfa.2019.108341}.
\end{center}
Should you wish to cite this paper, the authors would like to cordially ask you
to cite it appropriately.

\section{Introduction and main results}\label{intro}

The Ornstein-Uhlenbeck operator $\lgn$ is defined as
\begin{equation} \label{E:lgn-def}
	\lgn u = \Delta u - x \cdot \nabla u
\end{equation}
for a function $u\colon\rn\to\R$, with $n\in\N$. Here, $\Delta $ and $\nabla$
denote the usual Laplace and gradient operators, and the dot ``$\cdot$'' stands
for scalar product in $\rn$. The Ornstein--Uhlenbeck operator  is, under various respects, the
natural counterpart of the Laplace operator when the Euclidean ambient space
$\rn$, equipped with the Lebesgue measure, is replaced by the Gauss space
$\RG$. The latter  is still $\rn$, but endowed with the Gauss probability
measure $\gamma_n$, whose density obeys
\begin{equation} \label{E:phi-def}
	\dgn (x) = ( 2\pi)^{-\frac n2} e^{-\frac{|x|^2}{2}}\d x
		\quad\text{for $x\in\rn$}.
\end{equation}

The operator $\lgn$ is the infinitesimal generator of the Ornstein-Uhlenbeck
semigroup in the Gauss space, defined via the Mehler kernel, see \eg
\citep[Section 12.1]{Lun:15} or \citep[Chapter 2]{Urb:19}.  Hence, the   operator  $\lgn$ stands with
respect to the Ornstein-Uhlenbeck semigroup in the Gauss space that the Laplace
operator stands to the heat kernel in the Euclidean space.  Also, recall that
the classical Dirichlet integral is the Dirichlet form associated with the
Laplace operator in the Euclidean space; likewise, the functional
\begin{equation}
	\int_{\rn} |\nabla u|^2\, \dgn
\end{equation}
is the Dirichlet form associated with the Ornstein-Uhlenbeck operator in the
Gauss space.

The Ornstein-Uhlenbeck operator plays a role in a number of areas and is the
subject of a huge literature. The lecture notes  \citep{Lun:15}, the monograph
\citep{Urb:19} and the survey papers \citep{Sjo:97} and \citep{Bog:18} are
excellent sources for an introduction to these topics, as well as for a rich
collection of related references.

Here, we are concerned with a peculiar family of Sobolev type inequalities
involving the operator $\lgn$ in $\RG$. The bases for the study of Sobolev
inequalities in the Gauss space were laid  by L.~Gross \citep{Gro:75}, who
proved a first-order inequality for the $L^2\RG$ norm of the gradient. The work
of Gross paved the way to extensive researches on Sobolev type inequalities in
the Gauss space \citep{Bar:06, Bar:08, Bob:98, Bra:07, Bob:97, Car:01, Cip:00,
Fei:75, Fuj:11, Mil:09, Pel:93, Rot:85}. Inequalities in terms of the
Ornstein-Uhlenbeck operator can be found in \citep{Bet:02,Bla:07, Tia:10}.
The latter papers deal, in fact, with even more general second-order elliptic
operators, but are concerned with the somewhat different situation of functions
defined in open subsets of $\rn$ and vanishing on $\partial\Omega$.

Interestingly, the target spaces in the Gaussian Sobolev
inequalities for the Ornstein-Uhlenbeck operator substantially differ from
those appearing in the Euclidean inequalities for the Laplace operator.
A~distinctive feature of the former is that, as in the case  of first-order
inequalities, the gain in the degree of integrability inherited by a function
$u$ from that of $\lgn u$ in the space $\RG$ is much weaker than that
guaranteed by $\Delta u$ in domains of finite Lebesgue measure in $\rn$.

In this paper, we focus on inequalities for the operator $\lgn$ in spaces
of exponential type $\expLb$, with $\beta>0$. These are Orlicz spaces built
upon Young functions equivalent to $\smash{e^{t^\beta}}$ for $t$ near infinity, and
equipped with the Luxemburg norm, denoted by $\|\cdot\|_{\exp L^\beta\RG}$. In
these borderline spaces, the increase in the integrability of a function $u$
ensured by the integrability of $\lgn u$ deteriorates so that membership of
$\lgn u$ to $\expLb$ just ensures that $u$ belongs to the same space.
Specifically, the Sobolev type inequality
\begin{equation} \label{apr3}
	\|u\|_{\expLb}
		\le C \|\lgn u\|_{\expLb}
\end{equation}
holds for some constant $C=C(\beta)$, and for every function $u\in\Wlgn\expLb$
such that
\begin{equation} \label{E:mu}
	\m(u)=0,
\end{equation}
where $\m(u)$ stands for either the mean value $\mv(u)$ or the median $\med(u)$
of $u$ over $\RG$. Here, $\Wlgn \expLb$ denotes the Sobolev type space of those
functions $u$ such that $\lgn u$, defined in a suitable weak sense, belongs to
$\expLb$ - see Section~\ref{S:FS} for details.  Moreover, the target space in
inequality \eqref{apr3} is optimal, in the sense that the inequality  fails if
the norm in the space $\expLb$ is replaced by any stronger
rearrangement-invariant norm of $u$ on the left-hand side, see
\citep[Section~6]{Cia:20a}. Let us point out that, however, the target space in
\eqref{apr3} is still better than that entering parallel inequalities where
$\lgn u$ is substituted by $\nabla u$, or even by $\nabla ^2u$, the matrix   of
all second-order derivatives of $u$. The optimal target spaces in the relevant
inequalities are still of exponential type, but with an exponent smaller than
$\beta$, thus exhibiting a loss of integrability for $u$ with respect to
$\nabla u$ \cite[Proposition~4.4 (iii)]{Cia:09} (see also
\citep{Aid:94,Bob:99,Led:95} for special cases) or $\nabla ^2u$
\citep[Corollary~7.14 (ii)]{Cia:15}.

The stronger effect of the operator $\lgn$ is apparently due to the
presence of the term $x\cdot\nabla u$ and to the interaction of the function
$x$ with the decay of the density \eqref{E:phi-def} of the measure $\gamma_n$
near infinity.

Our specific concern is the identification of the optimal constant $\theta$ in
the integral inequality, equivalent to \eqref{apr3},
\begin{equation} \label{E:int}
	\sup_u \, \int_{\rn} \expb(\theta |u|)\,\dgn
		< \infty,
\end{equation}
where the supremum is extended over all functions $u$ in $\rn$ satisfying
the constraint
\begin{equation} \label{E:lgn-integral}
	\int_{\rn} \Expb(|\lgn u|)\,\dgn
		\le M
\end{equation}
for some constant $M>1$, and subject  to the normalization \eqref{E:mu}. Here,
$\expb$ denotes the function defined by $\smash{\expb(t)=e^{t^\beta}}$ for $t\ge 0$,
and $\Expb$ its convex envelope, namely the largest convex function not
exceeding $\expb$. Obviously, $\Expb$ agrees with $\expb$ near infinity for
every $\beta>0$, and globally if $\beta\ge 1$.

Problem \eqref{E:mu}--\eqref{E:lgn-integral} can be regarded as a Gaussian
analogue of (a special case of) that solved by D.~R.~Adams for the classical
Laplacian in the Euclidean setting \citep{Ada:88}---see also the related
contributions \citep{Alb:08,Fon:11,Fon:12,Lam:13,Lu:15,Ruf:13, Mas:14}.
Adams' result is in turn a second-order version of Moser's inequality
\citep{Mos:70} in the limiting case of the Sobolev embedding theorem. A
first-order companion to problem \eqref{E:mu}--\eqref{E:lgn-integral}, where
the Ornstein-Uhlenbeck operator is replaced by the plain gradient in the Gauss
space, has recently been addressed in \citep{Cia:20b}.

A trait shared by all the results on exponential inequalities alluded to above, and by their numerous
variants and extensions, is the existence of a threshold value in the exponential integrand of $u$, which
dictates the validity or not of the inequality in question. This phenomenon
also shapes the problem at hand here. The threshold for the constant $\theta$
in \eqref{E:int} only depends on $\beta$, and equals
\begin{equation} \label{E:thetab}
	\thetab = \iib.
\end{equation}
What distinguishes our results  about problem
\eqref{E:mu}--\eqref{E:lgn-integral} from those of \citep{Ada:88} and
\citep{Mos:70} in the Euclidean setting are its validity for the limiting value
$\thetab$ for $\theta$ in \eqref{E:int} and the impact of the constant $M$ in
\eqref{E:lgn-integral}.

In this connection, recall that Adams' inequality tells us that, if $\Omega$ is
an open bounded subset of $\rn$, with $n\ge 3$, then
\begin{equation}\label{supadams}
	\sup_u \, \int_{\Omega} \exp^{\frac {n}{n-2}}(\alpha_n |u|)\, \d x
		< \infty,
\end{equation}
where
\begin{equation*}
	\alpha_n = \frac{1}{\omega_{n}}
		\left(
			\frac{4\pi^{\frac{n}{2}}}{\Gamma\left(\frac{n}{2}-1\right)}
		\right)^{\frac{n}{n-2}},
\end{equation*}
$\omega_{n}$ is the Lebesgue measure of the unit ball in $\rn$, $\Gamma$ denotes
the gamma function, and the supremum is extended over all compactly supported
functions $u$ in $\Omega$ such that
\begin{equation} \label{supadams1}
	\int_{\Omega} |\Delta u|^{\frac n2} \,\d x
		\le 1.
\end{equation}
Both constants $\alpha_n$ on the left-hand side of inequality
\eqref{supadams} and 1 on the right-hand side of
\eqref{supadams1} are sharp: inequality \eqref{supadams} fails if either of
them is increased, while the other one is unchanged.

By contrast, the conclusions about inequality \eqref{E:int} are multifaceted,
and sensitive to the parameter $\beta$, but independent of the dimension $n$.
They can be summarized as follows. If $\beta\in(0,1]$, then inequality
\eqref{E:int} holds with $\theta\le\frac2\beta$ for any choice of $M$, and
fails, again for any $M$, when $\theta>\frac2\beta$. On the other hand, if
$\beta>1$, then for any $M$ inequality \eqref{E:int} holds with
$\theta<\frac2\beta$ and fails with $\theta>\frac 2\beta$; the value of $M$ has
a role only when $\theta=\frac2\beta$, in which case inequality \eqref{E:int}
holds for sufficiently small $M$ and does not hold if $M$ is too large. This is
the content of our first main result.

\begin{theorem}[Integral form] \label{T:integral-form}
Let $n \ge 1$.
\par \noindent
{\rm Part 1}. Assume that $\beta\in(0,1]$.
\begin{enumerate}[label={\rm(1.\roman*)}]
	\item If $0<\theta\le\frac2\beta$, then inequality \eqref{E:int} holds for
	every $M>1$.
	\item If $\theta>\frac2\beta$, then inequality \eqref{E:int} fails for every
	$M>1$. In particular, there exists a function $u$ obeying  \eqref{E:mu} and
	\eqref{E:lgn-integral} that makes the integral in \eqref{E:int} diverge.
\end{enumerate}
{\rm Part 2}.  Assume that $\beta\in(1,\infty)$.
\begin{enumerate}[label={\rm(2.\roman*)}]
	\item If $0<\theta<\frac2\beta$, then inequality \eqref{E:int} holds for every
	$M>1$.
	\item If $\theta=\frac2\beta$, then there exists $M>1$ such that inequality
	\eqref{E:int} holds, and there exists $M>1$ such that
	\eqref{E:int} fails.
	\item If $\theta>\frac2\beta$, then inequality \eqref{E:int} fails for every
	$M>1$. In particular, there exists a function $u$ obeying \eqref{E:mu} and
	\eqref{E:lgn-integral} that makes the integral in
	\eqref{E:int} diverge.
\end{enumerate}
\end{theorem}

A variant of problem  \eqref{E:mu}--\eqref{E:lgn-integral} is the subject of
the next theorem, where constraint \eqref{E:lgn-integral} is replaced by its
norm-form twin. The resultant inequality reads
\begin{equation} \label{E:norm}
	\sup_u \, \int_{\rn} \expb(\theta |u|)\,\dgn
		< \infty,
\end{equation}
where  the supremum is extended over all functions $u$ in $\rn$ fulfilling the
inequality
\begin{equation} \label{E:lgn-luxemburg}
	\|\lgn u\|_{L^B\RG}\le 1,
\end{equation}
and condition \eqref{E:mu}. Here, $B$ is any Young function such that
\begin{equation} \label{BN}
	B(\tau)=Ne^{\tau^\beta}
		\quad\text{for $\tau>\tau_0$,}
\end{equation}
for some constants $N>0$ and $\tau_0>0$.

The flavour of our result on this problem is similar to that of
Theorem~\ref{T:integral-form}, save that now the job of the parameter $M$ in
\eqref{E:lgn-integral} is performed by the behaviour near zero of the function
$B$ in \eqref{BN}.

\begin{theorem}[Norm form] \label{T:norm-form}
Let $n \ge 1$.
\par\noindent{\rm Part 1}. Assume that $\beta\in (0,1]$.
\begin{enumerate}[label={\rm(1.\roman*)}]
	\item If $0<\theta\le \frac 2 \beta$, then inequality \eqref{E:norm} holds for
	every $N>0$ and for every Young function $B$ as in \eqref{BN}.
	\item If $\theta>\frac 2\beta$, then inequality \eqref{E:norm} fails for every
	$N>0$ and for every Young function $B$ as in \eqref{BN}. In particular, there
	exists a function $u$ obeying  \eqref{E:mu} and \eqref{E:lgn-luxemburg}
	that makes the integral in \eqref{E:norm} diverge.
\end{enumerate}
{\rm Part 2}.  Assume that $\beta\in(1,\infty)$.
\begin{enumerate}[label={\rm(2.\roman*)}]
	\item If $0<\theta<\frac2\beta$, then inequality \eqref{E:norm} holds for every
	$N>0$ and for every Young function $B$ as in \eqref{BN}.
	\item If $\theta=\frac2\beta$, then for every $N>0$, there exist a Young
	function $B$ as in \eqref{BN}  such that inequality \eqref{E:norm} holds, and
	a Young function $B$ such that inequality \eqref{E:norm} fails.
	\item If $\theta>\frac2\beta$, then inequality \eqref{E:norm} fails for every
	$N>0$ and for every Young function $B$ as in \eqref{BN}. In particular, there
	exists a function $u$ obeying  \eqref{E:mu} and \eqref{E:lgn-luxemburg}
	that makes the integral in \eqref{E:norm} diverge.
\end{enumerate}
\end{theorem}

A question that naturally arises is the existence of extremal functions in the
inequalities considered so far. Namely, the question of whether the supremum in
\eqref{E:int} and \eqref{E:norm} is attained or not.  We give an affirmative
answer to this question in the critical case when $\theta=\frac2\beta$.
Accordingly, we consider values of  the parameter $\beta\in(0,1]$. In the light
of Theorems~\ref{T:integral-form} and \ref{T:norm-form}, this guarantees
that inequality \eqref{E:int} or \eqref{E:norm} holds for any constant $M$ or
any Young function $B$ as in \eqref{BN}, respectively. As will be clear from
the proof, the problem in the subcritical regime when $\theta<\frac2\beta$ is
easier, and can be approached along the same lines.

\begin{theorem}[Existence of maximizers] \label{T:maximizers}
Let $n\ge 1$ and $\beta\in(0,1]$.
\begin{enumerate}
	\item The supremum in \eqref{E:int} is attained for $\theta=\frac2\beta$ and
	for every $M>1$.
	\item The supremum in \eqref{E:norm} is attained for $\theta=\frac2\beta$ and
	for every Young function $B$ as in \eqref{BN}.
\end{enumerate}
\end{theorem}

With regard to the Euclidean setting, the existence of extremals in Moser's
first-order inequality is well known: the $2$-dimensional case goes back to
\citep{Flu:92}, whereas an exhaustive proof for arbitrary dimensions $n$ has
only recently been accomplished in \cite{Csa:21}, where, in particular,
arguments from the earlier paper \citep{Lin:96} are made rigorous.  Instead, a
parallel question for  Adams' inequality \eqref{supadams} for the Laplace
operator  seems to be only solved in the special case when $n=4$ - \citep{Lu:09}.

We conclude by tackling the limiting situation when $\beta$ is formally sent to
infinity in condition \eqref{E:lgn-luxemburg}, namely when $\lgn u$ is subject
to the constraint
\begin{equation} \label{E:lgn-inf}
	\|\lgn u\|_{L^\infty\RG}
		\le 1.
\end{equation}
The constant $\thetab$ degenerates to $0$ in the limit as $\beta\to\infty$.
Such a  behaviour hints that some singular phenomenon should be expected. This
is in fact the case, since the piece of information contained in
\eqref{E:lgn-inf} does not imply that $u\in L^\infty\RG$. It turns out that,
instead,
\begin{equation*}
	\|u\|_{\exp \exp L\RG} \le C \|\lgn u\|_{L^\infty\RG}
\end{equation*}
for some constant $C$, and for every function $u$ fulfilling \eqref{E:mu}.
Here, $\exp\exp L\RG$ denotes the Orlicz space built upon a Young function
equivalent to $e^{e^t}$ for $t$ near infinity, and is optimal among all
rearrangement-invariant target spaces - see~\citep{Cia:20a}.

The problem thus emerges of detecting the values of the constant $\eta$ such
that
\begin{equation} \label{E:infty}
	\sup_u \, \int_{\rn} \exp\exp(\eta |u|)\,\dgn
		< \infty,
\end{equation}
where now the supremum is extended over all functions $u$ in $\rn$
satisfying assumption \eqref{E:lgn-inf} and normalized as in \eqref{E:mu}.

A new threshold value appears, which is given by
\begin{equation}\label{E:etainf-def}
	\eta_\infty = 2.
\end{equation}
As shown by our last main result, this value is not admissible in inequality
\eqref{E:infty}.

\begin{theorem}[$L^\infty$ norm] \label{T:L-infty}
Let $n \ge 1$.
\begin{enumerate}
	\item If $0<\eta<2$, then inequality \eqref{E:infty} holds.
	\item If $\eta \ge 2$, then inequality \eqref{E:infty} fails.
	In particular, there exists a function $u$ obeying \eqref{E:lgn-inf} and
	\eqref{E:mu} that makes the integral in \eqref{E:infty} diverge.
\end{enumerate}
\end{theorem}

Extremal functions $u$ in inequality \eqref{E:infty} can be shown to exist for
every $\eta\in(0,2)$. As in the case of inequality \eqref{E:norm} with
$\theta<\frac2\beta$, a proof of this fact is analogous to, and simpler than
that of Theorem~\ref{T:maximizers}, and will be omitted.

\section{Outline of the approach}\label{S:outline}

Here, we sketch  the main ideas and ingredients employed in the proofs of our
results, and point out some  technical difficulties. Precise definitions of
some mathematical objects mentioned below are given in the subsequent sections.

\subsection*{Gaussian isoperimetric inequality}

The point of departure of our approach is a sharp
pointwise estimate, in rearrangement form, for  every
admissible  function $u\colon\rn\to\R$ in terms of $\lgn u$. Such an estimate
in turn rests upon the isoperimetric inequality in the Gauss space, which tells
us that half-spaces are the (only) minimizers for Gauss perimeter among all
measurable subsets of $\rn$ with prescribed Gauss measure \cite{Bor:75,
Sud:74}.  Recall that the Gauss perimeter $P_{\gamma_n}(E)$ of a measurable set
$E\subset \R^n$ can be defined by
\begin{equation*}
	P_{\gamma_n}(E) =
		\Hg(\partial^M E),
\end{equation*}
where
\begin{equation*}
	\dHg(x) = (2\pi)^{-\frac n2}
		 e^{-\frac{|x|^2}2}\,\d\HH^{n-1}(x),
\end{equation*}
$\partial^M E$ denotes the essential boundary of $E$ and $\HH^{n-1}$ is the
$(n-1)$-dimensional Hausdorff measure.  An analytic formulation of the Gaussian
isoperimetric inequality makes use of the isoperimetric function $I$, also
called isoperimetric profile, of the Gauss space. The function $I$ is the
largest function that renders the inequality
\begin{equation}\label{isopineq}
	I(\gamma_n (E)) \le P_{\gamma_n}(E)
\end{equation}
true for every measurable set $E\subset \R^n$.

\subsection*{Rearrangement estimate}
The rearrangement
estimate in question reads
\begin{equation} \label{E:pointwise-1}
	0 \le u^\circ(s) - u^\circ(\tfrac 12)
		\le \weight(s) \int_{0}^{s} \left( \lgn u \right)^*_+(r)\,\d r
			+ \int_{s}^{\frac{1}{2}} \left( \lgn u \right)^*_+(r)\,\weight(r)\,\d r
		\quad\text{for $s\in\left(0,\tfrac{1}{2}\right]$}
\end{equation}
and
\begin{equation} \label{E:pointwise-2}
	0 \le u^\circ(\tfrac 12) - u^\circ(1-s)
		\le \weight(s) \int_{0}^{s} \left( \lgn u \right)^*_-(r)\,\d r
			+ \int_{s}^{\frac{1}{2}} \left( \lgn u \right)^*_-(r)\,\weight(r)\,\d r
		\quad\text{for $s\in\left(0,\tfrac{1}{2}\right]$}.
\end{equation}
Here, $u^\circ\colon(0,1)\to\R$ denotes the signed decreasing rearrangement of
$u$ with respect to the Gauss measure, and $ \left(\lgn u\right)^*_+$ and $
\left(\lgn u\right)^*_-$ the decreasing rearrangements of the positive and the
negative parts of $\lgn u$, respectively, with respect to the same measure.
Moreover, $\smash{\weight\colon\left(0,\frac{1}{2}\right]\to[0,\infty)}$ is the
function defined  as
\begin{equation} \label{E:weight-def}
	\weight(s) = \int_{s}^{\frac{1}{2}} \frac{\d r}{I(r)^2}
		\quad\text{for $s\in\left(0,\tfrac{1}{2}\right]$}.
\end{equation}

Bounds in the spirit of  \eqref{E:pointwise-1} and \eqref{E:pointwise-2} are
rooted in the work of Maz'ya \citep{Maz:61,Maz:69} and Talenti \citep{Tal:76},
who made use of isoperimetric inequalities to estimate solutions to boundary
value problems for classes of elliptic equations---including the Laplace
equation---in subsets of the Euclidean space. Results in a similar vein for
solutions to Dirichlet problems, with homogeneous boundary conditions, for
elliptic equations on subsets of the Gauss space are the subject of
\citep{Bet:02}.
The specific inequalities \eqref{E:pointwise-1}--\eqref{E:pointwise-2}
can be proved  following  ideas of \citep{Bet:02} and
\citep{Cia:90}. A detailed proof can be found in~\citep[Theorem~3.2]{Cia:20a}.
Let us mention that the same inequalities hold for
more general linear partial differential operators enjoying the same
ellipticity property, with respect to the Gauss measure, as the Ornstein-Uhlenbeck
operator. Hence, results parallel to those stated in Section \ref{intro} hold
for these operators. We limit ourselves to dealing with the Ornstein-Uhlenbeck
operator for ease of presentation.

The special form of the extremal sets in the Gaussian isoperimetric inequality
entails that inequalities \eqref{E:pointwise-1}--\eqref{E:pointwise-2} hold as
equalities, provided that the function $u$ only depends  on one of the
coordinates of $x\in\rn$, and $\lgn u$ is  monotone
and odd in this variable. Critical features of this kind of functions $u$ in
connection with this property are the fact that the level sets of both $u$ and
$\lgn u$ are   half-spaces---the isoperimetric sets in the Gaussian
isoperimetric inequality---and the constancy of their gradient on the boundary
of their level sets.  \par Since both the functionals of $u$, appearing in
\eqref{E:int}, \eqref{E:norm} and \eqref{E:infty}, and the
constraints on $\lgn u$ prescribed by \eqref{E:lgn-integral},
\eqref{E:lgn-luxemburg}, \eqref{E:lgn-inf}, are rearrangement-invariant,
inequalities \eqref{E:pointwise-1}--\eqref{E:pointwise-2}  enable us to reduce
the original problems to one-dimensional inequalities. Let us emphasize that no
loss of information occurs in the involved constants after this reduction,
inasmuch as inequalities  \eqref{E:pointwise-1}--\eqref{E:pointwise-2} turn
into equalities for the class of functions described above.

\subsection*{Sharp Orlicz-H\"older inequality}
The proof of the resultant one-dimensional inequalities exploits a H\"older
type inequality for couples of functions in Orlicz spaces. The precise constant
in this inequality is again critical. To this purpose, we have to resort to a
special form of the H\"older  inequality   in Orlicz spaces, where two
different kinds of norms---the Luxemburg and the Orlicz norm---are applied to
the relevant functions. We are then led to the problem of evaluating the
function
\begin{equation} \label{june1}
	(0, \tfrac 12) \ni s \mapsto \normwB,
\end{equation}
where $\weight$ is the function given by \eqref{E:weight-def}, and
$\smash{\normxY{\cdot}{L^{\tilde B}}}$ stands for the Orlicz norm on $\smash{(s,\tfrac
12)}$
built upon the Young conjugate $\widetilde
B$ of a Young function $B$ which agrees with $\smash{e^{t^\beta}}$, or with $\smash{e^{e^t}}$,
near infinity. This is a very subtle task, inasmuch as neither $\weight$ nor
$\widetilde B$ take an explicit form. In addition, the Orlicz norm of a
function is not just defined via an integral. Instead, it requires the solution
of a constrained minimization problem for an integral functional.  This
notwithstanding, we are able to derive the first  two terms of an asymptotic
expansion of function~\eqref{june1} as $s \to 0^+$, which suffices for our
purposes. Very precise estimates are required in this derivation.  With the
asymptotic expansion at our disposal, the proof of inequality \eqref{E:norm}
follows quite easily.
On the other hand, problem \eqref{E:int}--\eqref{E:lgn-integral} is reduced
to~\eqref{E:norm}--\eqref{E:lgn-luxemburg}.
The proof of inequality \eqref{E:infty} relies upon an
analogous argument.

\subsection*{Optimality} 
As suggested by the considerations above, the proof of the optimality of our
results exploits suitable trial functions $u\colon\rn\to\R$ depending  on one
variable only. Heuristically, these functions should be chosen in  such a way
that the equality holds in the Orlicz-H\"older inequality mentioned above, for
each $\smash{s\in(0,\tfrac 12)}$. This requirement, however, prevents $\lgn u$ from
belonging to the prescribed Orlicz space. In order to restore membership in the
latter space, a truncation argument is introduced, which results in trial
functions $u$ attaining the equality in the H\"older inequality at least
asymptotically. Since $\lgn$ is a second-order differential operator, plain
truncation of functions cannot be employed, and has to be replaced by a smooth
truncation operation. In order not to violate the integral or the norm
constraint on $\lgn u$, an ad hoc smooth truncation operator has to be
introduced.

\subsection*{Maximizers}
The question of the existence of extremals, for $\beta \in (0, 1]$, in the
inequalities discussed in Theorems  \ref{T:integral-form} and \ref{T:norm-form}
displays some peculiar features. Their proofs rest upon a compactness argument,
which, besides standard lower semicontinuity theorems, applies thanks to a
uniform integrability property of maximizing sequences of functions. The
uniform integrability  is guaranteed by a slightly enhanced integrability
result, of independent interest, for functions fulfilling condition
\eqref{E:lgn-integral} or \eqref{E:lgn-luxemburg}. Indeed, we show that the
supremum in  \eqref{E:int} or \eqref{E:norm}, respectively, for the threshold
value $\theta = \tfrac 2 \beta$, is still finite even if  the integrand
$\smash{e^{(\frac 2\beta |u|)^\beta}}$ is multiplied by a function $\varphi(|u|)$,
where $\varphi (t)$ diverges to infinity at a sufficiently
mild rate when $t\to\infty$.  Such a result holds in sharp contrast with
inequality \eqref{supadams}, for which any improvement in this direction is
known to fail. This is possible thanks to the special nature of the exponential
type norms and measure entering the game. Let us stress that this augmentation
neither contradicts the optimality of the target space $\expLb$ in inequality
\eqref{apr3}, nor the sharpness of the constant $\frac 2\beta$ in \eqref{E:int}
or \eqref{E:norm}. Actually, under the condition to be imposed on $\varphi$,
the function $\smash{e^{(\frac 2\beta t)^\beta}}\varphi (t)$ is still equivalent to
$\smash{e^{(\frac 2\beta t)^{\beta}}}$, in the sense
of Young functions. Moreover, under the same condition, the function
$\smash{e^{(\theta t)^\beta}}$ grows faster than $\smash{e^{(\frac 2\beta t)^\beta}}\varphi
(t)$ near infinity for every $\theta>\frac2\beta$.

\subsection*{Article structure}
We begin by recalling the necessary
background on functions and function spaces in the next section. Section
\ref{asymp} has a technical content: it is devoted to the asymptotic expansion
of the function in \eqref{june1} and of related functions. The reduction to
one-dimensional problems is performed in Section \ref{basic}. Section
\ref{sharpness} deals with the smooth truncation operator, which is   applied
in the construction of trial functions showing the optimality of our results.
The proofs of Theorems \ref{T:integral-form}--\ref{T:L-infty} are
accomplished in Section \ref{proofmain}. In the final Section~\ref{final}, we
establish the improved integrability property, which is the key step   in the
proof of  Theorem \ref{T:maximizers}.

\section{Function spaces}\label{S:FS}

\subsection*{Rearrangements}

Let $(\RR,\nu)$ be a probability space, namely a~measure space $\RR$ endowed
with a~probability measure~$\nu$. Assume  that $(\RR,\nu)$ is non-atomic. In
fact, we shall just be concerned with the case when $\RR$ is either  $\rn$
endowed with the Gauss measure ${\gamma_n}$, or $(0,1)$ endowed with the
Lebesgue measure. In the latter case, the measure will   be omitted in the
notation. More generally, we shall simply write $\RR$ instead of $(\RR, \nu)$
when no ambiguity can arise.  The notation $\MM(\RR,\nu)$ is employed for the space
of real-valued, $\nu$-measurable functions on $\RR$. Also, $\MM_+(\RR,\nu)$ stands
for the subset of its nonnegative functions.

Let $\phi\in\MM(\RR,\nu)$. The decreasing rearrangement $\phi^*\colon[0,1]
\to[0,\infty]$ of $\phi$ is given by
\begin{equation*}
	\phi^*(s)
		= \inf\{t\ge0: \nu\left(\{x\in\RR : |\phi(x)|>t\}\right)\le s\}
		\quad\text{for $s\in [0,1]$.}
\end{equation*}
Similarly, the signed decreasing rearrangement $\phi^\circ\colon [0,1]
\to[-\infty, \infty]$ of $\phi$ is defined as
\begin{equation*}
	\phi^\circ(s)
		= \inf\{t\in\R: \nu(\{x\in\RR : \phi(x)>t\})\le s\}
		\quad\text{for $s\in[0,1]$.}
\end{equation*}
Note that, in particular,
\begin{equation} \label{E:equiintegrability-with-signed}
	\mv (\phi)
		= \int_{\RR} \phi(x)\, \d \nu
		= \int_{0}^{1}\phi^\circ(s)\,\d s
\end{equation}
for every function $\phi\in L^{1}(\R,\nu)$.
Also, we set
\begin{equation} \label{median}
	\med (\phi) = \phi^\circ (\tfrac 12)
\end{equation}
for every function $\phi\in\MM(\RR,\nu)$.

\subsection*{Orlicz Spaces}

A Young function $A\colon [0,\infty)\to[0,\infty]$ is a convex (nontrivial) function such
that $A(0)=0$. The Orlicz space $L^{A}(\RR,\nu)$ built upon the function $A$ is
defined as
\begin{equation*}
	L^{A}(\RR,\nu) =
		\left\{
			\phi\in\MM(\RR,\nu):
				\int_{\RR} A\left(\frac{|\phi|}{\lambda}\right)\d\nu < \infty
			\quad\text{for some $\lambda>0$}
		\right\}.
\end{equation*}
The space $L^{A}(\RR,\nu)$ is a Banach space with respect to the Luxemburg norm
defined by
\begin{equation*}
	\|\phi\|_{L^{A}(\RR,\nu)} = \inf
		\left\{
			\lambda>0:
				\int_{\RR} A\left(\frac{|\phi|}{\lambda}\right)\d\nu\le 1
		\right\}
\end{equation*}
for a function $u\in L^{A}(\RR,\nu)$.
The Luxemburg norm is equivalent to the Orlicz norm given by
\begin{equation*}
	\onorm{\phi}_{L^{A}(\RR,\nu)} = \sup
		\biggl\{
			\int_{\RR} \phi\psi\,\d\nu:
				\int_{\RR} \tilde{A}(|\psi|)\,\d\nu\le 1
		\biggr\}
\end{equation*}
for a function $\phi\in L^{A}(\RR,\nu)$.
Here, $\widetilde A\colon[0,\infty)\to[0,\infty]$
denotes the Young conjugate of $A$, defined as
\begin{equation*} 
	\widetilde A(t)
		= \sup\{\tau t-A(\tau): \tau\ge 0\}
	\quad\text{for $t\ge 0$},
\end{equation*}
which is also a Young function.  Notice that, if
$a\colon[0,\infty)\to[0,\infty]$ is the  non-decreasing left-continuous
function such that
\begin{equation*}
	A(t) = \int_0^t a(\tau)\, \d \tau
	\quad\text{for $t\ge 0$,}
\end{equation*}
then $\widetilde A$ admits the representation formula
\begin{equation*}
	\widetilde A (t) = \int_0^ta^{-1}(\tau)\, \d\tau
		\quad\text{for $t\ge 0$,}
\end{equation*}
where $a^{-1}$ denotes the (generalized) left-continuous inverse of $a$.
Young's inequality tells us that
\begin{equation}\label{E:Young-ineq}
	t \tau \le A(t) + \widetilde A(\tau)
		\quad \text{for $t, \tau \ge 0$,}
\end{equation}
and follows from the very definition of Young conjugate. Equality holds in
\eqref{E:Young-ineq} if and only if either $t=a^{-1}(\tau)$ or $\tau=a(t)$.

Both the Luxemburg norm and the Orlicz norm are rearrangement invariant. Hence,
\begin{equation*}
	\|\phi\|_{L^{A}(\RR,\nu)}
		= \|\phi^*\|_{L^{A}(0,1)}
		= \|\phi^\circ\|_{L^{A}(0,1)}
			\quad\text{and}\quad
	\onorm{\phi}_{L^{A}(\RR,\nu)}
		= \onorm{\phi^*}_{L^{A}(0,1)}
		= \onorm{\phi^\circ}_{L^{A}(0,1)}
\end{equation*}
for every function $\phi\in L^{A}(\RR,\nu)$.

A sharp form of the H\"older inequality in Orlicz spaces reads
\begin{equation} \label{E:Orl-Holder}
	\int_{\RR} \phi\psi\,\d\nu
		\le \|\phi\|_{L^A(\RR, \nu)} \onorm{\psi}_{L^{\tilde A}(\RR, \nu)}
\end{equation}
for  $\phi \in L^A(\RR, \nu)$ and $\psi\in L^{\tilde A}(\RR, \nu)$.

If $\phi\in L^{A}(\RR,\nu)$ and $E\subset\RR$ is a measurable set, we use the
abridged notations
\begin{equation*}
	\|\phi\|_{L^A(E)}
		= \|\phi\chi_E\|_{L^A(\RR,\nu)}
	 \quad\text{and}\quad
	 \onorm{\phi}_{L^A(E)}
	 	= \onorm{\phi\chi_E}_{L^A(\RR,\nu)}.
\end{equation*}
In particular,
\begin{equation} \label{E:Orl-char}
	\onorm{1}_{L^A(E)}
		= \nu(E)\tilde{A}^{-1}\bigl(1/\nu(E)\bigr),
\end{equation}
where $\tilde{A}^{-1}$ denotes the (generalized) right-continuous inverse of
$\tilde A$.

\subsection*{The isoperimetric function in Gauss space}

The Gaussian isoperimetric function $I$, also called the Gaussian isoperimetric
profile, of the space $\RG$  plays a pivotal role in our approach. Its name
stems from the fact that it governs  the isoperimetric inequality in Gauss
space - see~\citep{Bor:75, Sud:74}. The function $I\colon[0,1]\to[0,\infty)$
obeys
\begin{equation} \label{E:I-def}
	I(s) = \frac{1}{\sqrt{2\pi}} e^{-\frac{\Phi^{-1}(s)^2}{2}}
		\quad\text{ for $s\in(0,1)$,}
\end{equation}
and $I(0)=I(1)=0$, where $\Phi\colon\R\to(0,1)$ is the function defined as
\begin{equation} \label{E:Phi-def}
	\Phi(t) = \frac{1}{\sqrt{2\pi}} \int_{t}^{\infty} e^{-\frac{\tau^2}{2}}\,\d\tau
		\quad\text{for $t\in\R$.}
\end{equation}
Note that
\begin{equation} \label{E:relation-between-I-and-Phi}
	I\bigl(\Phi(t)\bigr) = -\Phi'(t) \quad\text{for  $t\in\R$.}
\end{equation}

\subsection*{Sobolev type spaces}

The Sobolev space $W^{1,2}\RG$ is defined as
\begin{equation*}
	W^{1,2}\RG
		= \bigl\{u\in L^2\RG:\, \text{$u$ is weakly differentiable and}
				\, |\nabla u|\in L^2\RG\bigr\}.
\end{equation*}
Similarly,
\begin{equation*}
	W^{2,2}\RG
		= \bigl\{u\in L^2\RG:\, \text{$u$ is twice weakly differentiable and}
				\, |\nabla u|, |\nabla ^2 u|\in L^2\RG\bigr\}.
\end{equation*}
The operator $\lgn$ is defined on functions $u\in W^{2,2}\RG$ via equation
\eqref{E:lgn-def}. Moreover, one has that $\lgn\colon W^{2,2}\RG\to L^2\RG$,
and
\begin{equation} \label{E:weak-formulation'}
	\int_{\rn} \nabla u \cdot \nabla v\,\dgn
		= - \int_{\rn}v\, \lgn u \,\dgn
\end{equation}
for every  $\phi\in W^{1,2}\RG$, see \eg\citep[Theorem~13.1.3]{Lun:15}.

As customary, equation \eqref{E:weak-formulation'}  enables one to extend the
operator $\lgn$ outside its natural domain, and to define it on all functions
$u\in W^{1,2}\RG$ such that there exists a function $f\in L^2\RG$ fulfilling
\begin{equation} \label{E:weak-formulation}
	\int_{\rn} \nabla u \cdot \nabla v\,\dgn
		= - \int_{\rn}v\, f \,\dgn
\end{equation}
for every  $v\in W^{1,2}\RG$. This function space will be denoted by $\Wlgn
L^2\RG$.  On setting
\begin{equation*}
	\lgn u = f
\end{equation*}
for $u\in\Wlgn L^2\RG$, one thus has that
$\lgn\colon\Wlgn L^2\RG\to L^2\RG$.

More generally, given an Orlicz space $L^A\RG$ contained in $L^2\RG$, we define
\begin{equation*} 
	\Wlgn L^A\RG
		= \{u\in\Wlgn L^2\RG: \lgn u\in L^A\RG\}.
\end{equation*}

\section{Asymptotic expansions}\label{asymp}

A substantial problem to be faced in attacking inequalities \eqref{E:int} and
\eqref{E:norm} is the computation of certain Orlicz  norms of functions
depending on the function $\Phi$ and on the exponent $\beta$. The norms in
question do not admit an expression in a closed form. However, we can to
describe their asymptotic behaviour in  an accurate form that enables us to
circumvent this problem.  The asymptotic expansions which come into play  in
this connection are collected in this section.

Given a function  $\FF$ defined in some neighbourhood of a point
$t_0\in[-\infty,\infty]$, and $k\in\N$, we write
\begin{equation} \label{E:expansion}
	\FF(t) = \EE_1(t) + \cdots + \EE_k(t) + \cdots
		\quad\text{as $t\to t_0$}
\end{equation}
to denote that
\begin{equation*}
	\lim_{t\to t_0} \frac{\FF(t)}
		{\EE_1(t)} = 1
			\quad\text{if  $k=1$},
\end{equation*}
and
\begin{equation*}
	\lim_{t\to t_0} \frac{\FF(t) - [\EE_1(t) + \cdots + \EE_{j}(t)]}
		{\EE_{j+1}(t)} = 1
			\quad\text{for $1\le j\le k-1$, otherwise.}
\end{equation*}

If $\FF(t)=\EE_1(t)+\EE_2(t)+\EE_3(t)+\cdots$ as $t\to t_0$ and $\EE_1(t)$
is positive on a neighbourhood of $t_0$, then for every $\sigma\in\R$, one has
\begin{equation} \label{E:power}
	\FF(t)^\sigma
		= \EE_1(t)^\sigma + \sigma \EE_1(t)^{\sigma-1} \EE_2(t)
			+ \cdots
		\quad\text{as $t\to t_0$}.
\end{equation}
Furthermore
\begin{equation} \label{E:log-of-expansion}
	\log\bigl(\FF(t)\bigr)
		= \log\bigl(\EE_1(t)\bigr) + \frac{\EE_2(t)}{\EE_1(t)} + \cdots
		\quad\text{as $t\to t_0$}.
\end{equation}

The behaviour of the function $\Phi$ defined by \eqref{E:Phi-def} was analyzed
in \citep[Lemma~5.1]{Cia:20b}. It tells us that
\begin{equation} \label{E:logPhi}
	\log \frac 1{\Phi(t)} = \frac{t^2}{2} +\log t + \cdots
		\ttoinf
\end{equation}
and
\begin{equation} \label{E:Phi-prime}
	-\Phi'(t) = t\Phi(t) + \frac{\Phi(t)}{t} + \cdots
		\ttoinf.
\end{equation}

The following asymptotic expansion for the isoperimetric function $I$ can be
derived through equations \eqref{E:relation-between-I-and-Phi} and
\eqref{E:logPhi}.

\begin{lemma} \label{L:I-expansion}
Let $I$ be the function defined by \eqref{E:I-def}. Then
\begin{equation} \label{E:I-expansion}
	I(s) = s \sqrt{2\log\tis}
					- \frac{s\log\log\is}{2\sqrt{2\log\is}}
					+ \cdots
		\stoo.
\end{equation}
\end{lemma}

The content of the lemma below is an estimate for an expression involving the
isoperimetric function $I$, which will be exploited on various occasions.

\begin{lemma} \label{L:I-inequality}
Let $I$ be the function defined by \eqref{E:I-def}. Then
\begin{equation} \label{E:I-inequality}
	\frac{s}{I(s)^2} \ge \frac{1}{2s\log\is}
		\quad\text{for $s\in(0,\tfrac12]$}.
\end{equation}
\end{lemma}

Inequality \eqref{E:I-inequality} appears in \cite[Lemma 4.2]{Bob:99a}. We provide a proof for completeness.

\begin{proof}
Since the function $\Phi\colon[0,\infty)\to(0,\frac 12]$ is bijective,
equation~\eqref{E:relation-between-I-and-Phi} ensures that inequality
\eqref{E:I-inequality} is equivalent to
\begin{equation} \label{E:I-inequality-equiv}
	2\Phi(t)^2\log\frac{1}{\Phi(t)} - \Phi'(t)^2 \ge 0
		\quad\text{for $t\ge 0$}.
\end{equation}
Owing to expansions \eqref{E:logPhi} and \eqref{E:Phi-prime}, it is easily
verified that the left-hand side of \eqref{E:I-inequality-equiv} tends to zero
as $t\to\infty$. Therefore, inequality \eqref{E:I-inequality-equiv} will follow
if we show that the function on its left-hand side  is
decreasing. The derivative of this function equals
\begin{equation} \label{E:I-inequality-diff}
	2\Phi'(t)\left[ 2\Phi(t)\log\frac{1}{\Phi(t)} - \Phi(t) + t\Phi'(t) \right]
		\quad\text{for $t>0$.}
\end{equation}
Notice that here  we have made use of the equality $\Phi''(t)=-t\Phi'(t)$ for
$t\in\R$. Denote the function in the square bracket in
\eqref{E:I-inequality-diff} by $F(t)$. Inasmuch as $\Phi'(t)<0$ for  $t>0$, it
suffices to show that $F(t)>0$ for $t>0$. We have that
\begin{equation} \label{E:F-diff}
	F'(t) = \Phi'(t)\left[ 2\log\frac{1}{\Phi(t)} - 2 -t^2 \right]
		\quad\text{for $t>0$}.
\end{equation}
Let us analyze the sign of $F'$. Denote the function in the square bracket in
\eqref{E:F-diff} by $G(t)$.  We claim that $G$ is increasing on $(0,\infty)$.
Indeed,
\begin{equation*} 
	G'(t) = -2\frac{\Phi'(t)}{\Phi(t)} - 2t
		\quad\text{for $t>0$}.
\end{equation*}
Hence, $G'(t)>0$ for $t>0$, thanks to the inequality $-\Phi'(t)>t\Phi(t)$ for
$t>0$, see \citep[Lemma~3.4]{Cia:11}. Furthermore,
\begin{equation*}
	\lim_{t\to 0^+} G(t) = 2(\log 2 - 1) < 0.
\end{equation*}
On the other hand, owing to equation \eqref{E:logPhi},
\begin{equation*}
	\lim_{t\to\infty} G(t)
		= \lim_{t\to\infty} \left[ \left( t^2 + 2\log t + \cdots \right) - 2 - t^2 \right]
		= \infty.
\end{equation*}
Thereby, there exists a unique $t_0>0$ satisfying $G(t_0)=0$. Consequently,
$F$ is increasing on $(0,t_0)$ and decreasing on $(t_0,\infty)$. Since
\begin{equation*}
	\lim_{t\to 0^+} F(t) = \log 2 - \frac{1}{2} > 0
		\quad\text{and}\quad
	\lim_{t\to\infty} F(t) = 0,
\end{equation*}
we can thus conclude that $F(t)>0$ for $t>0$.
\end{proof}

An asymptotic expansion near zero for the function $\Theta$, defined by \eqref{E:weight-def}, is stated in
the next lemma, and can be deduced via Lemma~\ref{L:I-expansion}, equation  \eqref{E:power}
and L'H\^opital's rule.

\begin{lemma} \label{L:weight-expansion}
Let $\weight$ be the function defined by \eqref{E:weight-def}.
Then
\begin{equation} \label{E:weight-expansion}
	\weight(s)
		= \frac{1}{2s\log\is}
			- \frac{\log\log\is}{4s\left( \log\is \right)^2}
			+ \cdots
		\stoo.
\end{equation}
\end{lemma}

The following lemma is a consequence of  Lemma~\ref{L:weight-expansion}, and
concerns the asymptotic behaviour near zero of the function
$\Lambda\colon(0,\tfrac12]\to[0,\infty)$ given by
\begin{equation} \label{E:Lambda-def}
	\Lambda(s) = \int_{s}^{\frac12} \weight(r)\,\d r + s\weight(s)
		\quad\text{for $s\in(0,\tfrac12]$.}
\end{equation}

\begin{lemma} \label{L:Lambda-expansion}
Let $\Lambda$ be the function defined by   \eqref{E:Lambda-def}.
Then
\begin{equation*}
	\Lambda(s) = \frac{1}{2}\log\log\is + \cdots
		\stoo.
\end{equation*}
\end{lemma}

The next result is contained in \citep[Lemma~3.4]{Cia:20c}.

\begin{lemma} \label{L:a-invers}
Let $\beta>0$, $N>0$ and let $B$ be a Young function obeying \eqref{BN}.
Let $b\colon[0,\infty)\to[0,\infty)$ be the left-continuous function such that
\begin{equation} \label{E:Bint}
	B(t) = \int_{0}^t b(\tau)\,\d\tau
		\quad\text{for $t\ge 0$}.
\end{equation}
Then
\begin{equation}\label{E:asymptotic-b-invers}
	b^{-1}(t)
		= (\log t)^\ib
		+ \frac{1-\beta}{\beta^2}
			(\log t)^{\ib-1}\log\log t
		- \frac{\log N\beta}{\beta} (\log t)^{\ib-1}
		+ \cdots
	\ttoinf.
\end{equation}
\end{lemma}

From equation \eqref{E:power} and Lemma~\ref{L:weight-expansion}, one obtains
the following lemma.

\begin{lemma} \label{L:second-term-expansion}
Let $\beta>0$, $N>0$ and let $B$ be a Young function
obeying \eqref{BN}. Then
\begin{equation} \label{E:second-term-exansion}
	\weight(s)sB^{-1}\left( \tis \right)
		= \frac{1}{2}\left( \log\tis \right)^{\ib-1}
			- \frac{1}{4}\left( \log\tis \right)^{\ib-2} \log\log\tis
			+ \cdots
	\stoo.
\end{equation}
\end{lemma}

Lemmas~\ref{L:I-expansion} and \ref{L:weight-expansion} enable one to derive
the result below.

\begin{lemma} \label{L:Iwi2-expansion}
Let $I$ and $\weight$ be the functions defined by \eqref{E:I-def} and
\eqref{E:weight-def}. Then
\begin{equation} \label{E:Iwi2-expansion}
	I\left( \weight^{-1}(t) \right)^2
		= \frac{1}{2t^2\log t}
			- \frac{5\log\log t}{4t^2(\log t)^2}
			+ \cdots
		\quad\text{as $t\to\infty$}.
\end{equation}
\end{lemma}

Combining Lemmas~\ref{L:a-invers} and \ref{L:Iwi2-expansion} yields the
expansion which is the subject of the next result.

\begin{lemma}\label{L:asymptotic-of-product}
Let $\beta>0$, $N>0$ and let $B$ be a Young function
obeying \eqref{BN}. Let $b$ be the function appearing in \eqref{E:Bint}. Let $I$ and $\weight$ be the functions defined by \eqref{E:I-def} and \eqref{E:weight-def}.
Assume that $\lambda>0$.   Then
\begin{equation*} 
	b^{-1}(\lambda t)\,tI\bigl(\weight^{-1}(t)\bigr)^{2}
			= \frac{1}{2t}\left(\log t\right)^{\ib-1}
				+ \left(\frac{1-\beta}{2\beta^{2}}-\frac{5}{4}\right)
					\frac{1}{t}(\log t)^{\ib-2}\log\log t  	
				+ \cdots
		\ttoinf.
\end{equation*}
\end{lemma}

Given $\sigma>0$, let $\Psi_\sigma\colon(1,\infty)\to(0,\infty)$
be the function defined by
\begin{equation} \label{E:Psi}
	\Psi_\sigma(t)
		= \int_{1}^{t} \frac{(\tau-1)^\sigma}{\tau}\,\d\tau
		\quad\text{for $t>1$}.
\end{equation}

Elementary considerations yield the following asymptotic expansion for
$\Psi_\sigma$ as $t\to\infty$.

\begin{lemma} \label{L:Psi}
Let $\sigma>0$ and let $\Psi_\sigma$ be the function defined by \eqref{E:Psi}.
Then
\begin{equation*}
	\Psi_\sigma(t)
		= \frac{1}{\sigma} t^\sigma
			- \begin{cases}
					\frac{\sigma}{\sigma-1} t^{\sigma-1} + \cdots
						& \text{if $\sigma\in(1,\infty)$}
						\\
					\log t + \cdots
						& \text{if $\sigma=1$}
						\\
					c + \cdots
						& \text{if $\sigma\in(0,1)$}
				\end{cases}
		\ttoinf,
\end{equation*}
for some constant $c$   depending on $\sigma$.
\end{lemma}

The next result provides us with a formula for  Orlicz norms, which
generalizes~\citep[Lemma~3.5]{Cia:20c}.

\begin{lemma} \label{L:Orlicz-norm-general}
Let $B$ be a finite-valued Young function of the form \eqref{E:Bint}, with $b$
strictly increasing and such that $b(0)=0$. Assume that the function $h$ belongs to $\MM_+(0,\frac 12)$
and does not vanish identically. Then,
\begin{equation} \label{E:Orlicz-norm-general}
	\normxB{\fn}
		= \int_{s}^{\frac12} b^{-1}\bigl( \lambdanew_s\,\fn(r) \bigr) \fn(r)\,\d r
		\quad\text{for $s\in(0,\tfrac12)$},
\end{equation}
where $\lambdanew_s>0$ is uniquely defined by
\begin{equation} \label{E:lambda-condition}
	\int_{s}^{\frac12} B\left(b^{-1}\bigl( \lambdanew_s\,\fn(r) \bigr)\right)\d r
		= 1\quad\text{for $s\in(0,\tfrac{1}{2})$.}
\end{equation}
\end{lemma}

\begin{proof}
Let $s\in(0,\tfrac12)$. By the definition of the Orlicz norm,
\begin{equation} \label{E:Orl-norm-w-def}
	\normxB{\fn}
		= \sup\left\{ \int_{s}^{\frac12} f(r)\fn(r)\,\d r:
			\int_{s}^{\frac12} B(|f(r)|)\,\d r \le 1\right\}.
\end{equation}
Let $\lambdanew_s>0$ and $f\in\MM_+(0,\frac 12)$. By Young's inequality
\eqref{E:Young-ineq},
\begin{equation*}
	\int_{s}^{\frac12} \fn(r) f(r)\,\d r
		\le \int_{s}^{\frac12} B\left( \frac{f(r)}{\lambdanew_s} \right)\d r
			+ \int_{s}^{\frac12} \tilde B\bigl( \lambdanew_s\fn(r) \bigr)\d r.
\end{equation*}
Let us define
\begin{equation*}
	f_s(r) = \lambdanew_s\,b^{-1}\bigl(\lambdanew_s \fn(r)\bigr)
		\quad\text{for $r\in(s,\tfrac12)$}.
\end{equation*}
By the equality cases in Young's inequality \eqref{E:Young-ineq},
\begin{equation*}
	f_s\, \fn
		= \lambdanew_s \fn b^{-1}(\lambdanew_s \fn)
		= B\left( b^{-1}(\lambdanew_s \fn) \right)
			+ \tilde B(\lambdanew_s \fn),
\end{equation*}
whence
\begin{equation} \label{E:fsw}
	f_s\, \fn
		= B\left( \frac{f_s}{\lambdanew_s} \right)
			+ \tilde B(\lambdanew_s \fn).
\end{equation}
Now, assume that $\lambdanew_s$ obeys \eqref{E:lambda-condition}, namely
\begin{equation} \label{E:int-of-fs-over-ls}
	\int_{s}^{\frac12} B\left( \frac{f_s(r)}{\lambdanew_s} \right)\d r = 1.
\end{equation}
Observe that, since $B(b^{-1})$ is strictly increasing, there exists a unique
$\lambdanew_s>0$ fulfilling condition \eqref{E:lambda-condition}. Integrating
both sides of equation \eqref{E:fsw} over $(s,\frac 12)$ and recalling
equation \eqref{E:int-of-fs-over-ls} yield
\begin{equation*}
	\int_{s}^{\frac12} f_s(r)\fn(r)\,\d r
		= 1 + \int_{s}^{\frac12} \tilde B\bigl(\lambdanew_s \fn(r)\bigr)\,\d r,
\end{equation*}
whence
\begin{equation*}
	\sup\left\{ \int_{s}^{\frac12} f(r)\fn(r)\,\d r:
		\int_{s}^{\frac12} B\left(\frac{f}{\lambdanew_s} \right) \le 1\right\}
	=  \int_{s}^{\frac12} f_s(r)\fn(r)\,\d r.
\end{equation*}
Therefore, owing to \eqref{E:Orl-norm-w-def},
\begin{align*}
	\normxB{\fn}
		& = \sup\left\{ \int_{s}^{\frac12} \frac{f(r)}{\lambdanew_s}\fn(r)\,\d r:
			\int_{s}^{\frac12} B\left(\frac{f}{\lambdanew_s} \right)\le 1\right\}
			\\
		& = \frac{1}{\lambdanew_s} \int_{s}^{\frac12} f_s(r)\fn(r)\,\d r
			= \int_{s}^{\frac12} b^{-1}\bigl(\lambdanew_s\fn(r)\bigr)\fn(r)\,\d r.
\end{align*}
Equation \eqref{E:Orlicz-norm-general} hence follows.
\end{proof}

In the special case when $\fn=\weight$, the conclusions of
Lemma~\ref{L:Orlicz-norm-general} can be rephrased as in the next lemma.  In
what follows, given $t>0$, we denote by $\lambda_{t}$ the unique positive
number such that
\begin{equation} \label{E:lambda-t-condition}
	\int_{0}^{t} B\bigl(b^{-1}(\lambda_t \tau)\bigr)\,
		I\bigl(\weight^{-1}(\tau)\bigr)^2\,\d \tau
		= 1.
\end{equation}

\begin{lemma} \label{L:Orlicz-norm-weight}
Let $B$ be as in Lemma~\ref{L:Orlicz-norm-general}, let $\weight$ be given by
\eqref{E:weight-def} and $\lambda_{\Theta(s)}$ by \eqref{E:lambda-t-condition},
with $t=\Theta (s)$.  Then
\begin{equation} \label{E:Orlicz-norm-weight}
	\normwB
		= \int_{0}^{\weight(s)} b^{-1}\left( \lambda_{\weight(s)} \rho \right)\,
			\rho I\bigl(\weight^{-1}(\rho)\bigr)^2\,\d \rho
		\quad\text{for $s\in(0,\tfrac12)$}.
\end{equation}
\end{lemma}

\begin{proof}
For each $s\in(0,\frac{1}{2})$, the function $\weight$ is strictly decreasing
in $(s,\tfrac12)$, and maps this interval onto $(0,\weight(s))$. The
conclusion thus follows from Lemma~\ref{L:Orlicz-norm-general}, on choosing
$\fn=\weight$, and making the change of variable $\rho=\weight(r)$ in the
integrals in \eqref{E:Orlicz-norm-general} and \eqref{E:lambda-condition}. The
equality $\d r = - I\bigl(\weight^{-1}(\rho)\bigr)^2\,\d\rho$ has to
be exploited here.
\end{proof}

An asymptotic expansion for the function $\lambda _t$, defined via equation
\eqref{E:lambda-t-condition}, is provided by the following lemma in the case
when the Young function $B$ fulfills condition~\eqref{BN}.

\begin{lemma} \label{L:lambda-asymptotics}
Let $B$ be a Young function of the form \eqref{BN} and let $\lambda_t>0$ be
defined by~\eqref{E:lambda-t-condition} for $t>0$. Then the function
$t\mapsto\lambda_t$ is decreasing on $(0,\infty)$, and
\begin{equation}\label{E:lambda-t-asymptotics}
	\lambda_t =
	\begin{cases}
		(2-2\beta) (\log t)^{1-\ib} + \cdots
			&\text{if $\beta\in(0,1)$}
			\\
		\displaystyle \frac{2}{\log\log t} + \cdots
			&\text{if $\beta=1$}
			\\
		\lambda + \cdots
			&\text{if $\beta\in(1,\infty)$}
	\end{cases}
	\quad\text{as $t\to\infty$.}
\end{equation}
for some $\lambda >0$, depending on $\beta$.
\end{lemma}

\begin{proof}
The monotonicity of the function $t\mapsto \lambda_t$ follows from equation~\eqref{E:lambda-t-condition} and the fact that the function $B\circ b^{-1}$ is
increasing.

In order to prove expansion \eqref{E:lambda-t-asymptotics}, we begin by
observing that
\begin{equation*}
	B\bigl(b^{-1}(t)\bigr)
		= \frac{t}{\beta}[b^{-1}(t)]^{1-\beta}
		\quad\text{for sufficiently large $t$}.
\end{equation*}
Consequently, if $\beta\neq 1$, then by Lemma~\ref{L:a-invers} and
equation~\eqref{E:power} with $\sigma=1/\beta$,
\begin{equation} \label{E:B-beta-invers}
	B\bigl(b^{-1}(t)\bigr)
 	= \ib t(\log t)^{\ib-1}
		+ \frac{(1-\beta)^2}{\beta^3} t (\log t)^{\ib-2}\log\log t
		+ \cdots
		\ttoinf.
\end{equation}
On the other hand, if $\beta=1$, then $b=B$ near infinity, whence
$B(b^{-1}(t))=t$ for large $t$.

Owing to expansion \eqref{E:B-beta-invers} and
Lemma~\ref{L:Iwi2-expansion}, for every
$\varepsilon\in(0,1)$ there exists $t_0>0$ such that
\begin{equation} \label{E:B-b-and-I-weight-lb}
	B\bigl(b^{-1}(t)\bigr)
		\ge \frac{t}{\beta}(\log t)^{\ib-1}
	\quad\text{and}\quad
	I\bigl( \weight^{-1}(t) \bigr)^2
		\ge (1-\varepsilon)\frac{1}{2t^2\log t}
\end{equation}
and
\begin{equation} \label{E:B-b-and-I-weight-ub}
	B\bigl(b^{-1}(t)\bigr)\le \frac{1+\varepsilon}{\beta}t(\log t)^{\ib-1}
	\quad\text{and}\quad
	I\bigl( \weight^{-1}(t) \bigr)^2
		\le \frac{1}{2t^2\log t}
\end{equation}
for $t>t_0$.

Assume first that $\beta\in(1,\infty)$. We claim that
$\lim_{t\to\infty}\lambda_{t}>0$.
Assume, by contradiction, that $\lim_{t\to\infty}\lambda_{t}=0$.
Inequalities \eqref{E:B-b-and-I-weight-ub} ensure that
\begin{equation*}
	\int_{0}^{\infty} B\bigl(b^{-1}(\tau)\bigr)
		\,I\bigl( \weight^{-1}(\tau) \bigr)^2\,\d \tau < \infty.
\end{equation*}
Therefore,  by the dominated convergence theorem and   the fact that  $b^{-1}(0)=0$,
\begin{equation*}
	1 = \lim_{t\to\infty}\int_{0}^{t}
			B\bigl(b^{-1}(\lt \tau)\bigr)\,I\bigl( \weight^{-1}(\tau) \bigr)^2\,\d \tau
    \le \int_{0}^{\infty} \lim_{t\to\infty}
			B\bigl(b^{-1}(\lt\tau)\bigr)\,I\bigl( \weight^{-1}(\tau) \bigr)^2\,\d \tau
		= 0.
\end{equation*}
This contradiction  proves our claim, and hence also equation
\eqref{E:lambda-t-asymptotics}.

Assume next that $\beta\in(0,1]$. We claim that $\lim_{t\to\infty}\lambda_{t}=0$.
Suppose, by contradiction, that $\lim_{t\to\infty}\lambda_{t}=\lambda>0$.
Fix $\varepsilon\in(0,1)$, and let $t_0$ be as above. Set
\begin{equation*}
	t_0' = \max\{\lambda,1\}t_0.
\end{equation*}
By~\eqref{E:lambda-t-condition},
Fatou's lemma and inequalities \eqref{E:B-b-and-I-weight-lb},
\begin{align*}
	1 & \ge \int_{0}^{\infty}B\bigl(b^{-1}(\lambda t)\bigr)
						\,I\bigl( \weight^{-1}(t) \bigr)^2\,\d t
			\ge \int_{t_0'}^{\infty}B\bigl(b^{-1}(\lambda t)\bigr)
						\,I\bigl( \weight^{-1}(t) \bigr)^2\,\d t
    	\ge \frac{1-\varepsilon}{2\beta}\lambda
				\int_{t_0'}^{\infty}\frac{\left(\log \lambda t\right)^{\ib-1}}{t\log t}\d t.
\end{align*}
Inasmuch as $1/\beta-1\ge0$, the latter integral diverges,  and we obtain  a
contradiction. Our claim is thus established.

Now, let $\beta\in(0,1)$. Equation ~\eqref{E:lambda-t-asymptotics} will follow
if  we show that
\begin{equation}\label{E:lambda-t-limit}
	\lim_{t\to\infty}\frac{\lambda_t}{(2-2\beta)\left(\log t\right)^{1-\ib}}
		= 1.
\end{equation}
Assume that equation ~\eqref{E:lambda-t-limit} does not hold. Then there exist
$\delta>0$ and an increasing sequence $\{t_k\}$ such that $t_k\to\infty$ and either
\begin{equation}\label{E:lambda-t-upper-contradiction}
	\ltk\ge (1+\delta){(2-2\beta)\left(\log t_k\right)^{1-\ib}}
\end{equation}
or
\begin{equation}\label{E:lambda-t-lower-contradiction}
	\ltk\le (1-\delta){(2-2\beta)\left(\log t_k\right)^{1-\ib}}
\end{equation}
for $k\in\N$. Assume first that~\eqref{E:lambda-t-upper-contradiction} is satisfied. Then
\begin{equation}\label{E:twostar}
	\lim_{k\to\infty}\ltk t_k=\infty
\end{equation}
and
\begin{equation}\label{E:tristar}
	\lim_{k\to\infty}\frac{\log t_k}{\log\iltk}
		= \infty.
\end{equation}
Fix $\varepsilon\in(0,1)$. By combining equation ~\eqref{E:B-beta-invers} and
Lemma~\ref{L:Iwi2-expansion} with the  piece of information that
$\lambda_t\to0$, we conclude that there exists $t_0>0$ such that
\eqref{E:B-b-and-I-weight-lb} and \eqref{E:B-b-and-I-weight-ub} hold  for
$t>t_0$, and
\begin{equation} \label{E:lambda-less-than-1}
	\lambda_{t} < 1
	\quad\text{for $t>t_0$}.
\end{equation}
By~\eqref{E:twostar}, there exists $k_0\in\N$ such that, if $k\ge k_0$, then
$t_k\ge t_0$ and
$t_k\ltk>t_0$. Thereby,
\begin{align} \label{E:estimate-with-big-difference}
	\begin{split}
	1 & = \int_{0}^{t_k}
					B\bigl(b^{-1}(\ltk \tau)\bigr)\, I\bigl(\weight^{-1}(\tau)\bigr)^2\,\d \tau
			\ge \int_{t_{0}/\ltk}^{t_k}
					B\bigl(b^{-1}(\ltk \tau)\bigr)\, I\bigl(\weight^{-1}(\tau)\bigr)^2\,\d \tau
			\\
		& \ge \frac{1-\varepsilon}{2\beta}\ltk
					\int_{t_{0}/\ltk}^{t_k}
						\frac{\left(\log \ltk \tau\right)^{\ib-1}}{\tau\log \tau}\d \tau
		\quad\text{for $k\ge k_0$.}
	\end{split}
\end{align}

The change of variables $\NV=\log \tau/\log\tfrac1{\ltk}$ yields
\begin{align} \label{E:change-of-var}
	\begin{split}
	\int_{{t_{0}}/{\ltk}}^{t_k}
			\frac{\left(\log \ltk \tau\right)^{\ib-1}}{\tau\log \tau}\d \tau
    & = \left(\log\iltk\right)^{\ib-1}
        \int_{{\log\frac{t_{0}}{\ltk}}\big/{\log\iltk}}
						^{{\log t_k}\big/{\log\iltk}}
				\frac{\left(\NV-1\right)^{\ib-1}}{\NV}\d\NV
			\\
    & = \left(\log\iltk\right)^{\ib-1}
      \left[
					\Psi_{\ib-1} \left( \frac{\log t_k}{\log\iltk} \right)
				- \Psi_{\ib-1} \left( \frac{\log\frac{t_{0}}{\ltk}}{\log\iltk} \right)
			\right]
	\end{split}
\end{align}
for $k\ge k_0$, where $\Psi_{\ib-1}$ is the function defined as in~\eqref{E:Psi}.
Observe that
\begin{equation*}
	\lim_{k\to\infty}
		\frac{\log\frac{t_{0}}{\ltk}}{\log\iltk}
		= 1,
\end{equation*}
whence, by~\eqref{E:Psi},
\begin{equation} \label{E:3.13.1}
	\lim_{k\to\infty} \Psi_{\ib-1}
		\left(
			\frac{\log\frac{t_{0}}{\ltk}}{\log\iltk}
		\right)
		= 0.
\end{equation}
On the other hand, it follows from equation~\eqref{E:tristar} and Lemma~\ref{L:Psi} that
\begin{equation} \label{E:3.13.2}
	\Psi_{\ib-1}
		\left( \frac{\log t_k}{\log\iltk} \right)
  	= \frac{\beta}{1-\beta}
				\left( \frac{\log t_k}{\log\iltk} \right)^{\ib-1}
			+ \cdots
	\quad\text{as $k\to\infty$.}
\end{equation}
Coupling \eqref{E:estimate-with-big-difference} with
\eqref{E:lambda-t-upper-contradiction} and \eqref{E:change-of-var}, and
making use of \eqref{E:3.13.1} and \eqref{E:3.13.2} enable one to deduce that
\begin{align*}
	1
		&\ge (1+\delta)(1-\varepsilon)\frac{1-\beta}{\beta}
			\lim_{k\to\infty}
				\left( \frac{\log t_k}{\log\iltk} \right)^{1-\ib}
				\left[
						\Psi_{\ib-1} \left( \frac{\log t_k}{\log\iltk} \right)
					- \Psi_{\ib-1} \left( \frac{\log\frac{t_{0}}{\ltk}}{\log\iltk} \right)
				\right]
		= (1+\delta)(1-\varepsilon).
\end{align*}
A contradiction follows from this chain, provided that $\varepsilon$ is chosen
so small that $(1+\delta)(1-\varepsilon)>1$.

Assume next that equation \eqref{E:lambda-t-lower-contradiction} is in force.
Therefore, fixing $\varepsilon \in (0,1)$, there exists $t_0\ge e$ such that
equation \eqref{E:B-b-and-I-weight-ub} holds for $t>t_0$ and \eqref{E:lambda-less-than-1} is fulfilled.
We claim that there exists $k_0\in\N$ such that
\begin{equation} \label{E:post-1}
	t_k\le \frac{t_{0}}{\ltk}
		\quad\text{for $k\ge k_0$}.
\end{equation}

Suppose, by contradiction, that there exists an increasing
subsequence of $\{t_k\}$, denoted again by $\{t_k\}$, satisfying
$t_k\ltk>t_0$ for $k\in\N$.
Thus,
\begin{align} \label{E:I1+I2}
	\begin{split}
	1	& = \int_{0}^{{t_{0}}/{\ltk}}
					B\bigl(b^{-1}(\ltk \tau)\bigr)\, I\bigl(\weight^{-1}(\tau)\bigr)^2\,\d \tau
				+ \int_{{t_{0}}/{\ltk}}^{t_k}
					B\bigl(b^{-1}(\ltk \tau)\bigr)\, I\bigl(\weight^{-1}(\tau)\bigr)^2\,\d \tau
			\\
		& = J_1(t_k)+ J_2(t_k)
	\quad\text{for $k\in\N$.}
	\end{split}
\end{align}
Since $B$ is a Young function, we have that $B(t)\le tb(t)$ for $t>0$.
Consequently,
\begin{align} \label{E:I1-ub}
	\begin{split}
	J_1(t_k)
	& = \int_{0}^{{t_{0}}/{\ltk}}
			B\bigl(b^{-1}(\ltk \tau)\bigr)\, I\bigl(\weight^{-1}(\tau)\bigr)^2\,\d \tau
		\le \ltk \int_{0}^{{t_{0}}/{\ltk}}
			b^{-1}(\ltk \tau)\,\tau I\bigl(\weight^{-1}(\tau)\bigr)^2\,\d \tau
			\\
	& \le \ltk b^{-1}(t_0)
				\int_{0}^{t_0} \tau I\bigl(\weight^{-1}(\tau)\bigr)^2\,\d \tau
			+ \ltk b^{-1}(t_0)
				\int_{t_0}^{{t_{0}}/{\ltk}} \tau I\bigl(\weight^{-1}(\tau)\bigr)^2\,\d \tau
	\quad\text{for $k\in\N$}.
	\end{split}
\end{align}
The first integral on the rightmost side of \eqref{E:I1-ub} is trivially
convergent. Moreover, since $t_0$ was chosen in such a way that $t_0\ge
e$ and ~\eqref{E:B-b-and-I-weight-ub} holds for every $\tau>t_0$, the second
integral on the rightmost side of \eqref{E:I1-ub} can be estimated as
\begin{equation*}
	\int_{t_0}^{{t_{0}}/{\ltk}}
			\tau I\bigl(\weight^{-1}(\tau)\bigr)^2\,\d \tau
		\le \int_{t_0}^{{t_{0}}/{\ltk}}
			\frac{\d \tau}{2\tau\log \tau}
		\le \frac{1}{2}\log\log\left( \frac{t_0}{\ltk} \right)
	\quad\text{for $k\in\N$.}
\end{equation*}
Consequently, since $\ltk\to0$ as $k\to\infty$,
\begin{equation} \label{E:I1-vanishes}
	\lim_{k\to\infty} J_1(t_k) = 0.
\end{equation}
On making use of equations \eqref{E:B-b-and-I-weight-ub} and
\eqref{E:change-of-var}, we infer that
\begin{equation} \label{E:I2-ub-old}
	J_2(t_k)
		\le \frac{1+\varepsilon}{2\beta}\ltk \left(\log\iltk\right)^{\ib-1}
			\Psi_{\ib-1} \left( \frac{\log t_k}{\log\iltk} \right)
	\quad\text{for $k\in\N$}.
\end{equation}
Observe, that for every $\sigma>0$,
\begin{equation*}
	\Psi_\sigma(t)
		\le \int_{1}^{t} \tau^{\sigma-1}\,\d\tau
		= \frac{1}{\sigma} t^\sigma - \frac{1}{\sigma}
	\quad\text{for $t\in[1,\infty)$.}
\end{equation*}
Hence,
\begin{equation} \label{E:Psi-pointwise-estimate}
	\sigma t^{-\sigma} \Psi_\sigma(t)
		\le 1 - t^{-\sigma}
		\le 1
	\quad\text{for $t\in[1,\infty)$}.
\end{equation}
Form equations \eqref{E:I1+I2}, \eqref{E:I1-vanishes},\eqref{E:I2-ub-old} and
\eqref{E:lambda-t-lower-contradiction}, we deduce that
\begin{equation*}
	1	= \lim_{k\to\infty}\left[ J_1(t_k) + J_2(t_k) \right]
		\le (1+\varepsilon)(1-\delta) \lim_{k\to\infty}
			\frac{1-\beta}{\beta} \left( \frac{\log t_k}{\log\iltk} \right)^{1-\ib}
				\Psi_{\ib-1}\left( \frac{\log t_k}{\log\iltk} \right).
\end{equation*}
This chain, coupled  with estimate \eqref{E:Psi-pointwise-estimate}, yields
$1\le (1+\varepsilon)(1-\delta)$, and hence a contradiction, provided that
$\varepsilon$ is chosen small enough. Thus, inequality \eqref{E:post-1} is established.  As a consequence,
\begin{equation} \label{E:pes-5}
	1 = \int_{0}^{t_k}
				B\bigl(b^{-1}(\ltk \tau)\bigr) I\bigl(\weight^{-1}(\tau)\bigr)^2\,\d \tau
		\le \int_{0}^{{t_{0}}/{\ltk}}
					B\bigl(b^{-1}( \ltk \tau)\bigr) I\bigl(\weight^{-1}(\tau)\bigr)^2\,\d \tau
		= J_1(t_k)
\end{equation}
for $k\ge k_0$. Equation \eqref{E:pes-5}, coupled with \eqref{E:I1-vanishes},
leads to a contradiction again.

The case when $\beta=1$ is analogous and even simpler. The details are omitted,
for brevity.
\end{proof}

With Lemma~\ref{L:lambda-asymptotics} at our disposal, we are able to derive
 an asymptotic expansion for the norm $\normwB$ as $s\to 0^+$.

\begin{lemma}
\label{L:Orlicz-norm}
Let $B$ be a Young function
obeying \eqref{BN} for some $\beta >0$ and $N>0$,
and let $\weight$ be the function defined by \eqref{E:weight-def}.
Then
\begin{equation}\label{july15}
	\normwB
		= \frac{\beta}{2} \left( \log\is \right)^\ib
			+
			\begin{cases}
				- \frac{2+3\beta}{4-4\beta}\left( \log\is \right)^{\ib-1}\log\log\is
				+ \cdots
					& \text{if $\beta\in(0,1)$}
					\\
				-\frac{5}{8}\left(\log\log\is\right)^{2} + \cdots
					& \text{if $\beta=1$}
					\\
				\displaystyle
				c_{\beta, N}
				+ \cdots
					& \text{if $\beta\in(1,\infty)$}
			\end{cases}
		\stoo,
\end{equation}
where $c_{\beta,N}$ is a  constant depending on $\beta$ and $N$.
\end{lemma}

\begin{proof}
Given $s\in(0,\frac{1}{2})$, set $t=\weight(s)$. By Lemma~\ref{L:Orlicz-norm-weight},
\begin{equation*}
	\normwB
		= \int_{0}^{t} b^{-1}(\lambda_t \tau)\,
			\tau I\bigl(\weight^{-1}(\tau)\bigr)^2\,\d \tau
		\quad\text{for $s\in(0,\tfrac12)$},
\end{equation*}
where $\lt>0$ is uniquely defined by~\eqref{E:lambda-t-condition}.

Let $\beta\in(0,1]$. Owing to Lemma~\ref{L:lambda-asymptotics}, the function
$t\mapsto\lt$ is decreasing and $\lt\to0$ as $t\to\infty$. Thus, there
exists $t_0>e$ such that $\lt<1$ for $t>t_0$. Notice that, owing to
equation \eqref{E:lambda-t-asymptotics}, one has that $t>{t_0}/{\lt}$ if
$t$ is sufficiently large. Consequently,
\begin{align} \label{E:division}
	\begin{split}
	\normwB
		& = \int_{0}^{{t_0}/{\lt}}
			b^{-1}(\lt \tau)\,\tau I\bigl(\weight^{-1}(\tau)\bigr)^{2}\,\d \tau
				+ \int_{{t_0}/{\lt}}^{t}
			b^{-1}(\lt \tau)\,\tau I\bigl(\weight^{-1}(\tau)\bigr)^{2}\,\d \tau
				\\
		& = J_1(t) + J_2(t)
 \end{split}
\end{align}
for large $t$, depending on $t_0$.  Let us focus on $J_1$ first. By
Lemma~\ref{L:Iwi2-expansion}, we may assume that $t_0$ is so large that
\begin{equation} \label{E:pes-6}
	tI\bigl(\weight^{-1}(t)\bigr)^2
		\le \frac{1}{2t\log t}
	\quad\text{for $t>t_0$}.
\end{equation}
One has that
\begin{equation*}
	0 \le J_1(t)
    \le b^{-1}(t_0)
				\left(
					\int_{0}^{t_0}
						\tau I\bigl(\weight^{-1}(\tau)\bigr)^{2}\,\d \tau
    		+ \int_{t_0}^{{t_0}/{\lt}}
						\tau I\bigl(\weight^{-1}(\tau)\bigr)^{2}\,\d \tau
				\right)
	\quad\text{for $t>t_0$}.
\end{equation*}
Note that the first integral on the rightmost side of the last equation is
finite and independent of $t$. Let us set
\begin{equation*}
    c(t_0) = \int_{0}^{t_0} \tau I\bigl(\weight^{-1}(\tau)\bigr)^{2}\,\d \tau.
\end{equation*}
 Next, by inequality \eqref{E:pes-6},
\begin{equation*}
	\int_{t_0}^{{t_0}/{\lt}}
		\tau I\bigl(\weight^{-1}(\tau)\bigr)^{2}\,\d \tau
		\le \frac{1}{2} \int_{t_0}^{{t_0}/{\lt}} \frac{\d \tau}{\tau\log \tau}
		\le \frac{1}{2}\log\log\tolt
	\quad\text{for $t>t_0$}.
\end{equation*}
Altogether,
\begin{equation}\label{E:upper-bound-I1}
	0 \le J_1(t)
		\le \frac{b^{-1}(t_0)}{2}\log\log\ilt
			+ \cdots
	\quad\text{as $t\to\infty$}.
\end{equation}

Let us now consider $J_2$. First, assume that $\beta\in(0,1)$.  Fix
$\varepsilon>0$. By Lemmas~\ref{L:a-invers} and \ref{L:Iwi2-expansion}, we may
assume that $t_0$ is so large that
\begin{equation*}
	b^{-1}(t)
		\le (\log t)^\ib
		+ (1+\varepsilon)\frac{1-\beta}{\beta^2}
			(\log t)^{\ib-1}\log\log t
\end{equation*}
and
\begin{equation*}
	b^{-1}(t)
		\ge (\log t)^\ib
		+ (1-\varepsilon)\frac{1-\beta}{\beta^2}
			(\log t)^{\ib-1}\log\log t
\end{equation*}
and, simultaneously,
\begin{equation} \label{E:Iwi2-ub}
	tI\bigl( \weight^{-1}(t) \bigr)^2
		\le \frac{1}{2t\log t}
			- (1-\varepsilon)\frac{5\log\log t}{4t(\log t)^2}
\end{equation}
and
\begin{equation} \label{E:Iwi2-lb}
	tI\bigl( \weight^{-1}(t) \bigr)^2
		\ge \frac{1}{2t\log t}
			- (1+\varepsilon)\frac{5\log\log t}{4t(\log t)^2}
\end{equation}
for $t>t_0$.
Thus, on setting
\begin{align*}
	J_{21}(t)
		& = \int_{t_0/\lt}^{t}
					\frac{(\log \lt \tau)^\ib}{\tau\log \tau}\d \tau,
  			& J_{22}(t)
					& = \int_{t_0/\lt}^{t}
    						\frac{(\log\lt \tau)^{\ib-1}\log\log(\lt \tau)}{\tau\log \tau}\d \tau,
       \\
   J_{23}(t)
	 	& = \int_{t_0/\lt}^{t}
					\frac{(\log\lt \tau)^{\ib}\log\log \tau}{\tau(\log \tau)^2}\d \tau,
  			& J_{24}(t)
					& = \int_{t_0/\lt}^{t}
								\frac{(\log\lt \tau)^{\ib-1}\log\log(\lt \tau)\log\log \tau}
									{\tau(\log \tau)^2}\d\tau
\end{align*}
for large $t$, we deduce that
\begin{equation}\label{E:upper-bound-for-I2}
	J_{2}(t)
		\le \frac{1}{2}J_{21}(t)
			+ (1+\varepsilon) \frac{1-\beta}{2\beta^2} J_{22}(t)
			+ (\varepsilon-1)\frac{5}{4} J_{23}(t)
			+ (\varepsilon^{2}-1)\frac{5(1-\beta)}{4\beta^2} J_{24}(t)
\end{equation}
and
\begin{equation}\label{E:lower-bound-for-I2}
	J_{2}(t)
		\ge \frac{1}{2}J_{21}(t)
			+ (1-\varepsilon) \frac{1-\beta}{2\beta^2} J_{22}(t)
			- (1+\varepsilon)\frac{5}{4} J_{23}(t)
			+ (\varepsilon^{2}-1)\frac{5(1-\beta)}{4\beta^2} J_{24}(t)
\end{equation}
for large $t$.

As a next step, we evaluate the asymptotic behaviour of the terms
$J_{21}(t)$--$J_{24}(t)$. Let us begin with $J_{21}(t)$. Via the change of
variables $\log \tau=\NV\log1/\lt$, one obtains that
\begin{equation} \label{E:I21-cov}
	J_{21}(t)
		= \left(\log\tilt\right)^\ib
				\int_{{\log\tolt}\big/{\log\ilt}}^{{\log t}\big/{\log\ilt}}
					\frac{(\NV-1)^{\beta}}{\NV}\d \NV
		= \left(\log\tilt\right)^\ib
			\left[
				\Psi_{\ib}\left(\frac{\log t}{\log\ilt}\right)
				- \Psi_{\ib}\left(\frac{\log\tolt}{\log\ilt}\right)
			\right]
\end{equation}
for large $t$,
where $\Psi_{\ib}$ is defined as in \eqref{E:Psi}.
Lemma~\ref{L:lambda-asymptotics} implies that
\begin{equation} \label{E:lambda-tau-asymptotics}
	\lt = (2-2\beta)(\log t)^{1-\ib} + \cdots
		\quad\text{as $t\to\infty$},
\end{equation}
whence
\begin{equation} \label{E:log-lambda-tau}
	\log\ilt = \left( \ib - 1 \right)\log\log t + \cdots
		\quad\text{as $t\to\infty$}.
\end{equation}
Consequently,
\begin{equation} \label{E:lims-tau}
	\lim_{t\to\infty} \frac{\log t}{\log\ilt} = \infty
		\quad\text{and}\quad
	\lim_{t\to\infty} \frac{\log\tolt}{\log\ilt} = 1.
\end{equation}
Lemma~\ref{L:Psi} tells us that
\begin{equation*}
	\Psi_{\ib}(t)
		= \beta t^\ib - \frac{1}{1-\beta} t^{\ib-1}
			+ \cdots
	\ttoinf
	\quad\text{and}\quad
	\Psi_{\ib}(t) \to 0
	\quad\text{as $t\to 1^+$}.
\end{equation*}
Therefore, equation \eqref{E:I21-cov} reads
\begin{align} \label{E:I21-asy}
	\begin{split}
	J_{21}(t)
		& = \beta (\log t)^\ib
				- \frac{1}{1-\beta}\log\ilt(\log t)^{\ib-1}
				+ \cdots
				\\
		& = \beta (\log t)^\ib
				- \ib(\log t)^{\ib-1}\log\log t
				+ \cdots
	\quad\text{as $t\to\infty$}.
	\end{split}
\end{align}

As for the term $J_{22}$, the same change of variables as above yields
\begin{align*}
	J_{22}(t)
		& = \left(\log\tilt\right)^{\ib-1}
					\int_{{\log\tolt}\big/{\log\ilt}}^{{\log t}\big/{\log\ilt}}
						\frac{(\NV-1)^{\ib-1}}{\NV}\log\left((\NV-1)\log\tilt\right)\d \NV
				\\
		& = \left(\log\tilt\right)^{\ib-1}
					\int_{{\log\tolt}\big/{\log\ilt}}^{{\log t}\big/{\log\ilt}}
						\frac{(\NV-1)^{\ib-1}}{\NV}\log(\NV-1)\,\d \NV
						\\
		& \quad
				+ \left(\log\tilt\right)^{\ib-1} \log\log\tilt
					\int_{{\log\tolt}\big/{\log\ilt}}^{{\log t}\big/{\log\ilt}}
						\frac{(\NV-1)^{\ib-1}}{\NV}\,\d \NV
\end{align*}
for large $t$.  Hence
\begin{align*}
	J_{22}(t)
    & = \left(\log\tilt\right)^{\ib-1}
					\left[
						\Upsilon_{\ib-1}\left(\frac{\log t}{\log\ilt}\right)
						- \Upsilon_{\ib-1}\left(\frac{\log\tolt}{\log\ilt}\right)
					\right]
       \\
    &\quad
			+ \left(\log\tilt\right)^{\ib-1} \log\log\tilt
					\left[
						\Psi_{\ib-1}\left(\frac{\log t}{\log\ilt}\right)
						- \Psi_{\ib-1}\left(\frac{\log\tolt}{\log\ilt}\right)
					\right]
	\quad\text{for $t>t_0$},
\end{align*}
where
$\Upsilon_\sigma\colon(1,\infty)\to\R$ is the function defined for $\sigma>0$
by
\begin{equation*}
	\Upsilon_\sigma(t)
		= \int_{1}^{t} \frac{(\tau-1)^\sigma}{\tau}\log(\tau-1)\,\d \tau
		\quad\text{for $t>1$}.
\end{equation*}
Observe that
\begin{equation*}
	\Upsilon_{\ib-1}(t)
		= \frac{\beta}{1-\beta} t^{\ib-1} \log t
			+ \cdots
		\ttoinf.
\end{equation*}
Furthermore, by Lemma~\ref{L:Psi},
\begin{equation*}
	\Psi_{\ib-1}(t)
		= \frac{\beta}{1-\beta} t^{\ib-1}
			+ \cdots
		\ttoinf.
\end{equation*}
Altogether, since $\Upsilon_{\ib-1}(t)\to 0$ and $\Psi_{\ib-1}(t)\to 0$
as $t\to 1^+$, we conclude, via equation~\eqref{E:lims-tau}, that
\begin{equation*}
	J_{22}(t)
		= \frac{\beta}{1-\beta} (\log t)^{\ib-1}\log\left(\frac{\log t}{\log \lt}\right)
			+ \frac{\beta}{1-\beta} (\log t)^{\ib-1}\log\log\ilt
			+ \cdots
	\quad\text{as $t\to\infty$}.
\end{equation*}
Hence,
\begin{equation} \label{E:I22-asy}
	J_{22}(t)
		= \frac{\beta}{1-\beta}(\log t)^{\ib-1} \log\log t + \cdots
	\quad\text{as $t\to\infty$}.
\end{equation}

The term $J_{23}(t)$ can be estimated similarly. First, changing variables as
above, we get
\begin{align*}
	J_{23}(t)
    & = \left(\log\tilt\right)^{\ib-1}
					\int_{{\log\tolt}\big/{\log\lt}}^{{\log t}\big/{\log\lt}}
						\frac{(\NV-1)^{\ib}}{\NV^{2}}\left[\log \NV + \log\log\tilt\right]\,\d \NV
				\\
    & = \left(\log\tilt\right)^{\ib-1}
					\left[
						\bar\Upsilon_{\ib-1}\left(\frac{\log t}{\log\ilt}\right)
        		- \bar\Upsilon_{\ib-1}\left(\frac{\log\tolt}{\log\ilt}\right)
					\right]
				\\
    & \quad
			+ \left(\log\tilt\right)^{\ib-1} \log\log\ilt
					\left[
						\bar\Psi_{\ib-1}\left(\frac{\log t}{\log\ilt}\right)
        		- \bar\Psi_{\ib-1}\left(\frac{\log\tolt}{\log\ilt}\right)
					\right]
\end{align*}
for large $t$, where the functions
$\bar\Upsilon_\sigma\colon(1,\infty)\to(0,\infty)$ and
$\bar\Psi_\sigma\colon(1,\infty)\to(0,\infty)$ are defined, for $\sigma>0$, by
\begin{equation*}
	\bar\Upsilon_\sigma(t)
		= \int_{1}^{t} \frac{(\tau-1)^{\sigma+1}}{\tau^2}\log \tau\,\d \tau
	\quad\text{and}\quad
	\bar\Psi_\sigma(t)
		= \int_{1}^{t} \frac{(\tau-1)^{\sigma+1}}{\tau^2}\,\d \tau
	\quad\text{for $t>1$}.
\end{equation*}
It is easily verified that
\begin{equation*}
	\bar\Upsilon_{\ib-1}(t)
		= \frac{\beta}{1-\beta} t^{\ib-1}\log t + \cdots
	\quad\text{and}\quad
	\bar\Psi_{\ib-1}(t)
		= \frac{\beta}{1-\beta} t^{\ib-1} + \cdots
	\ttoinf.
\end{equation*}
Also $\bar\Upsilon_{\ib-1}(t)\to 0$ and $\bar\Psi_{\ib-1}(t)\to 0$
as $t\to 1^+$. Therefore, owing to \eqref{E:lims-tau}, we conclude that
\begin{equation} \label{E:I23-asy}
	J_{23}(t)
		= \frac{\beta}{1-\beta} (\log t)^{\ib-1} \log\log t + \cdots
	\quad\text{as $t\to\infty$}.
\end{equation}

Finally, we claim that $J_{24}(t)$ is of lower order than $J_{21}(t)$,
$J_{22}$ and $J_{23}(t)$ as $t\to\infty$. Indeed, since $\lt<1$, one has that
\begin{equation*}
	0 \le J_{24}(t)
		\le \int_{e}^{t} \frac{(\log \tau)^{\ib-1}(\log\log \tau)^2}{\tau(\log \tau)^2}\,\d \tau
		= Y_{\ib-2}(\log t)
\end{equation*}
for large $t$,
where $Y_\sigma\colon(1,\infty)\to(0,\infty)$ is defined by
\begin{equation*}
	Y_\sigma(t) = \int_{1}^{t} \tau^{\sigma-1}(\log \tau)^2\,\d \tau
		\quad\text{for $t>1$}.
\end{equation*}
Observe that
\begin{equation*}
	Y_\sigma(t) =
	\begin{cases}
		\frac{1}{\sigma} t^\sigma(\log t)^2 + \cdots
			& \text{if $\sigma>0$}
			\\
		\frac{1}{3}(\log t)^3 + \cdots
			& \text{if $\sigma=0$}
			\\
		c + \cdots
			& \text{if $\sigma<0$}
	\end{cases}
	\ttoinf
\end{equation*}
for a suitable constant $c\in\R$ depending on $\sigma$. Therefore,
\begin{equation} \label{E:I24-est}
	0 \le J_{24}(t)
		\le
		\begin{cases}
			\frac{\beta}{1-2\beta} (\log t)^{\ib-2}(\log\log t)^2 + \cdots
				& \text{if $\beta\in(0,\tfrac12)$}
				\\
			\frac{1}{3}(\log\log t)^3 + \cdots
				& \text{if $\beta=\tfrac12$}
				\\
			c + \cdots
				& \text{if $\beta\in(\tfrac12,1)$}
		\end{cases}
	\quad\text{as $t\to\infty$.}
\end{equation}
This proves our claim.

On making use of equations ~\eqref{E:I21-asy}--\eqref{E:I24-est}, we infer from
\eqref{E:upper-bound-for-I2} that
\begin{equation*}
	J_{2}(t)
		\le \frac{\beta}{2}(\log t)^{\ib}
			+ \left(
					- \frac{1}{2\beta}
					+ (1+\varepsilon)\frac{1}{2\beta}
					+ (\varepsilon-1)\frac{5\beta}{4-4\beta}
				\right)(\log t)^{\ib-1} \log\log t
			+ \cdots
	\quad\text{as $t\to\infty$,}
\end{equation*}
and from \eqref{E:lower-bound-for-I2} that
\begin{equation*}
	J_{2}(t)
		\ge \frac{\beta}{2}(\log t)^{\ib}
			+ \left(
					- \frac{1}{2\beta}
					+ (1-\varepsilon)\frac{1}{2\beta}
					- (\varepsilon+1)\frac{5\beta}{4-4\beta}
				\right)(\log t)^{\ib-1} \log\log t
			+ \cdots
	\quad\text{as $t\to\infty$.}
\end{equation*}
Hence, by the arbitrariness of $\varepsilon$,
\begin{equation*}
  J_{2}(t)
		= \frac{\beta}{2 }(\log t)^\ib
			- \frac{5\beta}{4-4\beta} (\log t)^{\ib-1} \log\log t
			+ \cdots
	\quad\text{as $t\to\infty$}.
\end{equation*}
By~\eqref{E:upper-bound-I1}, $J_1(t)$ is of lower order than $J_2(t)$ as
$t\to\infty$. Hence, we conclude via \eqref{E:division} that
\begin{equation} \label{E:norm-with-tau}
	\normwB
		= \frac{\beta}{2 }(\log t)^\ib
			- \frac{5\beta}{4-4\beta} (\log t)^{\ib-1} \log\log t
			+ \cdots
	\stoo,
\end{equation}
where $t=\weight(s)$. Thanks to Lemma~\ref{L:weight-expansion} and
\eqref{E:log-of-expansion},
\begin{equation}\label{E:log-tau}
	\log t
		= \log\weight(s)
		= \log\tis
			- \log\log\tis
			+ \cdots
	\stoo.
\end{equation}
Hence, owing to ~\eqref{E:power}, we obtain that
\begin{equation*}
	\normwB
		= \frac{\beta}{2} (\log\tis)^\ib
			- \frac{1}{2} (\log\tis)^{\ib-1} \log\log\tis
			- \frac{5\beta}{4-4\beta}(\log\tis)^{\ib-1}\log\log\tis
			+ \cdots
	\stoo
\end{equation*}
and \eqref{july15} follows.

Assume next that $\beta=1$. Fix $\varepsilon>0$. There exists $t_0>e$ such
that, if $t>t_0$, then $\lt<1$, equations \eqref{E:Iwi2-ub} and
\eqref{E:Iwi2-lb} hold, and moreover
\begin{equation*}
	b^{-1}(t) = \log t - \log N.
\end{equation*}
If we now define
\begin{align*}
	\overline J_{21}(t)
		& = \int_{t_0/\lt}^{t}
					\frac{\log \lt \tau}{\tau\log \tau}\d \tau,
  	& \overline J_{22}(t)
		& = \int_{t_0/\lt}^{t} \frac{\d \tau}{\tau\log \tau},
       \\
  \overline J_{23}(t)
	 	& = \int_{t_0/\lt}^{t}
					\frac{\log(\lt \tau)\log\log \tau}{\tau(\log \tau)^2}\d \tau,
  	& \overline J_{24}(t)
		& = \int_{t_0/\lt}^{t}
					\frac{\log\log \tau}{\tau(\log \tau)^2}\d \tau
\end{align*}
for large $t$,
then
\begin{equation}\label{E:I2-ub}
	J_2(t)
		\le \frac{1}{2} \overline J_{21}(t)
			- \frac{1}{2} \log N \overline J_{22}(t)
			+ (\varepsilon-1) \frac{5}{4} \overline J_{23}(t)
			+ (1-\varepsilon) \frac{5}{4} \log N \overline J_{24}(t)
\end{equation}
and
\begin{equation}\label{E:I2-lb}
	J_2(t)
		\ge \frac{1}{2} \overline J_{21}(t)
			- \frac{1}{2} \log N \overline J_{22}(t)
			- (\varepsilon+1) \frac{5}{4} \overline J_{23}(t)
			+ (1+\varepsilon) \frac{5}{4} \log N \overline J_{24}(t)
\end{equation}
for large $t$.
Let us evaluate the asymptotic behaviour of the terms $\overline
J_{21}(t)$--$\overline J_{24}(t)$ as $t\to\infty$. First observe that
\begin{align*}
		\overline J_{21}(t)
		= \int_{t_0/\lt}^{t} \frac{\d \tau}{\tau}
				- \log\tilt \int_{t_0/\lt}^{t} \frac{\d \tau}{\tau\log \tau}
    = \log t - \log\tolt
				- \log\ilt\left[ \log\log t- \log\log\tolt \right]
\end{align*}
for large $t$.
From Lemma~\ref{L:lambda-asymptotics} we deduce that
\begin{equation*}
	\log\tilt = \log\log\log t + \cdots
		\quad\text{as $t\to\infty$}.
\end{equation*}
Consequently,
\begin{equation} \label{E:I21-asy-2}
		\overline J_{21}(t)
		= \log t
			- \log\log t \log\log\log t
			+ \cdots
	\quad\text{as $ t\to\infty$}.
\end{equation}
Next,
\begin{equation*}
	\overline	J_{22}(t) = \log\log t + \cdots
		\quad\text{as $t\to\infty$.}
\end{equation*}
On the other hand, $\overline J_{24}(t)\to 0$ as $t\to\infty$ since
\begin{equation*}
	0\le 	\overline J_{24}(t)
		\le \int_{t_0/\lt}^{\infty} \frac{\log\log \tau}{\tau(\log \tau)^2}\d \tau
\end{equation*}
for large $t$, and $\lt\to 0$ as $t\to\infty$.
It remains to deal with the  term $\overline J_{23}(t)$. We have that
\begin{align*}
	\overline	J_{23}(t)
		& = \int_{t_0/\lt}^{t} \frac{\log\log \tau}{\tau\log \tau}\,\d \tau
				- \log\ilt \int_{t_0/\lt}^{t} \frac{\log\log \tau}{\tau(\log \tau)^2}\,\d \tau
				\\
		& = \frac{1}{2}(\log\log t)^2 - \frac{1}{2}\left( \log\log\tolt \right)^2
				- \log\ilt J_{24}(t)
			= \frac{1}{2}(\log\log t)^2 + \cdots
	\quad\text{as $t\to\infty$}.
\end{align*}
Altogether, by equation \eqref{E:I2-ub},
\begin{equation*}
	J_{2}(t)
		\le \frac{1}{2}\log t
			+(\varepsilon -1) \frac{5}{8}(\log\log t)^{2}
			+ \cdots
	\quad\text{as $t\to\infty$}
\end{equation*}
and, by \eqref{E:I2-lb},
\begin{equation*}
	J_{2}(t)
		\ge \frac{1}{2}\log t
			- (1+\varepsilon) \frac{5}{8}(\log\log t)^{2}
			+ \cdots
	\quad\text{as $t\to\infty$}.
\end{equation*}
Therefore, by the arbitrariness of $\varepsilon$,
\begin{equation*}
	J_{2}(t)
		= \frac{1}{2}\log t
			- \frac{5}{8}(\log\log t)^{2}
			+ \cdots
	\quad\text{as $t\to\infty$}.
\end{equation*}
The term $J_1(t)$ is  of lower order than $J_2(t)$ as $t\to\infty$ also in this
case. Therefore, one infers from
 \eqref{E:division} that
\begin{equation*}
	\normwB
		= \frac{1}{2} \log t
			- \frac{5}{8} (\log\log t)^2
			+ \cdots
	\stoo,
\end{equation*}
where $t=\weight(s)$. Equation~\eqref{july15} hence follows via~\eqref{E:log-tau}.

Finally, let $\beta\in(1,\infty)$. By Lemma~\ref{L:lambda-asymptotics}, the
function $t\mapsto\lt$ is decreasing and $\lt\to\lambda$ for some $\lambda>0$.
Thus, given $\varepsilon>0$, there exists $t_0>0$ such that
\begin{equation*}
	\lambda \le \lt < \lambda + \varepsilon
		\quad\text{for $t>t_0$}.
\end{equation*}
As a consequence,
\begin{equation*}
	\int_{0}^{t} b^{-1}(\lambda \tau)\,
		\tau I\bigl(\weight^{-1}(\tau)\bigr)^2\,\d \tau
	\le \normwB
	\le \int_{0}^{t} b^{-1}\bigl((\lambda+\varepsilon)\tau \bigr)\,
			\tau I\bigl(\weight^{-1}(\tau)\bigr)^2\,\d \tau
\end{equation*}
for $t>t_0$.
Observe that
\begin{align*}
	c_\varepsilon
		& = \lim_{t\to\infty}
				\left[
					\int_{0}^{t} b^{-1}\bigl((\lambda+\varepsilon)\tau\bigr)\,
						\tau I\bigl(\weight^{-1}(\tau)\bigr)^2\,\d \tau
					- \frac{\beta}{2}(\log t)^\ib
				\right]
			\\
		& = \int_{0}^{1} b^{-1}\bigl((\lambda+\varepsilon)\tau\bigr)\,
					\tau I\bigl(\weight^{-1}(\tau)\bigr)^2\,\d \tau
				+ \lim_{t\to\infty}
				\left[
					\int_{1}^{t} b^{-1}\bigl((\lambda+\varepsilon)\tau\bigr)\,
						\tau I\bigl(\weight^{-1}(\tau)\bigr)^2\,\d \tau
					- \int_{1}^{t} \frac{1}{2\tau}(\log \tau)^{\ib-1}\,\d \tau
				\right]
			\\
		& = \int_{0}^{1} b^{-1}\bigl((\lambda+\varepsilon)\tau\bigr)\,
					\tau I\bigl(\weight^{-1}(\tau)\bigr)^2\,\d \tau
				+ \int_{1}^{\infty}
					\left[
						b^{-1}\bigl((\lambda+\varepsilon)\tau\bigr)\, \tau I\bigl(\weight^{-1}(\tau)\bigr)^2
						- \frac{1}{2\tau}(\log \tau)^{\ib-1}
					\right]\d \tau,
\end{align*}
where the last integral converges thanks to
Lemma~\ref{L:asymptotic-of-product}. Hence, by the dominated convergence
theorem, one can deduce that
\begin{equation} \label{E:norm-for-beta-large-tau}
	\normwB = \frac{\beta}{2}(\log t)^\ib + c + \cdots
		\quad\text{as $t\to\infty$},
\end{equation}
where $t=\weight(s)$ and
\begin{equation*}
	c = \lim_{\varepsilon\to 0^+} c_\varepsilon
		= \int_{0}^{1} b^{-1}(\lambda \tau)\,
				\tau I\bigl(\weight^{-1}(\tau)\bigr)^2\,\d \tau
			+ \int_{1}^{\infty}
				\left[
					b^{-1}(\lambda \tau)\, \tau I\bigl(\weight^{-1}(\tau)\bigr)^2
					- \frac{1}{2\tau}(\log \tau)^{\ib-1}
				\right]\d \tau.
\end{equation*}
Equation \eqref{july15} follows from  \eqref{E:norm-for-beta-large-tau} and the
expansion \eqref{E:log-tau}.
\end{proof}

Given a Young function $B$ and a constant $C\in\R$,
we define the function $\GG_{B,C}\colon(0,\tfrac12]\to\R$ by
\begin{equation} \label{E:G-def}
	\GG_{B,C}(s) = \normwB + \weight(s)sB^{-1}\left(\tfrac{1}{s}\right) + C
		\quad\text{for $s\in\left(0,\tfrac{1}{2}\right]$}.
\end{equation}

Combining  Lemmas~\ref{L:second-term-expansion} and \ref{L:Orlicz-norm} with
equation \eqref{E:power} yields the following asymptotic expansion for the
function $\thetab\,\GG_{B,C}$.

\begin{corollary} \label{C:G-expansion}
Let $\beta>0$ and $C\in\R$. Assume that $B$ is a Young function obeying
\eqref{BN}, and let $\GG_{B,C}$ be the function given by \eqref{E:G-def}. Then
\begin{equation} \label{E:G-expansion}
	\bigl[ \thetab\,\GG_{B,C}(s) \bigr]^\beta
		= \log\is +
		\begin{cases}
			- \frac{2+3\beta}{2-2\beta}\log\log\is + \cdots
				& \text{if $\beta\in(0,1)$}
				\\
			- \frac{5}{4}(\log\log\is)^2 + \cdots
				& \text{if $\beta=1$}
				\\
			2(C+c_{\beta,N})(\log\is)^{1-\ib} + \cdots
				& \text{if $\beta\in(1,\infty)$}
		\end{cases}
	\stoo.
\end{equation}
Here, $c_{\beta,N}$ denotes the constant appearing in equation \eqref{july15}.
\end{corollary}

\section{Basic estimates}\label{basic}

Inequalities \eqref{E:pointwise-1}--\eqref{E:pointwise-2} are crucial for the
bounds exhibited in the sequel of this section.

\begin{proposition} \label{P:med-mv}
Let $n\in\N$ and $u\in\Wlgn L^2\RG$. Then
\begin{equation} \label{E:med-mv}
	|\med (u) - \mv (u)|
		\le \int_{0}^{\frac12} (\lgn u)^*(r) \Lambda(r)\,\d r,
\end{equation}
where $\Lambda$ is defined in \eqref{E:Lambda-def}.
\end{proposition}

\begin{proof}
Integrating both sides of inequality~\eqref{E:pointwise-1} over the interval
$(0,\tfrac12)$ and making use of Fubini's theorem enable us to infer that
\begin{align} \label{july1}
	\begin{split}
	0 & \le \int_{0}^{\frac12} \left[ u^\circ(s) - u^\circ(\tfrac12) \right]\d s
			\le \int_{0}^{\frac12} \weight(s) \int_{0}^{s} (\lgn u)_+^*(r)\,\d r\d s
				+ \int_{0}^{\frac12} \int_{s}^{\frac12} (\lgn u)_+^*(r)\weight(r)\,\d r \d s
				\\
		& = \int_{0}^{\frac12} (\lgn u)_+^*(r) \int_{r}^{\frac12} \weight(s)\,\d s\d r
				+ \int_{0}^{\frac12} (\lgn u)_+^*(r) r\weight(r)\,\d r
			= \int_{0}^{\frac12} (\lgn u)_+^*(r) \Lambda(r)\,\d r.
	\end{split}
\end{align}
Analogously, owing to \eqref{E:pointwise-2},
\begin{equation}\label{july2}
	0 \le \int_{0}^{\frac12} \left[ u^\circ(\tfrac12) - u^\circ(1-s) \right]\d s
		\le \int_{0}^{\frac12} (\lgn u)_-^*(r) \Lambda(r)\,\d r.
\end{equation}
By equations \eqref{E:equiintegrability-with-signed} and \eqref{median},
\begin{align*}
	\mv(u) - \med(u)
		& = \int_{\rn} u\,\dgn - u^\circ(\tfrac12)
			= \int_{0}^{1} \left[ u^\circ(s) - u^\circ(\tfrac12) \right]\d s
			\\
		& = \int_{0}^{\frac12} \left[ u^\circ(s) - u^\circ(\tfrac12) \right]\d s
				- \int_{0}^{\frac12} \left[ u^\circ(\tfrac12) - u^\circ(1-s) \right]\d s.
\end{align*}
Therefore,
coupling inequalities \eqref{july1} and \eqref{july2} yields
\begin{equation*}
	- \int_{0}^{\frac12} (\lgn u)_-^*(r) \Lambda(r)\d r
		\le \mv(u) - \med(u)
		\le \int_{0}^{\frac12} (\lgn u)_+^*(r) \Lambda(r)\d r.
\end{equation*}
Hence
\begin{equation*}
	|\mv(u) - \med(u)|
		\le\max
			\left\{
				\int_{0}^{\frac12} (\lgn u)_-^*(r) \Lambda(r)\d r,
				\int_{0}^{\frac12} (\lgn u)_+^*(r) \Lambda(r)\d r
			\right\},
\end{equation*}
and \eqref{E:med-mv} follows.

\end{proof}

\begin{lemma} \label{L:norm-estimates}
Let $n\in\N$, let $\Lambda$ be the function defined by \eqref{E:Lambda-def}, and let
$B$ be a Young function such that $\onorm*{\Lambda}_{L^{\tilde B}(0,\frac12)}<\infty$.
Assume that $u\in \Wlgn L^B\RG$ and fulfills condition \eqref{E:mu}. Let
$\GG_{B,C}$ be the function defined by~\eqref{E:G-def}, where
\begin{equation} \label{E:C-mu}
	C = 0
	\quad\text{if $\med(u)=0$}
	\quad\text{and}\quad
	C = \onorm*{\Lambda}_{L^{\tilde B}(0,\frac12)}
	\quad\text{if $\mv(u)=0$}.
\end{equation}
Then
\begin{equation} \label{E:norm-estimate-1}
	|u^\circ(s)|
		\le \GG_{B,C}(s)  \|\lgn u\|_{L^B\RG}
		\quad\text{for $s\in\left(0,\tfrac{1}{2}\right]$}
\end{equation}
and
\begin{equation} \label{E:norm-estimate-2}
	|u^\circ(1-s)|
		\le \GG_{B,C}(s)  \|\lgn u\|_{L^B\RG}
		\quad\text{for $s\in\left(0,\tfrac{1}{2}\right]$}.
\end{equation}
\end{lemma}

\begin{proof}
Assume first that $\mv(u)=0$. By inequality \eqref{E:pointwise-1}, we have
\begin{equation*}
	0 \le u^\circ(s) - u^\circ(\tfrac12)
		\le \weight(s) \int_{0}^{s} \left( \lgn u \right)^*(r)\,\d r
			+ \int_{s}^{\frac{1}{2}} \left( \lgn u \right)^*(r)\,\weight(r)\,\d r
		\quad\text{for $s\in\left(0,\tfrac{1}{2}\right]$}.
\end{equation*}
Hence, via H\"older's inequality in the form \eqref{E:Orl-Holder} and equation
\eqref{E:Orl-char}, we deduce that
\begin{align} \label{E:us-uhalf}
	\begin{split}
	0 \le u^\circ(s) - u^\circ(\tfrac12)
		& \le \weight(s) \|(\lgn u)^*\|_{L^B(0,1)} \onorm{\chi_{(0,s)}}_{L^{\tilde B}(0,1)}
			+ \|(\lgn u)^*\|_{L^B(0,1)} \onorm*{\weight\chi_{\left(s,\frac12\right)}}_{L^{\tilde B}(0,1)}
			\\
		& = \left(
					\weight(s)sB^{-1}\left(\tfrac{1}{s}\right) + \normwB
				\right)
				\|\lgn u\|_{L^B\RG}
			\quad\text{for $s\in\left(0,\tfrac{1}{2}\right]$}.
	\end{split}
\end{align}
Proposition~\ref{P:med-mv} and  inequality \eqref{E:Orl-Holder} again tell us that
\begin{equation} \label{E:med-ub}
	|u^\circ(\tfrac12)|
		= |\med(u)|
		\le \int_{0}^{\frac12} (\lgn u)^*(r) \Lambda(r)\,\d r
		\le \|\lgn u\|_{L^B\RG} \onorm{\Lambda}_{L^{\tilde B}(0,\frac12)}.
\end{equation}
Under the current assumption that $\mv(u)=0$, inequality \eqref{E:norm-estimate-1} follows
from \eqref{E:us-uhalf} and~\eqref{E:med-ub}.  On the other hand,
under the
assumption that $\med(u)=0$, inequality \eqref{E:norm-estimate-1} is a
straightforward consequence of \eqref{E:us-uhalf}.

The proof of \eqref{E:norm-estimate-2} is analogous: it just rests upon
\eqref{E:pointwise-2} instead of~\eqref{E:pointwise-1}.
\end{proof}

\begin{corollary} \label{C:sup-estimate-expb}
Let $n\in\N$ and $\beta>0$, and  let $B$ be a Young function as in \eqref{BN}.
Assume that $\EF\colon[0,\infty)\to[0,\infty)$ is a non-decreasing
function  and that $\theta >0$.  Then
\begin{equation} \label{E:sup-estimate-expb}
	\sup_u\, \int_{\rn} \expb(\theta|u|)\,\EF(|u|)\,\dgn
		\le 2 \int_{0}^{\frac{1}{2}}
			\expb\bigl(\theta\,\GG_{B,C}(s)\bigr)\,\EF\bigl(\GG_{B,C}(s)\bigr)\,\d s,
\end{equation}
where the supremum is extended over all functions $u$ satisfying
\eqref{E:lgn-luxemburg} and \eqref{E:mu}. Here, $\GG_{B,C}$ denotes the
function defined by \eqref{E:G-def}, with $C$ given according to the
alternative ~\eqref{E:C-mu}, depending on whether $\m=\med$ or
$\m=\mv$ in \eqref{E:mu}.
\end{corollary}

\begin{proof}
Since the function $\widetilde B(t)$ is equivalent to $t(\log t)^{\frac1\beta}$
near infinity, one has that $\onorm*{\Lambda}_{L^{\tilde B}(0,\frac12)}<\infty$,
and hence the constant $C$ in \eqref{E:C-mu} is well defined when
$\m=\mv$. Given any function $u$ as in the statement, we have
that
\begin{align*}
	\int_{\rn} & \expb(\theta|u|)\,\EF(|u|)\,\dgn
        \\
		&	= \int_{0}^{\frac{1}{2}}
				\expb\bigl( \theta |u^\circ(s)| \bigr)
					\EF\bigl(|u^\circ(s)|\bigr)\,\d s
			+ \int_{0}^{\frac{1}{2}}
				\expb\bigl( \theta |u^\circ(1-s)| \bigr)
					\EF\bigl(|u^\circ(1-s)|\bigr)\,\d s.
\end{align*}
The conclusion thus follows by Lemma~\ref{L:norm-estimates}.
\end{proof}

\begin{corollary} \label{C:sup-estimate-expexp}
Let $n\in\N$, let $\Lambda$ be the function defined by \eqref{E:Lambda-def} and
let $\eta>0$. Then
\begin{equation} \label{E:sup-estimate-expexp}
	\sup_u\, \int_{\rn} \exp\exp(\eta|u|)\,\dgn
		\le 2 \int_{0}^{\frac{1}{2}}
			\exp\exp(\eta\Lambda(s) + \eta C)\,\d s,
\end{equation}
where the supremum is extended  over all functions $u$ obeying
\eqref{E:lgn-inf} and \eqref{E:mu}, and $\GG_{B,C}$ is as in Corollary
\ref{C:sup-estimate-expb}.
\end{corollary}

\begin{proof}
Let $B$ be the Young function given by $B(t)=0$ if $t\in [0,1]$ and
$B(t)=\infty$ if $t\in(1,\infty)$. Thus, $L^B(\RR,\nu)=L^\infty(\RR,\nu)$ for any
probability space $(\RR,\nu)$. Moreover, $B^{-1}=1$ on $(0,\infty)$ and hence
\begin{equation*}
	\normwB
		= \|\weight\|_{L^1(s,\frac12)}
		= \int_{s}^\frac12 \weight(r)\,\d r
		\quad\text{for $s\in\left(0,\tfrac{1}{2}\right]$.}
\end{equation*}
Thus,
\begin{equation*}
	\weight(s)sB^{-1}\left(\tfrac{1}{s}\right) + \normwB
		= \Lambda(s)
		\quad\text{for $s\in\left(0,\tfrac{1}{2}\right]$}.
\end{equation*}
By Lemma~\ref{L:Lambda-expansion}, the function $\Lambda$ is integrable on
$\left(0,\tfrac{1}{2}\right)$.  Hence, the constant $C$ appearing in
\eqref{E:C-mu} is well defined. The rest of the proof is analogous to that of
Corollary~\ref{C:sup-estimate-expb}, and will be omitted.
\end{proof}

\section{Sharpness}\label{sharpness}

Here we establish some technical results that are critical in showing the
optimality of the conclusions of Theorems~\ref{T:integral-form} and
\ref{T:norm-form}. The nature of the Ornstein-Uhlenbeck operator and of the
measure in the Gauss space call for the use of ad hoc smooth truncations of
suitable twice weakly differentiable trial functions. Also, the estimates to be
derived for the truncated functions are  quite subtle. By contrast, a more
standard smooth truncation argument suffices in proving the sharpness of the
constant $\alpha_n$ in the companion Euclidean inequality \eqref{supadams}.

We begin with two lemmas, whose proofs rest upon standard properties of weakly
differentiable functions. Notice that the proof of formulas
\eqref{E:u-gradient} and \eqref{E:u-lgn} of the latter also requires the use of
the identities  $\bigl(I\bigl(\Phi(t)\bigr)\bigr)'= -\Phi''(t)=t\Phi'(t)$ for
$t>0$.

\begin{lemma} \label{L:composition-formula}
Let $u\in W^{2,2}\RG$ and let $\psi\in\CC^2(\R)$ be such that $\psi',
\psi''\in L^\infty (\R)$. Then $\psi\circ u\in\Wlgn L^2\RG$ and
\begin{equation} \label{E:composition-formula}
	\lgn(\psi\circ u) = \psi'(u) \lgn u + \psi''(u) |\nabla u|^2.
\end{equation}
\end{lemma}

\begin{lemma} \label{L:solution-formula}
Let $\fno\colon(0,1)\to[0,\infty)$ be a function in $L^2(0,1)$ and let
$\so\in(0,1)$. Define the function $\fni\colon(0,\iso)\to\R$ by
\begin{equation} \label{E:f1-def}
	\fni(s)
		= \int_{s\so}^{\so} \frac{1}{I(\rho)^2} \int_{0}^{\rho}
				\fno\left( \frac{r}{\so} \right)\d r\d\rho
		\quad\text{for $s\in(0,\tiso)$},
\end{equation}
where the function $\fno$ is extended by $0$ in $(1,\iso)$, and
the integral $\int_{s\so}^{\so}\dots\d\rho$ has to be interpreted as
$-\int_{\so}^{s\so}\dots\d\rho$ if $ss_0>s_0$.
Moreover, let $u\colon\rn\to\R$ be the function defined as
\begin{equation} \label{E:solution-def}
	u(x) = \fni\left( \frac{\phxi}{\so} \right)
		\quad \text{for $x\in\R$.}
\end{equation}
Then $u \in  W^{2,2}\RG$,
\begin{equation} \label{E:u-gradient}
	\nabla u(x)
		= \big( \ho\bigl(\phxi\bigr), 0, \dots ,0\big)
\end{equation}
and
\begin{equation} \label{E:u-lgn}
	\Delta u(x) - x\cdot\nabla u(x)
		= - \fno\left( \frac{\phxi}{\so} \right)
\end{equation}
for \ae $x\in\rn$. Here $\ho\colon(0,1)\to[0,\infty)$ is the function given by
\begin{equation} \label{E:h0-def}
	\ho(s)
		= \frac{1}{I(s)} \int_{0}^{s}
				\fno\left( \frac{r}{\so} \right) \d r
		\quad\text{for $s\in(0,1)$}.
\end{equation}
\end{lemma}

The asymptotic behaviour of the function $\fni$ for functions $\fno$ with a
particular behaviour near zero is the subject of the next result. Its proof is
elementary, and is omitted. This is not the case for the subsequent Lemma
\ref{L:eps-so-si}, which provides us with sharp bounds for a function
implicitly defined by prescribing the value of integral on the right-hand side
of equation \eqref{E:f1-def}.

\begin{lemma} \label{L:solution-expansion}
Let $\beta>0$. Assume that the function $\fno\colon(0,1)\to[0,\infty)$ satisfies
\begin{equation} \label{E:fno-expansion}
	\fno(s)
		= \left( \log\is \right)^\ib + \cdots
		\stoo.
\end{equation}
Let $\fni$ be the function associated with $\fno$ as in \eqref{E:f1-def} for
some $\so\in(0,1)$. Then
\begin{equation} \label{E:f1-expansion}
	\fni(s) = \frac{\beta}{2} \left( \log\is \right)^\ib + \cdots
		\stoo.
\end{equation}
\end{lemma}

\begin{lemma} \label{L:eps-so-si}
Let $\beta>0$ and $\varepsilon>0$. Assume that $\fno\colon(0,1)\to[0,\infty)$
is a decreasing function satisfying~\eqref{E:fno-expansion}.
Then the equation
\begin{equation} \label{E:eps-so-si}
	\varepsilon
		= \int_{\si}^{\so} \frac{1}{I(\rho)^2}
				\int_{0}^{\rho} \fno\left( \frac{r}{\so} \right)\d r\d\rho
\end{equation}
implicitly defines a function $\si=\si(\so)$, with $\si\colon(0,\frac 12)\to
(0,\frac 12)$, such that $\si(\so) \in (0,\so)$,
\begin{equation} \label{E:si-so-limit}
	\frac{\si}{\so}\to 0
		\stoo
\end{equation}
and, for any $\varepsilon'>\varepsilon$,
\begin{equation} \label{E:log-so-si}
	\log\frac{\so}{\si}
		\le \left[ 2\varepsilon'\left( \ib+1 \right)\log\iso \right]^{\frac{\beta}{1+\beta}}
		\quad\text{for $\so$ near zero}.
\end{equation}
\end{lemma}

\begin{proof}
Equation \eqref{E:eps-so-si} actually defines a function $\si(\so)$ since, for
each $\so\in (0,\tfrac 12)$, the function
\begin{equation*}
 	s \mapsto \int_{s}^{\so} \frac{1}{I(\rho)^2}
							\int_{0}^{\rho} \fno\left( \frac{r}{\so} \right)\d r\d\rho
		\quad\text{for $s\in(0,\so)$}
\end{equation*}
is strictly decreasing and, owing to \eqref{E:f1-expansion}, it maps $(0,\so)$
into $(0,\infty)$. Let $\weight$ be the function given by
\eqref{E:weight-def}. From equation \eqref{E:eps-so-si} one can easily infer
that
\begin{equation*}
	\varepsilon
		\le \int_{\si}^{\frac12} \frac{1}{I(\rho)^2}
					\int_{0}^{\so} \fno\left( \frac{r}{\so} \right)\d r\d\rho
		= \weight(\si) \int_{0}^{\so} \fno\left( \frac{r}{\so} \right)\d r
		= \so\weight(\si) \int_{0}^1 \fno(r)\,\d r.
\end{equation*}
Notice that the last integral is convergent. By Lemma~\ref{L:weight-expansion},
$\si\to0^+$ and $\si\weight(\si)\to 0$ as $\so\to0^+$. Hence, equation
\eqref{E:si-so-limit} follows, inasmuch as
\begin{equation*}
	\frac{\si}{\so}\varepsilon
		\le \si\weight(\si) \int_{0}^1 \fno(r)\,\d r.
\end{equation*}
Next, an application of Fubini's theorem to the integral on the right-hand side
of equation \eqref{E:eps-so-si} tells us that
\begin{align} \label{E:eps-three-terms}
\begin{split}
	\varepsilon
		& = \int_{\si}^{\so} \frac{1}{I(\rho)^2}
					\int_{0}^{\si} \fno\left( \frac{r}{\so} \right)\d r\d\rho
				+ \int_{\si}^{\so} \frac{1}{I(\rho)^2}
						\int_{\si}^{\rho} \fno\left( \frac{r}{\so} \right)\d r\d\rho
				\\
		&	= [\weight(\si) - \weight(\so)]
					\int_{0}^{\si} \fno\left( \frac{r}{\so} \right)\d r
				+ \int_{\si}^{\so} \fno\left( \frac{r}{\so} \right)
						[\weight(r)-\weight(\so)]\,\d r
				\\
		& = \weight(\si) \int_{0}^{\si} \fno\left( \frac{r}{\so} \right)\d r
				+ \int_{\si}^{\so} \fno\left( \frac{r}{\so} \right) \weight(r)\,\d r
				- \weight(\so) \int_{0}^{\so} \fno\left( \frac{r}{\so} \right)\d r.
\end{split}
\end{align}
The third addend on the rightmost side of equation \eqref{E:eps-three-terms}
tends to $0$ as $\so\to0^+$ thanks to Lemma~\ref{L:weight-expansion}, since
\begin{equation}\label{july7}
	\weight(\so) \int_{0}^{\so} \fno\left( \frac{r}{\so} \right)\d r
		= \so\weight(\so) \int_{0}^1 \fno(r)\,\d r.
\end{equation}
Let us now focus on the second addend. Fixing $\varepsilon'>\varepsilon$,
choose $\delta>0$ so small that $\varepsilon'(1-\delta)^2>\varepsilon$. By
equations ~\eqref{E:weight-expansion} and~\eqref{E:fno-expansion}, there exists
$r_0\in(0,\tfrac12)$ such that
\begin{equation*}
	\weight(r) > \frac{1-\delta}{2} \frac{1}{r\log\ir}
		\quad\text{and}\quad
	\fno(r) > (1-\delta) \left( \log\tir \right)^\ib
		\quad\text{for $r\in(0,r_0)$}.
\end{equation*}
If $\so$ is sufficiently small, then, by
\eqref{E:si-so-limit},  $\so r_0>\si$. Furthermore,
\begin{equation} \label{E:middle-int-lb}
	\int_{\si}^{\so} \fno\left( \frac{r}{\so} \right) \weight(r)\,\d r
		\ge \int_{\si}^{\so r_0} \fno\left( \frac{r}{\so} \right) \weight(r)\,\d r
		\ge \frac{(1-\delta)^2}{2} \int_{\si}^{\so r_0}
					\frac{\left( \log\frac{\so}{r} \right)^\ib}{r\log\ir}\,\d r.
\end{equation}
The change of variables $\log\tir = t\log\tfrac{1}{\so}$ yields
\begin{equation} \label{E:middle-int-psi}
	\int_{\si}^{\so r_0} \frac{\left( \log\frac{\so}{r} \right)^\ib}{r\log\ir}\,\d r
		= \left( \log\tiso \right)^\ib
				\left[
						\Psi_{\ib}\left( \frac{\log\isi}{\log\iso} \right)
					- \Psi_{\ib}\left( \frac{\log\frac{1}{\so r_0}}{\log\iso} \right)
				\right],
\end{equation}
where $\Psi_\ib$ is the function defined as in  \eqref{E:Psi}. Clearly
\begin{equation} \label{E:trivial-lim}
	\lim_{\so\to0^+} \frac{\log\frac{1}{\so r_0}}{\log\iso} = 1.
\end{equation}
We claim that
\begin{equation} \label{E:log-lim}
	\lim_{\so\to0^+} \frac{\log\isi}{\log\iso} = 1.
\end{equation}
Trivially,
\begin{equation*}
	\liminf_{\so\to0^+} \frac{\log\isi}{\log\iso}
		\ge 1.
\end{equation*}
Thus, on setting
\begin{equation*}
	L = \limsup_{\so\to0^+} \frac{\log\isi}{\log\iso},
\end{equation*}
equation \eqref{E:log-lim} will follow if we show that $L=1$.  Suppose, by
contradiction, that $L>1$.  Therefore, fixing $L'\in(1,L)$, there exists
a~decreasing sequence $\{\so^{k}\}$ such that $\so^{k}\to0^+$ and that, on
defining $\si^k=\si(\so^k)$,
\begin{equation*}
	\frac{\log\frac{1}{\si^k}}{\log\frac{1}{\so^k}} > L'
		\quad\text{for $k\in\N$}.
\end{equation*}
By equations \eqref{E:middle-int-lb},
\eqref{E:middle-int-psi} and \eqref{E:trivial-lim},
\begin{equation} \label{E:contra-1}
	\int_{\si^k}^{\so^k} \fno\left( \frac{r}{\so^k} \right) \weight(r)\,\d r
		\ge \frac{(1-\delta)^2}{2} \left( \log\frac{1}{\so^k} \right)^\ib \Psi_\ib(L')
				\quad\text{for large $k\in\N$.}
\end{equation}
Owing to equation ~\eqref{E:eps-three-terms}, the integral on the left-hand side
of~\eqref{E:contra-1} does not exceed
\begin{equation*}
	\varepsilon
		+ \weight(\so^k) \int_{0}^{\so^k} f_0\left(\frac{r}{\so^k}\right)\,\d r.
\end{equation*}
By equation \eqref{july7}, this quantity tends to $\varepsilon$ as
$k\to\infty$. On the other hand, the right hand side of~\eqref{E:contra-1}
tends to infinity as $k\to\infty$. This contradiction
establishes~\eqref{E:log-lim}.

Next, as is easily verified,
\begin{equation*}
	\Psi_\ib(s)
		= \frac{\beta}{1+\beta} (s-1)^{\ib+1} + \cdots
			\quad\text{as $s\to 1^+$.}
\end{equation*}
Hence,  owing to \eqref{E:trivial-lim} and \eqref{E:log-lim}, equation
\eqref{E:middle-int-psi} tells us that
\begin{equation} \label{E:middle-int-psi-expansion}
	\int_{\si}^{\so r_0} \frac{\left( \log\frac{\so}{r} \right)^\ib}{r\log\ir}\,\d r
		= \frac{\beta}{1+\beta} \left( \log\tiso \right)^\ib
			\left[
					\left( \frac{\log\isi}{\log\iso} - 1 \right)^{\ib+1} + \cdots
				- \left( \frac{\log\frac1{r_0}}{\log\iso} \right)^{\ib+1} + \cdots
			\right]
\end{equation}
as $\so\to0^+$. Assume, by contradiction, that \eqref{E:log-so-si} does not
hold. Thus, there exists a sequence $\{\so^k\}$ such that $\so^k\to0^+$ and
that, on setting $\si^k=\si(\so^{k})$,
\begin{equation*}
	\log\frac{\so^k}{\si^k}
	 > \left[
	 		2\varepsilon'\left( \ib+1 \right) \log\frac{1}{\so^k}
		\right]^\frac{\beta}{1+\beta}
		\quad\text{for $k\in\N$}.
\end{equation*}
This inequality is equivalent to
\begin{equation} \label{E:log-lim-contra}
	\left( \frac{\log\frac{1}{\si^k}}{\log\frac{1}{\so^k}} - 1 \right)^{\ib+1}
		> 2\varepsilon'\left( \ib+1 \right)\left( \log\frac{1}{\so^k} \right)^{-\ib}
		\quad\text{for $k\in\N$}.
\end{equation}
Combining equations \eqref{E:eps-three-terms}, \eqref{E:middle-int-lb},
\eqref{E:middle-int-psi-expansion} and \eqref{E:log-lim-contra} yields
\begin{align*}
	\varepsilon
		& \ge \lim_{k\to\infty} \frac{(1-\delta)^2}{2} \int_{\si^k}^{\so^k r_0}
						\frac{\left( \log\frac{\so^k}{r} \right)^\ib}{r\log\ir}\,\d r
				\\
		& \ge \frac{(1-\delta)^2}{2}\frac{\beta}{1+\beta}
				\left[
					2\varepsilon'\left( \ib+1 \right)
					- \lim_{k\to\infty} \left(
							\frac{\left(\log\frac{1}{r_0}\right)^{\ib+1}}{\log\frac{1}{\so^k}}
							+ \cdots \right)
				\right]
			= (1-\delta)^2\varepsilon',
\end{align*}
a contradiction, because of our choice of $\delta$.
\end{proof}

The failure of inequalities \eqref{E:int} and \eqref{E:norm} for exponents
$\theta>\frac 2\beta$ is a consequence of the following proposition.

\begin{proposition} \label{P:sharpness-supercritical}
Let $\beta>0$, $M>1$ and $\theta>\frac 2 \beta$. Suppose that
$B$ is a Young function of the form \eqref{BN} for some $N>0$. Then,
there exists a function $u\in \WexpLb$ such that
$\med(u) = \mv(u) = 0$,
\begin{equation} \label{E:constraints-single}
	\|\lgn u\|_{L^B\RG} \le 1,
	\quad
	\int_{\rn} \Expb(|\lgn u|)\,\dgn \le M
\end{equation}
and
\begin{equation} \label{E:diverges-for-single-function}
	\int_{\rn} \expb(\theta|u|)\,\dgn = \infty.
\end{equation}
\end{proposition}

\begin{proof}
Let $t_0$ be a positive number to be chosen later and let $\lambda\in(0,1)$.
Assume that $\psi\in\CC^2(\R)$ is a function such that $\psi=0$ on
$(-\infty,0]$ and $\psi (t)=t$ for $t \in [1,\infty)$, $\psi'(0)=0$. Hence, in
particular, $\psi'(1)=1$, $|\psi'|\le C$ and $|\psi''|\le C$ for a suitable
constant $C>0$.  Define the function $\fno\colon(0,1)\to[0,\infty)$ as
\begin{equation} \label{E:fo-with-lambda-def}
	\fno(s) = \left( \lambda\log\tis \right)^\ib
		\quad\text{for $s\in(0,1)$}.
\end{equation}
Let $\fni\colon(0,\frac1\so)\to[0,\infty)$ be the function defined as in
\eqref{E:f1-def}, where $\so=\Phi(t_0)$. Define the function $u\colon\rn\to\R$
by
\begin{equation} \label{E:u-single-def}
	u(x) = \sgn(x_1) \,\psi
		\left(
			\fni\left( \frac{\phaxi}{\phto} \right)
		\right)
	\quad\text{for $x\in\rn$.}
\end{equation}
Since $u$ is an odd function of the sole  variable $x_1$, it obeys $\med(u)=0$
and $\mv(u)=0$. Of course, the verification of the latter assertion requires
that $u$ be integrable on $\RG$, a property that will follow from the embedding
$\WexpLb\to L^1\RG$, once we have shown that $u \in\WexpLb$. This membership is
in its turn a consequence of Lemmas~\ref{L:composition-formula} and
\ref{L:solution-formula}. Note that the factor $\sgn(x_1)$ and the presence
of the absolute value of $x_1$ do not affect this conclusion, since the
function $\fni\left( \frac{\phxi}{\phto} \right)$ is non-positive for $x_1$ in
a neighborhood of zero, and $\psi$ vanishes identically~in~$(-\infty,0]$.

Denote by $t_1>t_0$ the value, depending on $t_0$, which is implicitly defined
by the equation
\begin{equation} \label{E:t1-def}
	1 = \fni\left( \frac{\phti}{\phto} \right).
\end{equation}
Thereby, equation~\eqref{E:u-single-def} can be rewritten as
\begin{equation} \label{E:u-single-eq}
	u(x) = \sgn(x_1) \times
	\begin{cases}
		\fni\left( \frac{\phaxi}{\phto} \right)
			& \text{for $|x_1|>t_1$}
			\\
		\psi\left( \fni\left( \frac{\phaxi}{\phto} \right) \right)
			& \text{for $t_0<|x_1|\le t_1$}
			\\
		0
			& \text{for $|x_1|\le t_0$}.
	\end{cases}
\end{equation}
Lemmas ~\ref{L:composition-formula} and~\ref{L:solution-formula} then tell us that
\begin{equation} \label{E:u-single-lgn}
	\lgn u(x) = \sgn(x_1) \times
	\begin{cases}
		- \fno\left( \frac{\phaxi}{\phto} \right)
			& \text{for $|x_1|>t_1$}
			\\
		- \psi'\left( \fni\left( \frac{\phaxi}{\phto} \right) \right)
				\fno \left( \frac{\phaxi}{\phto} \right)
				& \\
		\quad + \psi''\left( \fni\left( \frac{\phaxi}{\phto} \right) \right)
				\ho\bigl(\phaxi\bigr)^2
			& \text{for $t_0<|x_1|\le t_1$}
			\\
		0
			& \text{for $|x_1|\le t_0$},
	\end{cases}
\end{equation}
where the function $\ho$ is defined as in \eqref{E:h0-def}.
Owing to our bounds on the derivatives of $\psi$,
\begin{equation} \label{E:u-single-lgn-ub}
	|\lgn u(x)| \le
	\begin{cases}
		\fno\left( \frac{\phaxi}{\phto} \right)
			& \text{for $|x_1|>t_1$}
			\\
		C \fno \left( \frac{\phaxi}{\phto} \right)
			+ C \ho\bigl(\phaxi\bigr)^2
			& \text{for $t_0<|x_1|\le t_1$}
			\\
		0
			& \text{for $|x_1|\le t_0$}.
	\end{cases}
\end{equation}

We begin by observing that $\lambda$ can be chosen in such a way that equation
\eqref{E:diverges-for-single-function} holds. Plainly,
\begin{equation} \label{E:exp-u-lb}
	\int_{\rn} \expb(\theta|u|)\,\dgn
		\ge \int_{t_1}^{\infty}
			\expb\left( \theta \fni\left( \frac{\Phi(x_1)}{\phto} \right) \right)\dgn(x_1)
		= \int_{0}^{\phti} \expb\left( \theta \fni\left( \frac{s}{\phto} \right) \right) \d s.
\end{equation}
Lemma~\ref{L:solution-expansion} tells us that
\begin{equation*}
	\fni\left( \frac{s}{\phto} \right)
		= \frac{1}{\thetab}\left( \lambda\log\is \right)^\ib
			+ \cdots
		\stoo.
\end{equation*}
Hence, the integral on the rightmost side of equation \eqref{E:exp-u-lb}
diverges provided that
\begin{equation} \label{E:lambda-condition-2}
	\lambda\left( \frac{\theta}{\thetab} \right)^\beta > 1.
\end{equation}

The remaining part of the proof is devoted to showing that $t_0$ can be chosen
so large that the inequalities in \eqref{E:constraints-single} are fulfilled as
well. Let $B$ be a Young function such that $B(\tau)=N\expb(\tau)$ for
$\tau\ge\tau_0$. Then
\begin{align} \label{E:B-lgn-ub}
	\begin{split}
	\int_{\rn} B(|\lgn u|)\,\dgn
		& \le \int_{\{0<|\lgn u|<\tau_0\}} B(|\lgn u|)\,\dgn
			+ N \int_{\{|\lgn u|\ge\tau_0\}} \expb(|\lgn u|)\,\dgn
			\\
		& \le B(\tau_0) \int_{\{|\lgn u|>0\}} \dgn
			+ N \int_{\{|x_1|>t_0\}} \expb(|\lgn u|)\,\dgn.
	\end{split}
\end{align}
Note that
\begin{equation} \label{E:measure-of-halfspaces}
	\int_{\{|\lgn u|>0\}} \dgn  \le  2\Phi(t_0),
\end{equation}
since the support of $\lgn u$ is contained in the union of the halfspaces $\{x_1>t_0\}$ and
$\{x_1<-t_0\}$, both having Gauss measure equal to $\Phi(t_0)$. From equation
\eqref{E:u-single-lgn-ub} and   the change of variables $s=\phxi$, we obtain that
\begin{align} \label{E:exp-lgn}
	\begin{split}
	\int_{\{|x_1|>t_0\}} \expb(|\lgn u|)\,\dgn
		& \le 2 \int_{0}^{\phti} \expb\left( \fno\left( \frac{s}{\phto} \right) \right)\d s
			\\
		& \quad + 2 \int_{\phti}^{\phto}
				\expb\left( C\fno\left( \frac{s}{\phto} \right) + C\ho(s)^2 \right)\d s.
	\end{split}
\end{align}
The first integral  on the right-hand side of inequality  \eqref{E:exp-lgn} can
be estimated as
\begin{equation} \label{E:exp-f0}
	\int_{0}^{\phti} \expb\left( \fno\left( \frac{s}{\phto} \right) \right)\d s
		\le \int_{0}^{\phto} \exp\left( \lambda\log\frac{\phto}{s} \right)\d s
		= \phto \int_{0}^{1} \frac{\d s}{s^\lambda}
		= \frac{\phto}{1-\lambda}.
\end{equation}
Let us next focus on the second integral on the right-hand side of
\eqref{E:exp-lgn}. Since $\fno$ is nonnegative and decreasing, its integral
average is also decreasing. On the other hand, $s/I(s)$ is increasing.
Therefore
\begin{equation} \label{E:h0-ub}
	\ho(s)
		= \frac{s}{I(s)}\,\is
				\int_{0}^{s} \fno\left( \frac{r}{\phto} \right)\d r
		\le \frac{\phto}{I\bigl(\phto\bigr)}\,\frac{1}{\phti}
					\int_{0}^{\phti} \fno\left( \frac{r}{\phto} \right)\d r
\end{equation}
for $s\in(\phti,\phto)$. Note that, by~\eqref{E:t1-def},
\begin{equation*}
	\int_{\Phi(t_1)}^{\Phi(t_0)} \frac{1}{I(s)^2}
				\int_{0}^{s} \fno\left( \frac{r}{\Phi(t_0)} \right)\d r\d s = 1.
\end{equation*}
Hence Lemma~\ref{L:eps-so-si} with $\varepsilon=1$, $\so=\phto$ and $\si=\phti$
implies that
\begin{equation} \label{E:Pht1-Pht0-to-zero}
	\frac{\phti}{\phto} \to 0
		\quad\text{as $t_0\to\infty$.}
\end{equation}
Also, by equations~\eqref{E:log-so-si} and \eqref{E:logPhi}, there exists a constant
$C_{\beta,\lambda}$, depending on $\beta$ and $\lambda$, such that
\begin{equation} \label{E:Pht1-Pht0-log-ub}
	\log\frac{\phto}{\phti}
		\le C_{\beta,\lambda} t_0^{\frac{2\beta}{1+\beta}}
		\quad\text{for large $t_0$.}
\end{equation}
Next, since
\begin{equation*}
	\int_{0}^{s} \left( \log\tir \right)^\ib \d r
		= s\left( \log\tis \right)^\ib
			+ \cdots
		\stoo,
\end{equation*}
from equation \eqref{E:Pht1-Pht0-to-zero} and a change of
variables one deduces that
\begin{align} \label{E:int-f0-asymp}
	\begin{split}
	\frac{1}{\phti} \int_{0}^{\phti} \fno\left( \frac{r}{\phto} \right)\d r
		 = \frac{\phto}{\phti} \int_{0}^{\frac{\phti}{\phto}}
					\left( \lambda\log\tir \right)^\ib \d r
		 = \left( \lambda\log\frac{\phto}{\phti} \right)^\ib
				+ \cdots
	\end{split}
\end{align}
as $t_0\to\infty$. By equations \eqref{E:relation-between-I-and-Phi} and
\eqref{E:Phi-prime}, one has that $I\bigl(\Phi(t)\bigr) = -\Phi'(t) = t\Phi(t)
+ \cdots$ as $t\to\infty$. Hence,
\begin{equation} \label{E:Pht1-IPht1-asymp}
	\frac{\phto}{I\bigl(\phto\bigr)}
		= \frac{1}{t_0}
			+ \cdots
		\quad\text{as $t_0\to\infty$}.
\end{equation}
Thanks to equations \eqref{E:Pht1-Pht0-log-ub}--\eqref{E:Pht1-IPht1-asymp}, inequality \eqref{E:h0-ub} implies that
\begin{equation} \label{E:exp-h0-ub-2}
	\ho(s)^{2\beta}
		\le C_{\beta,\lambda}'
			t_0^{2\beta\left( \frac{2}{1+\beta}-1 \right)}
		= C_{\beta,\lambda}'
			t_0^{(1-\beta)\frac{2\beta}{1+\beta}}
\end{equation}
for large $t_0$ and for $s\in(\phti,\phto)$, where $C_{\beta,\lambda}'$ is
a suitable constant depending on $\beta$ and $\lambda$.  The use of inequality
\eqref{E:Pht1-Pht0-log-ub} yields
\begin{equation} \label{E:exp-f0-ub}
	\fno\left( \frac{\phti}{\phto} \right)^\beta
		\le \lambda C_{\beta,\lambda}\, t_0^{\frac{2\beta}{1+\beta}}
		\quad\text{for large $t_0$.}
\end{equation}
Also, by equation \eqref{E:Phi-prime},
\begin{equation} \label{E:Phi-expansion}
	\phto
		= - \frac{\Phi'(t_0)}{t_0}
			+ \cdots
		= \frac{1}{t_0 \sqrt{2\pi}} \exp\left( -\frac{t_0^2}{2} \right)
			+ \cdots
		\quad\text{as $t_0\to\infty$}.
\end{equation}
Finally, we infer from equations \eqref{E:exp-h0-ub-2}--\eqref{E:Phi-expansion} that
\begin{align} \label{E:exp-h0-ub-3}
	\begin{split}
	& \int_{\phti}^{\phto} \expb\left( C\fno\left( \frac{s}{\phto} \right) + C\ho(s)^2 \right)\d s
		\\
		& \qquad \le
			\phto
			\exp\left( C_\beta \fno\left( \frac{\phti}{\phto} \right)^\beta
				 + C_\beta \sup_{s\in(\phti,\phto)} \ho(s)^{2\beta} \right)
			\\
		& \qquad \le \frac{1}{t_0\sqrt{2\pi}} \exp
					\left(
						- \frac{t_0^2}{2}
						+ \lambda C_\beta C_{\beta,\lambda} t_0^{\frac{2\beta}{1+\beta}}
						+ C_\beta C_{\beta,\lambda}' t_0^{(1-\beta)\frac{2\beta}{1+\beta}}
					\right)
				+ \cdots
		\quad\text{as $t_0\to\infty$},
	\end{split}
\end{align}
for some constant $C_\beta$ depending on $C$ and $\beta$. Equation
\eqref{E:exp-h0-ub-3} tells us that its leftmost side tends to $0$ as
$t_0\to\infty$ for any fixed   $\beta$ and $\lambda$. Thus, combining equations
\eqref{E:B-lgn-ub}--\eqref{E:exp-f0} and \eqref{E:exp-h0-ub-3} enables us to deduce that
\begin{equation} \label{E:B-lgn-ub-2}
	\int_{\rn} B(|\lgn u|)\,\dgn
		\le 2B(\tau_0) \phto
			+ \frac{2N}{1-\lambda} \phto
			+ R(\beta,\lambda,t_0),
\end{equation}
where the expression $R(\beta,\lambda,t_0)$ has the property that
$R(\beta,\lambda,t_0)\to 0$ as $t_0\to\infty$.

Moreover, since there exists $\tau_0\ge 0$ such that
\begin{equation*}
	\Expb(\tau) \le
	\begin{cases}
		\frac{\tau}{\tau_0}\left( \expb(\tau_0) - 1 \right) + 1
			& \text{for $\tau\in[0,\tau_0)$}
			\\
		\expb(\tau)
			& \text{for $\tau\in[\tau_0,\infty)$},
	\end{cases}
\end{equation*}
we have that
\begin{align} \label{E:exp-lgn-ub}
	\begin{split}
	\int_{\rn} \Expb(|\lgn u|)\,\dgn
		& \le \int_{\{\lgn u=0\}} \dgn
				+ \int_{\{0<|\lgn u|<\tau_0\}} \expb(\tau_0)\,\dgn
				+ \int_{\{|\lgn u|\ge\tau_0\}} \expb(|\lgn u|)\,\dgn
			\\
		& \le 1
				+ \expb(\tau_0) \int_{\{|\lgn u|>0\}} \dgn
				+ \int_{\{|x_1|>t_0\}} \expb(|\lgn u|)\,\dgn
			\\
		& \le 1
				+ 2\expb(\tau_0)\phto
				+ \frac{1}{1-\lambda}\phto
				+ R(\beta,\lambda,t_0).
	\end{split}
\end{align}
Note that here we have also made use of equations
\eqref{E:measure-of-halfspaces}--\eqref{E:exp-f0} and
\eqref{E:exp-h0-ub-3}. Inequalities \eqref{E:B-lgn-ub-2} and
\eqref{E:exp-lgn-ub} ensure that, given $\beta>0$, $\lambda\in(0,1)$ and
either a function $B$ as in \eqref{BN} or a number $M>1$, the
number $t_0>0$ can be chosen large enough for both inequalities in
\eqref{E:constraints-single} to hold.
\end{proof}

Propositions~\ref{P:sharpness-critical-negative} and
\ref{P:sharpness-critical-positive} deal with the case when $\beta>1$. They
imply that there actually exist Young functions $B$ obeying \eqref{BN} for
which the threshold value $\frac2\beta$ of the constant $\theta$ in inequality
\eqref{E:norm} is not attained, and Young functions $B$ for which it is
attained.

\begin{proposition} \label{P:sharpness-critical-negative}
Let $\beta\in(1,\infty)$. For every $N>0$, there exists $\tau_0>0$, a Young
function $B$ as in \eqref{BN}, and a sequence of functions $\{u_k\}\subset
\Wlgn\expLb$, such that $\med(u_k)=\mv(u_k)=0$,
\begin{equation} \label{E:constraint-series}
	\|\lgn u_k\|_{L^B\RG} \le 1
		\quad\text{for $k\in\N$}
\end{equation}
and
\begin{equation} \label{E:diverges-for-series}
	\lim_{k\to\infty} \int_{\rn} \expb(\tfrac 2\beta |u_k|)\,\dgn = \infty.
\end{equation}
\end{proposition}

\begin{proof}
Let us set, for brevity of notation, $A(\tau)=\expb(\tau)$, $\tau\ge 0$, and
call $a$ the derivative of this function.  Given $\tau_0>0$, define the function
$\AU\colon[0,\infty)\to[0,\infty)$ as
\begin{equation} \label{E:A-sub-def}
	\AU(\tau) =
	\begin{cases}
		0
			& \text{for $\tau\in[0,\tau_0']$}
			\\
		 A(\tau_0) + a(\tau_0)(\tau-\tau_0)
			& \text{for $\tau\in(\tau_0',\tau_0]$}
			\\
		A(\tau)
			& \text{for $\tau\in(\tau_0,\infty)$},
	\end{cases}
\end{equation}
where $\tau_0'\in(0,\tau_0)$ is given by
\begin{equation} \label{E:t0'-def}
	\tau_0' = \tau_0 - \frac{A(\tau_0)}{a(\tau_0)}.
\end{equation}
Observe that $\AU$ is a Young function. Moreover,
\begin{equation} \label{E:a-sub}
	\aU(\tau) =
	\begin{cases}
		0
			& \text{for $\tau\in(0,\tau_0']$}
			\\
		a(\tau_0)
			& \text{for $\tau\in(\tau_0',\tau_0]$}
			\\
		a(\tau)
			& \text{for $\tau\in(\tau_0,\infty)$}
	\end{cases}
	\quad\text{and}\quad
	\aU^{-1}(t) =
	\begin{cases}
		\tau_0'
			& \text{for $t\in(0,a(\tau_0)]$}
			\\
		a^{-1}(t)
			& \text{for $t\in(a(\tau_0),\infty)$},
	\end{cases}
\end{equation}
where $\aU$ stands for the left-continuous derivative of $\AU$ and $\aU^{-1}$
for its generalized left-continuous inverse. Define the function
$\fno\colon(0,1)\to[0,\infty)$ by
\begin{equation} \label{E:f0-def}
	\fno(r) = \aU^{-1}\left( \weight\left( \tfrac{r}{2} \right) \right)
		\quad\text{for $r\in(0,1)$,}
\end{equation}
and the function $\fni\colon(0,1)\to[0,\infty)$ as in \eqref{E:f1-def}, with $\so=1/2$.
If we set
\begin{equation} \label{E:tau-def}
	t_0 = \weight^{-1}\bigl(a(\tau_0)\bigr),
\end{equation}
then, thanks to equation \eqref{E:a-sub}, we have that
\begin{equation} \label{E:f0-eq}
	\fno(r) =
	\begin{cases}
		a^{-1}\left( \weight\left( \tfrac{r}{2} \right) \right)
			& \text{for $r\in(0,2t_0]$}
			\\
		\tau_0'
			& \text{for $r\in(2t_0,1)$}.
	\end{cases}
\end{equation}
Observe that, by Lemmas~\ref{L:weight-expansion} and
\ref{L:a-invers},
\begin{equation} \label{E:f0-expansion}
	\fno(r) = \left( \log\tir \right)^\ib + \cdots
		\quad\text{as $r\to 0^+$}.
\end{equation}
Let $\psi\in\CC^2(\R)$ be a function such that $\psi(t)=t$ for
$t\in(-\infty,1]$, $\psi(t)=1$ for $t\in[2,\infty)$, $\psi(t)\ge 1$ for
$t\in(1,2)$ and $|\psi'|\le 1$. Hence, $\psi'(1)=1$ and there exists a positive
constant $C$ such that $|\psi''|\le C$.

Define, for each $k\in\N$, the function $v_k\colon\rn\to\R$ by
\begin{equation*}
	v_k(x)
		= \fni\bigl(2\phk\bigr)\,
			\psi\left( \frac{\fni\bigl(2\phxi\bigr)}{\fni\bigl(2\phk\bigr)} \right)
		\quad\text{for $x\in\rn$}
\end{equation*}
and the function $w_k\colon\rn\to\R$ as
\begin{equation} \label{E:uk-minus-def}
	w_k(x)=-v_k(-x)
		\quad\text{for $x\in\rn$}.
\end{equation}
From  Lemmas~\ref{L:composition-formula} and~\ref{L:solution-formula} one
deduces that $v_k, w_k\in W^{2,2}\RG$ for every $k\in\N$. One can also verify
that the function $u_k\colon\rn\to\R$, given by
\begin{equation*}
	u_k(x) =
	\begin{cases}
		v_k(x)
			& \text{for $x_1\ge 0$}
			\\
		w_k(x)
			& \text{for $x_1<0$,}
	\end{cases}
\end{equation*}
 obeys $u_k\in W^{2,2}\RG$, and
\begin{equation*}
	\lgn u_k(x) =
	\begin{cases}
		\lgn v_k(x)
			& \text{for $x_1>0$}
			\\
		\lgn w_k(x)
			& \text{for $x_1<0$}.
	\end{cases}
\end{equation*}
Each function $u_k$ is odd in variable $x_1$, whence $\med(u_k)=0$. The fact
that $\mv(u_k)=0$ is a consequence of the same property of $u_k$. Of course,
the existence of $\mv(u_k)$ requires that $u_k$ be integrable, and this is a
consequence of equation \eqref{E:constraint-series}, and of the embedding
$\Wlgn\expLb\to L^1\RG$.

In order to prove property \eqref{E:constraint-series}, denote by $t_k$ the
number implicitly defined by
\begin{equation} \label{E:tk-def}
	\fni\bigl(2\phtk\bigr) = 2 \fni\bigl(2\phk\bigr)
		\quad\text{for $k\in\N$}.
\end{equation}
Notice that $t_k>k$ for $k \in \N$.
Then, $u_k$ takes the form
\begin{equation} \label{E:uk-def}
	u_k(x) = \sgn(x_1)\times
	\begin{cases}
		\fni\bigl(2\phk\bigr)
			& \text{for $|x_1|\ge t_k$}
			\\
		\fni\bigl(2\phk\bigr)\,
		\psi\left( \frac{\fni(2\phaxi)}{\fni(2\phk)} \right)
			& \text{for $k\le |x_1|< t_k$}
			\\
		\fni\bigl(2\phaxi\bigr)
			& \text{for $|x_1|< k$}.
	\end{cases}
\end{equation}
Furthermore, by Lemmas~\ref{L:composition-formula}  and
\ref{L:solution-formula}, with $\so=1/2$,
\begin{equation} \label{E:uk-lgn}
	\lgn u_k(x) = \sgn(x_1)\times
	\begin{cases}
		0
			& \text{for $|x_1|\ge t_k$}
			\\
		- \psi'\left( \frac{\fni(2\phaxi)}{\fni(2\phk)} \right)
			\fno\bigl(2\phaxi\bigr)
			&\\
		\quad + \psi''\left( \frac{\fni(2\phaxi)}{\fni(2\phk)} \right)
			\frac{\ho(\phaxi)^2}{\fni(2\phk)}
			& \text{for $k\le |x_1|< t_k$}
			\\
		-\fno\bigl(2\phaxi\bigr)
			& \text{for $|x_1|< k$},
	\end{cases}
\end{equation}
where $\ho$ is given by \eqref{E:h0-def}. Since $|\psi'|\le 1$ and
$|\psi''|\le C$,
\begin{equation} \label{E:uk-lgn-ub}
	|\lgn u_k(x)| \le
	\begin{cases}
		0
			& \text{for $|x_1|\ge t_k$}
			\\
		\fno\bigl(2\phaxi\bigr)
		+ C \frac{\ho(\phaxi)^2}{\fni(2\phk)}
			& \text{for $k\le |x_1|< t_k$}
			\\
		\fno\bigl(2\phaxi\bigr)
			& \text{for $|x_1|< k$}.
	\end{cases}
\end{equation}
Equation \eqref{E:constraint-series} will be shown to hold with $B=N\AU$.
Thanks to equation \eqref{E:uk-lgn-ub} and the change of variables $\phxi=s$,
\begin{align} \label{E:B-uk-lgn-ub}
	\begin{split}
	 \int_{\rn} B(|\lgn u_k|)\,\dgn
		& \le \int_{\{|x_1|<k\}} B\bigl(\fno\bigl(2\phaxi\bigr)\bigr)\,\d\gamma_1(x_1)
			\\
			&  \quad+ \int_{\{k\le |x_1|<t_k\}} B\left( \fno\bigl(2\phaxi\bigr)
					+  C \frac{\ho\bigl(\phaxi\bigr)^2}{\fni\bigl(2\phk\bigr)}\right)\d\gamma_1(x_1)
		\\
		&   = 2N \int_{\phk}^{\frac{1}{2}} \AU\bigl(\fno(2s)\bigr)\,\d s
				+ 2N \int_{\phtk}^{\phk}
					\AU\left( \fno(2s) + C\frac{\ho(s)^2}{\fni\bigl(2\phk\bigr)} \right)\d s.
	\end{split}
\end{align}
Consider  the first integral on the rightmost side of inequality
\eqref{E:B-uk-lgn-ub}. Assume that $k$ is so large that
\begin{equation} \label{E:Phk-condition}
	\phk < t_0.
\end{equation}
Then
\begin{equation} \label{E:Mtau0-def}
	\int_{\phk}^{\frac{1}{2}} \AU\bigl(\fno(2s)\bigr)\,\d s
		= \int_{\phk}^{t_0} \AU\left( a^{-1}\bigl(\weight(s)\bigr) \right)\d s
		\le \int_{0}^{t_0} A\left( a^{-1}\bigl(\weight(s)\bigr) \right)\d s
		= M(t_0).
\end{equation}
We claim that $M(t_0)\to0$ as $t_0\to0^+$. To verify this fact, it suffices to
prove that
\begin{equation} \label{E:A-a-converges}
	\int_{0} A\left( a^{-1}\bigl(\weight(s)\bigr) \right)\d s < \infty.
\end{equation}
By equation \eqref{E:B-beta-invers},
\begin{equation*}
	A\bigl(a^{-1}(t)\bigr)
				= \ib t (\log t)^{\ib-1} + \cdots
			\ttoinf,
\end{equation*}
whence, owing to Lemma~\ref{L:weight-expansion}, one deduces that
\begin{equation*}
	A\left( a^{-1}\bigl(\weight(s)\bigr) \right)
		= \frac{1}{2\beta s}\left( \log\tis \right)^{\ib-2} + \cdots
		\stoo.
\end{equation*}
Hence, \eqref{E:A-a-converges} follows, since we are assuming that $\beta>1$.

Let us next focus on the second integral on the rightmost side of inequality
\eqref{E:B-uk-lgn-ub}. Let $h_0$ be the function defined by~\eqref{E:h0-def}.
Assume that $k$ is large enough for inequality \eqref{E:Phk-condition} to hold.
Via equation \eqref{E:tk-def} and the monotonicity of $\fni$ we obtain that
\begin{align} \label{E:B-uk-lgn-second}
	\begin{split}
	\int_{\phtk}^{\phk}&
		\AU\left( \fno(2s) + C\frac{\ho(s)^2}{\fni\bigl(2\phk\bigr)} \right)\d s
		 = \int_{\phtk}^{\phk}
		\AU\left( \fno(2s) + 2C\frac{\ho(s)^2}{\fni\bigl(2\phtk\bigr)} \right)\d s
			\\
		& \le \int_{\phtk}^{\phk}
		\AU\left( \fno(2s) + 2C\frac{\ho(s)^2}{\fni(2s)} \right)\d s
		 \le \int_{0}^{t_0}
		\expb\left( \fno(2s) + 2C\frac{\ho(s)^2}{\fni(2s)} \right)\d s.
	\end{split}
\end{align}
Our next task is to show that the integral on the right-hand side of
\eqref{E:B-uk-lgn-second} tends to $0$ as $t_0\to 0^+$. From equation
\eqref{E:f0-expansion}, Lemma~\ref{L:I-expansion} and L'H\^opital's rule, one
can infer that
\begin{equation} \label{E:h0-expansion}
	\ho(s)
		= \frac{1}{\sqrt{2}}\left( \log\tis \right)^{\ib-\frac12}
			+ \cdots
		\stoo.
\end{equation}
Hence, thanks to Lemma~\ref{L:solution-expansion} and inequality~\eqref{E:f0-expansion},
\begin{equation} \label{E:f1-expansion-repeat}
	\fni(s)
		= \frac{\beta}{2}\left( \log\tis \right)^\ib
			+ \cdots
		\stoo.
\end{equation}
Coupling expansions \eqref{E:h0-expansion} and \eqref{E:f1-expansion-repeat} yields
\begin{equation} \label{E:h02/f1-expansion}
	\frac{\ho(s)^2}{\fni(2s)}
		= \ib\left( \log\tis \right)^{\ib-1}
			+ \cdots
		\stoo.
\end{equation}
The asymptotic expansion \eqref{E:h02/f1-expansion} and \eqref{E:f0-expansion}
ensure that $\fno(2s)$ is the leading term in the sum $\fno(2s) +
2C\frac{\ho(s)^2}{\fni(2s)}$. Therefore,  by equation \eqref{E:power} with
$\sigma=\beta$,
\begin{align} \label{E:all-expansion}
	\begin{split}
	\left( \fno(2s) + 2C\frac{\ho(s)^2}{\fni(2s)} \right)^\beta
		& = \fno(2s)^\beta
				+ \beta 2C\fno(2s)^{\beta-1} \frac{\ho(s)^2}{\fni(2s)}
				+ \cdots
			\\
		& = \fno(2s)^\beta
				+ 2C
				+ \cdots
			\stoo.
	\end{split}
\end{align}
Notice that  the second equality relies upon equations \eqref{E:f0-expansion}
and \eqref{E:h02/f1-expansion}. Now, equations \eqref{E:all-expansion}
and~\eqref{E:f0-eq} tell us that
\begin{equation} \label{E:exp-all-expansion}
	\expb\left( \fno(2s) + 2C\frac{\ho(s)^2}{\fni(2s)} \right)
		\le e^{4C}\expb\bigl( \fno(2s) \bigr) = e^{4C}A(a^{-1}(\weight(s)))
		\quad\text{for $s\in(0,t_0)$,}
\end{equation}
provided that $t_0$ is  sufficiently small. Therefore,  from  equations
\eqref{E:B-uk-lgn-second} and \eqref{E:exp-all-expansion}  we conclude that
\begin{equation} \label{E:int-all}
	\int_{\phtk}^{\phk}
		\AU\left( \fno(2s) + C\frac{\ho(s)^2}{\fni\bigl(2\phk\bigr)} \right)\d s
		\le e^{4C} M(t_0),
\end{equation}
if $k$ is so large that inequality \eqref{E:Phk-condition} holds. Here,
$M(t_0)$ is the function defined by~\eqref{E:Mtau0-def}. From equations
\eqref{E:B-uk-lgn-ub}, \eqref{E:Mtau0-def} and \eqref{E:int-all} one obtains
that
\begin{equation*}
	\int_{\rn} B(|\lgn u_k|)\,\dgn \le 2N\left( 1+e^{4C} \right) M(t_0),
\end{equation*}
provided that $k$ is large enough for inequality \eqref{E:Phk-condition} to be
fulfilled. By the definition of the Luxemburg norm, if $t_0$ is chosen so
small that $2N(1+e^{4C})M(t_0)\le 1$, then   inequality
\eqref{E:constraint-series} holds for sufficiently large $k$.

It remains to establish property \eqref{E:diverges-for-series}. Since, by our
assumptions on $\psi$,
\begin{equation*}
	|u_k(x)| \ge \fni\bigl(2\phk\bigr)
		\quad\text{for $x_1\ge k$},
\end{equation*}
we have that
\begin{equation} \label{E:exp-uk-lb}
	\int_{\rn} \expb(\thetab |u_k|)\,\dgn
		\ge \int_{\{x_1>k\}} \expb\left( \thetab \fni\bigl(2\phk\bigr) \right)\dgn
		= \phk \expb\left( \thetab \fni\bigl(2\phk\bigr) \right).
\end{equation}
An application of Fubini's theorem tells us that
\begin{align} \label{july20}
	\begin{split}
	\fni(2s)
		& = \int_{s}^{\frac12} \frac{1}{I(\rho)^2} \int_{0}^{\rho} \fno(2r)\,\d r\d\rho
			\ge \int_{s}^{\frac12} \frac{1}{I(\rho)^2} \int_{s}^{\rho} \fno(2r)\,\d r\d\rho
		 	= \int_{s}^{\frac12} \fno(2r) \int_{r}^{\frac12} \frac{1}{I(\rho)^2} \,\d\rho\d r
				\\
		& = \int_{s}^{\frac12} \fno(2r)\weight(r)\,\d r
		 	= \int_{s}^{\tau_0} \fno(2r)\weight(r)\,\d r
				+ \int_{\tau_0}^{\frac12} \fno(2r)\weight(r)\,\d r
		\quad\text{for $s\in[0,t_0)$}.
	\end{split}
\end{align}
By equations \eqref{E:tau-def} and \eqref{E:f0-eq}, the change of variables
$t=\weight(r)$ in the first integral on the rightmost side of equation
\eqref{july20} results in
\begin{equation*}
	\int_{s}^{t_0} \fno(2r)\weight(r)\,\d r
		= \int_{a(\tau_0)}^{\weight(s)} a^{-1}(t)\,t I\bigl(\weight^{-1}(t)\bigr)^2\,\d t
			\quad\text{for $s\in[0,t_0)$}.
\end{equation*}
From Lemma~\ref{L:asymptotic-of-product}, we infer that, if
$\tau_0$ is sufficiently large, then
\begin{equation*}
	a^{-1}(t)\,t I\bigl(\weight^{-1}(t)\bigr)^2
		\ge \frac{1}{2t}(\log t)^{\ib-1}
			- c_\beta \frac{1}{t}(\log t)^{\ib-2}\log\log t,
\end{equation*}
provided that the constant $c_\beta>5/4+(\beta-1)/2\beta^2$, and $t>a(\tau_0)$. Therefore,
\begin{align*}
	\int_{s}^{t_0} \fno(2r)\weight(r)\,\d r
		& \ge \frac{1}{2} \int_{a(\tau_0)}^{\weight(s)} \frac{1}{t}(\log t)^{\ib-1}\, \d t
				- c_\beta \int_{a(\tau_0)}^{\weight(s)} \frac{1}{t}(\log t)^{\ib-2}\log\log t\, \d t
			\\
		& \ge \frac{\beta}{2} \bigl[\log\weight(s)\bigr]^\ib
				- \frac{\beta}{2} \bigl[\log a(\tau_0)\bigr]^\ib
				- c_\beta \int_{a(\tau_0)}^{\infty} \frac{1}{t}(\log t)^{\ib-2}\log\log t\,\d t
\end{align*}
for $s\in[0,t_0)$.
Altogether,
\begin{equation} \label{E:f1-lb}
	\fni(2s)
		\ge \frac{\beta}{2} \bigl[\log\weight(s)\bigr]^\ib
			+ \lambda(\tau_0)
		\quad\text{for $s\in(0,t_0)$,}
\end{equation}
where we have set
\begin{equation} \label{E:lambda-t0-def}
	\lambda(\tau_0)
		= \int_{t_0}^{\frac12} \fno(2r)\weight(r)\,\d r
			- \frac{\beta}{2} \bigl[\log a(\tau_0)\bigr]^\ib
			- c_\beta \int_{a(\tau_0)}^{\infty} \frac{1}{t}(\log t)^{\ib-2}\log\log t\,\d t.
\end{equation}
Let us analyze the asymptotic behaviour of $\lambda(\tau_0)$ as
$\tau_0\to\infty$. Owing to equations \eqref{E:t0'-def}--\eqref{E:f0-def},
\begin{equation} \label{E:int-f0-weight-2}
	\int_{t_0}^{\frac12} \fno(2r)\weight(r)\,\d r
		= \left( \tau_0 - \frac{A(\tau_0)}{a(\tau_0)} \right)
				\int_{t_0}^{\frac12} \weight(r)\,\d r.
\end{equation}
Therefore, by the change of variables $t=\weight(r)$, L'H\^opital's rule and
Lemma~\ref{L:Iwi2-expansion},
\begin{equation*}
	\int_{t_0}^{\frac12} \weight(r)\,\d r
		= \int_{0}^{a(\tau_0)} t I\bigl(\weight^{-1}(t)\bigr)^2\,\d t
		= \frac{1}{2} \log\log a(\tau_0) + \cdots
		= \frac{\beta}{2} \log \tau_0 + \cdots
			\quad\text{as $\tau_0\to\infty$}.
\end{equation*}
Also $A(\tau_0)/a(\tau_0)\to 0$ as $\tau_0\to\infty$. Thus, equation
\eqref{E:int-f0-weight-2} can be rewritten as
\begin{equation*}
	\int_{t_0}^{\frac12} \fno(2r)\weight(r)\,\d r
		= \frac{\beta}{2} \tau_0 \log \tau_0 + \cdots
			\quad\text{as $\tau_0\to\infty$}.
\end{equation*}
Next,
\begin{equation*}
	\bigl[ \log a(\tau_0) \bigr]^\ib
		= \tau_0 + \cdots
			\quad\text{as $\tau_0\to\infty$}
\end{equation*}
and, since the last integral on the right-hand side of equation
\eqref{E:lambda-t0-def} tends to $0$ as $\tau_0\to\infty$, we conclude that
\begin{equation*}
	\lambda(\tau_0)
		= \frac{\beta}{2} \tau_0 \log \tau_0 + \cdots
			\quad\text{as $\tau_0\to\infty$}.
\end{equation*}
In particular, $\lambda(\tau_0)\to\infty$ as $\tau_0\to\infty$. Therefore,
given any $\lambda>0$, one can chose $\tau_0$ so large that
\begin{equation} \label{E:lambda-t0-set}
	\lambda(\tau_0) > \lambda.
\end{equation}
 Inequalities \eqref{E:f1-lb},
\eqref{E:lambda-t0-set} and an elementary inequality yield
\begin{equation*}
	\bigl[\thetab\fni(2s)\bigr]^\beta
		\ge \log\weight(s) + \beta\thetab\lambda\bigl[\log\weight(s)\bigr]^{1-\ib}
		\quad\text{for $s\in(0,t_0)$.}
\end{equation*}
Therefore,
we deduce from \eqref{E:exp-uk-lb} that
\begin{equation} \label{E:exp-uk-lb-2}
	\int_{\rn} \expb(\thetab|u_k|)\,\dgn
		\ge \phk \weight\bigl(\phk\bigr)
			\exp\left( 2\lambda \left[ \log\weight\bigl(\phk\bigr) \right]^{1-\ib}\right)
\end{equation}
for every $k$ obeying \eqref{E:Phk-condition}. Thanks to Lemma~\ref{L:weight-expansion}
and equation \eqref{E:logPhi},
\begin{equation*}
	\phk \weight\bigl(\phk\bigr)
		= \frac{1}{k^2} + \cdots
		\quad\text{as $k\to\infty$.}
\end{equation*}
and
\begin{equation*}
	\log\weight\bigl(\phk\bigr)
		= \frac{k^2}{2} + \cdots
		\quad\text{as $k\to\infty$.}
\end{equation*}
These asymptotic expansions ensure that the right-hand side of inequality
\eqref{E:exp-uk-lb-2} tends to infinity as $k\to\infty$.  Property
\eqref{E:diverges-for-series} hence follows.
\end{proof}

\begin{proposition} \label{P:sharpness-critical-positive}
Let $\beta\in(1,\infty)$. For any $N>0$, there exists $\tau_0>0$
and a Young function $B$ of the form \eqref{BN} such that
\begin{equation} \label{E:converges-for-all}
	\sup_u \int_{\rn} \expb(\tfrac 2\beta|u|)\,\dgn < \infty,
\end{equation}
where the supremum is extended over all functions $u\in\Wlgn\expLb$ fulfilling
conditions \eqref{E:mu} and \eqref{E:lgn-luxemburg}.
\end{proposition}

\begin{proof}
Thanks to Corollary~\ref{C:sup-estimate-expb},
\begin{equation} \label{E:converges-for-all-ub}
	\sup_u\, \int_{\rn} \expb(\thetab|u|)\,\dgn
		\le 2 \int_{0}^{\frac{1}{2}}
			\expb\bigl(\thetab\,\GG_{B,C}(s)\bigr)\,\d s,
\end{equation}
for any Young function $B$, where $\GG_{B,C}$ is the function defined by
\eqref{E:G-def} and $C$ the constant given by \eqref{E:C-mu}.  Let $\lambda>0$.
We claim that for any $N>0$, there exists $\tau_0>0$ and a Young function $B$
of the form \eqref{BN} such that
\begin{equation} \label{E:G-ub-lambda}
	\GG_{B,C}(s) \le \frac{\beta}{2} \left( \log\is \right)^\ib - \lambda
		\quad\text{for $s$ near zero}.
\end{equation}
Set $A(\tau)=\expb(\tau)$ for $\tau\ge 0$, and define the function
$\AA\colon[0,\infty)\to[0,\infty)$ by
\begin{equation*}
	\AA(\tau) =
	\begin{cases}
		\tau\frac{A(\tau_0)}{\tau_0}
			& \text{for $r\in[0,\tau_0)$}
			\\
		A(\tau)
			& \text{for $r\in[\tau_0,\infty)$},
	\end{cases}
\end{equation*}
where the number $\tau_0>0$ will be specified later. Observe that the function
$B$, defined as $B=N\AA$, is a Young function provided that $\tau_0$ is large
enough. The left-continuous derivative of $B$, denoted by $b$, obeys
\begin{equation*}
	b(\tau) =
	\begin{cases}
		N\frac{A(\tau_0)}{\tau_0}
			& \text{for $\tau\in(0,\tau_0]$}
			\\
		Na(\tau)
			& \text{for $\tau\in(\tau_0,\infty)$}
	\end{cases}
	\quad\text{and}\quad
	b^{-1}(t) =
	\begin{cases}
		0
			& \text{for $t\in\left( 0, N\frac{A(\tau_0)}{\tau_0} \right]$}
			\\
		\tau_0
			& \text{for $t\in\left( N\frac{A(\tau_0)}{\tau_0}, Na(\tau_0) \right]$}
			\\
		a^{-1}\left( \frac{t}{N} \right)
			& \text{for $t\in\left( Na(\tau_0), \infty \right)$},
	\end{cases}
\end{equation*}
where $a$ stands for the derivative of $A$ and $b^{-1}$ for the generalized
left-continuous inverse of $b$.  Lemmas~\ref{L:Orlicz-norm-weight}
and~\ref{L:lambda-asymptotics} tell us that
\begin{equation*}
	\normwB
		= \int_{0}^{\weight(s)} b^{-1}\left( \lambda_{\weight(s)}\tau \right)
			\tau I\bigl(\weight^{-1}(\tau)\bigr)^2\,\d\tau
		\quad\text{for $s\in(0,\tfrac12)$},
\end{equation*}
where the function $\lambda_{\weight(s)}\to\lambda$ as $s\to0^+$, for some
constant $\lambda>0$. From Lemma~\ref{L:asymptotic-of-product} we infer that
there exists $t_0>0$ such that
\begin{equation*}
	a^{-1}\left( \tfrac{2\lambda}{N}t \right)t I\bigl(\weight^{-1}(t)\bigr)^2
		\le \frac{1}{2t}(\log t)^{\ib-1}
		\quad\text{for $t>t_0$}.
\end{equation*}
Now, choose $\tau_0$ so large that $Na(\tau_0)>2\lambda t_0$,
and $s_0=s_0(\tau_0)$ in such a way that $\lambda_{\weight(s)}\in(\lambda,2\lambda)$
and $2\lambda\weight(s)>Na(\tau_0)$ for $s\in(0,s_0)$. Then,
\begin{align} \label{E:Orl-norm-weight-ub-2}
	\begin{split}
	\normwB
		& \le \int_{0}^{\weight(s)} b^{-1}\left( 2\lambda\tau \right)
			\tau I\bigl(\weight^{-1}(\tau)\bigr)^2\,\d\tau
			\\
		& = \tau_0 \int_{\frac{N}{2\lambda}\frac{A(\tau_0)}{\tau_0}}^{\frac{N}{2\lambda}a(\tau_0)}
			\tau I\bigl(\weight^{-1}(\tau)\bigr)^2\,\d\tau
			+ \int_{\frac{N}{2\lambda}a(\tau_0)}^{\weight(s)}
					a^{-1}\left( \frac{2\lambda}{N}\tau \right)
					\tau I\bigl(\weight^{-1}(\tau)\bigr)^2
			\\
		& \le \tau_0 \int_{\frac{N}{2\lambda}\frac{A(\tau_0)}{\tau_0}}^{\frac{N}{2\lambda}a(\tau_0)}
			\tau I\bigl(\weight^{-1}(\tau)\bigr)^2\,\d\tau
			+ \frac12 \int_{\frac{N}{2\lambda}a(\tau_0)}^{\weight(s)}
					\frac{1}{\tau}(\log\tau)^{\ib-1}\,\d\tau
			\\
		& = \frac{\beta}{2}\bigl[\log\weight(s)\bigr]^\ib
				+ \nu(\tau_0)
			\quad\text{for $s\in(0,s_0)$},
	\end{split}
\end{align}
where we have set
\begin{equation} \label{E:lambda-t0-def-1}
	\nu(\tau_0)
		= - \frac{\beta}{2}\left[\log\left( \frac{N}{2\lambda} a(\tau_0)\right)\right]^\ib
			+ \tau_0 \int_{\frac{N}{2\lambda}\frac{A(\tau_0)}{\tau_0}}^{\frac{N}{2\lambda}a(\tau_0)}
					\tau I\bigl(\weight^{-1}(\tau)\bigr)^2\,\d\tau.
\end{equation}
Observe that, thanks to Lemma~\ref{L:weight-expansion}, we may assume that
$s_0$ is so small that $\log\weight(s)\le \log\tis$ for $s\in(0,s_0)$.
Therefore, it suffices to show that $\nu(\tau_0)\to-\infty$ as
$\tau_0\to\infty$. To this purpose, first notice that, since $a(\tau)=\beta
\tau^{\beta-1}\expb(\tau)$,
\begin{equation}\label{july10}
	\left[\log\left( \frac{N}{2\lambda} a(\tau_0)\right)\right]^\ib
		= \tau_0 + \cdots
		\quad\text{as $\tau_0\to\infty$}.
\end{equation}
Next, let us show that the second addend on the right-hand side of
\eqref{E:lambda-t0-def-1} tends to $0$ as $\tau_0\to\infty$. From L'H\^opital's
rule and Lemma~\ref{L:Iwi2-expansion}, we deduce that
\begin{equation*}
	\int_{0}^{t} \tau I\bigl(\weight^{-1}(\tau)\bigr)^2\,\d\tau
		= \frac{1}{2}\log\log t + C + \cdots
			\ttoinf,
\end{equation*}
for some constant $C$. By formula ~\eqref{E:log-of-expansion},
\begin{equation*}
	\frac{1}{2}\log\log\left( \frac{N}{2\lambda} a(\tau_0)\right)
		= \frac{1}{2}\log\left( \tau_0^\beta + (\beta-1)\log \tau_0 + \cdots \right)
		= \frac{\beta}{2} \log \tau_0 + \frac{\beta-1}{2} \tau_0^{-\beta}\log \tau_0 + \cdots
		\quad\text{as $\tau_0\to\infty$}
\end{equation*}
and
\begin{equation*}
	\frac{1}{2}\log\log\left( \frac{N}{2\lambda} \frac{A(\tau_0)}{\tau_0}\right)
		= \frac{1}{2}\log\left( \tau_0^\beta -\log \tau_0 + \cdots \right)
		= \frac{\beta}{2} \log \tau_0 - \frac{1}{2} \tau_0^{-\beta}\log \tau_0 + \cdots
		\quad\text{as $t_0\to\infty$}.
\end{equation*}
Therefore
\begin{equation*}
	\tau_0 \int_{\frac{N}{2\lambda}\frac{A(\tau_0)}{\tau_0}}^{\frac{N}{2\lambda}a(\tau_0)}
		\tau I\bigl(\weight^{-1}(\tau)\bigr)^2\,\d\tau
		= \frac{\beta}{2} \tau_0^{1-\beta} \log \tau_0 + \cdots
		\quad\text{as $\tau_0\to\infty$.}
\end{equation*}
This asymptotic expansion ensures that the expression on the left-hand side
tends to $0$ as $\tau_0\to\infty$. This piece of information, combined with
equations \eqref{E:lambda-t0-def-1} and \eqref{july10}, ensures that
$\nu(\tau_0)\to-\infty$ as $\tau_0\to\infty$. Thus, given $\lambda>0$, there
exists $\tau_0>0$ such that $B$ is a Young function and, simultaneously,
$\nu(\tau_0)<-\lambda-C$. Estimate \eqref{E:Orl-norm-weight-ub-2} therefore yields
\begin{equation} \label{E:Orl-norm-weight-ub}
	\normwB \le \frac{\beta}{2}\left( \log\is \right)^\ib - \lambda - C
		\quad\text{for $s\in(0,s_0)$}
\end{equation}
for some $s_0=s_0(\tau_0)$. Next, owing to Lemma~\ref{L:second-term-expansion},
we have that $\weight(s)\,sB^{-1}(\tis)\to 0$ as $s\to0^+$. This fact, coupled
with inequality \eqref{E:Orl-norm-weight-ub}, implies that
\begin{equation*}
	\GG_{B,C}(s) =
	\normwB + \weight(s)\,sB^{-1}\left( \tis \right) + C
		\le \frac{\beta}{2} \left( \log\tis \right)^\ib - \lambda,
\end{equation*}
provided that $s$ is sufficiently small. Inequality \eqref{E:G-ub-lambda} is
thus established. From inequality \eqref{E:G-ub-lambda} we infer that
\begin{equation*}
	\bigl[\thetab\, \GG_{B,C}(s) \bigr]^\beta
		\le \log\tis - 2\lambda\left( \log\tis \right)^{1-\ib} + \cdots
		\stoo.
\end{equation*}
Thereby, the integral on the right-hand side of equation
\eqref{E:converges-for-all-ub} converges. Inequality
\eqref{E:converges-for-all} hence follows.
\end{proof}

\section{Proofs of Theorems~\ref{T:integral-form}, \ref{T:norm-form} and \ref{T:L-infty}}\label{proofmain}

The following lemmas are variants of \citep[Lemmas~6.1-6.3]{Cia:20b}.  There,
the function $\expb$ appears in the place of its convex envelope
$\Expb$, which enters
Lemmas~\ref{L:modular-implies-norm-to-B-is-M}-\ref{L:modular-implies-norm-to-M-is-B}
below. The proofs of the latter only require minor modifications, and will be
omitted.

\begin{lemma} \label{L:modular-implies-norm-to-B-is-M}
Let $\beta>0$.  Assume that $B$ is a Young function satisfying condition
\eqref{BN} for some $N\in(0,1)$. Then there exists a constant $M>1$ such that
\begin{equation} \label{E:phi-norm}
	\|\phi\|_{L^B\RG} \le 1
\end{equation}
for every function $\phi\in\MM(\rn)$ fulfilling
\begin{equation} \label{E:phi-integral}
	\int_{\rn} \Expb(|\phi|)\,\dgn \le M.
\end{equation}
\end{lemma}

\begin{lemma} \label{L:norm-implies-modular-to-B-is-M}
Let $\beta>0$.  Assume that $B$ is a Young function satisfying condition
\eqref{BN} for some $N>0$.  Then there exists a constant $M>1$ such that
inequality \eqref{E:phi-integral} holds for every function $\phi\in\MM(\rn)$
fulfilling condition \eqref{E:phi-norm}.
\end{lemma}

\begin{lemma} \label{L:modular-implies-norm-to-M-is-B}
Let $\beta>0$ and $M>1$. There exists a Young function $B$ of the form
\eqref{BN} such that inequality \eqref{E:phi-norm} holds
for every function $\phi\in\MM(\rn)$ satisfying
\eqref{E:phi-integral}.
\end{lemma}

We are now in a position to accomplish the proofs of our inequalities for the
Ornstein-Uhlenbeck operator.

\begin{proof}[Proof of Theorem~\ref{T:norm-form}]
Let $N>0$ and let $B$ be a Young function of the form \eqref{BN}. The choice
$\varphi (t) =1$ in Corollary~\ref{C:sup-estimate-expb} tells us  that
\begin{equation} \label{E:sup-holder-estimate}
	\sup_u\, \int_{\rn} \exp^\beta(\theta|u|)\,\dgn
		\le 2 \int_{0}^{\frac{1}{2}}
			\expb\bigl(\theta\,\GG_{B,C}(s)\bigr)\,\d s,
\end{equation}
where the supremum is extended over all functions $u\in\Wlgn\expLb$ satisfying
\eqref{E:lgn-luxemburg} and \eqref{E:mu}, and $C$ is defined by \eqref{E:C-mu}. Let $\thetab$ be the constant given by~\eqref{E:thetab}.

If $\beta\in(0,1)$ and $\theta=\thetab$, then Corollary~\ref{C:G-expansion}
ensures that  the integral on the right-hand side of
\eqref{E:sup-holder-estimate} converges, since
\begin{equation*}
	\int_{0} \exp\left( \log\is - \frac{2+3\beta}{2-2\beta}\log\log\is \right)\d s
		= \int_0 \left(\log \frac 1s\right)^{- \frac{2+3\beta}{2-2\beta}}\frac {\d s}s
		< \infty.
\end{equation*}

If $\beta=1$ and $\theta=\theta_{1}$, then Corollary~\ref{C:G-expansion} implies
that  the integral on the right-hand side of \eqref{E:sup-holder-estimate}
converges, inasmuch as
\begin{equation*}
	\int_{0} \exp\left( \log\is -  \frac{5}{4} \left( \log\log\is \right)^2 \right)\d s
		= \int_0 \exp\left(- \frac{5}{4} \left( \log \log \frac 1s\right)^2\right)\frac {\d s}s
		< \infty.
\end{equation*}
Part (1.i) is thus established.

Assume next that $\beta\in(1,\infty)$ and $\theta<\theta_\beta$. Then Part
(2.i) follows via Corollary~\ref{C:G-expansion}, which tells us that the
integral on the right-hand side of \eqref{E:sup-holder-estimate} converges,
since
\begin{equation*}
	\int_{0} \exp\left( \left( \frac{\theta'}{\thetab} \right)^\beta \log\is \right)\d s
		= \int_{0} s^{-\left( \frac{\theta'}{\thetab} \right)^\beta}\d s
		< \infty
\end{equation*}
for every $\theta'\in(\theta,\theta_\beta)$.

Parts (1.ii) and (2.iii) follow from Proposition~\ref{P:sharpness-supercritical}.
Finally, Part (2.ii) is a consequence of Propositions~\ref{P:sharpness-critical-positive}
and~\ref{P:sharpness-critical-negative}.
\end{proof}

\begin{proof}[Proof of Theorem~\ref{T:integral-form}]
Lemma~\ref{L:modular-implies-norm-to-M-is-B} tells us that there exists $N>0$
and a Young function $B$ of the form \eqref{BN} such that every function
$u\in\Wlgn\expLb$ satisfying \eqref{E:lgn-integral} also obeys
\eqref{E:lgn-luxemburg}.  Parts (1.i) and (2.i) thus follow from the
respective parts of Theorem~\ref{T:norm-form}.

Parts (1.ii) and (2.iii) are the subject of Proposition~\ref{P:sharpness-supercritical}.

Consider now Part (2.ii). Let $N\in(0,1)$. By Theorem~\ref{T:norm-form}, Part
(2.ii), there exists  a Young function $B$ of the form \eqref{BN} which renders
inequality \eqref{E:norm} true. On the other hand,
Lemma~\ref{L:modular-implies-norm-to-B-is-M} tells us that inequality
\eqref{E:norm} implies inequality \eqref{E:int} for some $M>1$. This
establishes the positive assertion of Part (2.ii). As for the negative
assertion, by Theorem~\ref{T:norm-form} Part (2.ii), for every $N>0$, there
exists a Young function $B$ such that inequality \eqref{E:norm} fails.
Lemma~\ref{L:norm-implies-modular-to-B-is-M} states that there exists $M>1$
such that inequality \eqref{E:lgn-luxemburg} holds for any function $u$
satisfying \eqref{E:lgn-integral}. Consequently, inequality \eqref{E:int}
fails as well.
\end{proof}

\begin{proof}[Proof of  Theorem~\ref{T:L-infty}]
Let $\eta \in (0, 2)$.
Corollary~\ref{C:sup-estimate-expexp} implies that
\begin{equation} \label{E:u-expexp}
	\sup_u\,\int_{\rn} \exp\exp(\eta|u|)\,\dgn
		\le 2 \int_{0}^{\frac12} \exp\exp(\eta \Lambda(s)+\eta C)\,\d s,
\end{equation}
where the supremum is extended over all functions $u$ satisfying conditions
\eqref{E:lgn-inf} and \eqref{E:mu}.  Thanks to Lemma~\ref{L:Lambda-expansion},
the integral on the right side of \eqref{E:u-expexp} converges, since
\begin{equation*}
	\int_{0} \exp\left( \exp\left( \frac{\eta'}{2} \log\log\is \right) \right)\d s <\infty,
\end{equation*}
if $\eta' \in (\eta, 2)$.
This Proves part (i).

In order to establish Part (ii), choose $\fno=1$ and $\so=1/2$ in  definition
\eqref{E:f1-def} of the function $\fni$. Let
$v\colon\rn\to\R$ be the function given  by
\begin{equation*}
	v(x) = \fni\bigl(2\phxi\bigr)
		\quad\text{for $x\in\rn$},
\end{equation*}
and let $w\colon\rn\to\R$ be the function defined as $w(x)=-v(-x)$ for $x\in\rn$.  By
Lemma~\ref{L:solution-formula}, we have that $v, w\in W^{2,2}\RG$. One can also
verify that the function $u\colon\rn\to\R$, defined as
\begin{equation*}
	u(x) =
	\begin{cases}
		v(x)
			& \text{for $x_1\ge 0$}
			\\
		w(x)
			& \text{for $x_1<0$},
	\end{cases}
\end{equation*}
also belongs to  $W^{2,2}\RG$, and
\begin{equation*}
	\lgn u(x) = - \sgn(x_1)
		\quad\text{for \ae $x\in\rn$}.
\end{equation*}
Therefore, the function $u$ satisfies condition \eqref{E:lgn-inf}. Moreover,
since $u$ is odd, $\med(u)=\mv(u)=0$.
One has that
\begin{equation} \label{E:u-expexp-critical}
	\int_{\rn} \exp\exp(2|u|)\,\dgn
		= 2 \int_{0}^\frac12 \exp\exp(2\fni(2s))\,\d s.
\end{equation}
We claim that
\begin{equation} \label{E:f1-expansion-Linfty}
	\fni(2s) = \frac{1}{2} \log\log\is + C + \cdots
		\stoo,
\end{equation}
for some constant $C>0$. Indeed, from the definition of $\fni$ one obtains that
\begin{align} \label{E:C-comp}
	\begin{split}
		\lim_{s\to 0^+} \left( \fni(2s) - \frac12\log\log\is \right)
			 & = \lim_{s\to 0^+} \left(
						\int_{s}^\frac12 \frac{r}{I(r)^2}\,\d r
						- \int_{s}^\frac12 \frac{1}{2r\log\ir}\,\d r
						- \frac{1}{2}\log\log 2
					\right)
				\\
			& = - \frac{1}{2}\log\log 2
					+ \int_{0}^\frac12 \left(
							\frac{r}{I(r)^2} - \frac{1}{2r\log\ir}
					\right)\d r.
	\end{split}
\end{align}
Notice that   the last integral  converges, since,
by Lemma~\ref{L:I-expansion} and equation \eqref{E:power},
\begin{equation*}
	\frac{r}{I(r)^2}
		= \frac{1}{2r\log\ir}
			+ \frac{\log\log\ir}{4r(\log\ir)^2}
			+ \cdots
		\quad\text{as $r\to 0^+$.}
\end{equation*}
Furthermore, Lemma~\ref{L:I-inequality} implies that the  integral in question
is positive. Equation \eqref{E:C-comp} hence follows, since $\tfrac12\log\log 2<0$.
Equations \eqref{E:u-expexp-critical} and \eqref{E:f1-expansion-Linfty}
imply that
\begin{equation*}
	\int_{\rn} \exp\exp(2|u|)\,\dgn
		= \infty.
\end{equation*}
This concludes the proof of Part (ii).
\end{proof}

\section{Improved inequalities and  existence of maximizers}\label{final}

The key step in our proof of Theorem~\ref{T:maximizers} on the existence of
extremals in inequalities \eqref{E:int} and \eqref{E:norm}, for $\beta \in (0,
1]$ and $\theta =\frac 2\beta$, is an improved version of these
inequalities. The improvement amounts to allowing integrands of the function
$u$ in \eqref{E:int} and \eqref{E:norm} which grow slightly faster than
$\expb(\frac 2\beta |u|)$ as $|u|$ tends to $\infty$. Namely, we show that if
$\EF\colon[0,\infty)\to[0,\infty)$  is an increasing function that diverges to
$\infty$ as $t\to\infty$ with a sufficiently mild growth, then
\begin{equation}\label{E:sup-improved}
	\sup_u \int_{\rn} \expb(\tfrac 2\beta |u|)\,\EF (|u|)\,\dgn
		< \infty,
\end{equation}
where the supremum is extended over all functions $u\in\WexpLb$ satisfying
condition \eqref{E:mu}, and either \eqref{E:lgn-integral} or
\eqref{E:lgn-luxemburg}. This is the content of the next result, of
independent interest, where an explicit condition on the function $\EF$ for
inequality \eqref{E:sup-improved} to hold is offered. Let us emphasize that, by
contrast, Adams' inequality \eqref{supadams} in Euclidean domains does not
admit any enhancement of this kind.

\begin{theorem}[Improved integrability] \label{T:integral-form-improved}
Let $\EF\colon[0,\infty)\to[0,\infty)$ be a non-decreasing function.  Assume
that either $\beta\in(0,1)$ and
\begin{equation} \label{E:improvement-conditions-1}
	\int^\infty t^{-1-\frac{5\beta^2}{2-2\beta}+\varepsilon}\,\EF(t) \,\d t
		< \infty,
\end{equation}
or $\beta=1$ and
\begin{equation} \label{E:improvement-conditions-2}
		\int^\infty \exp\Bigl(
				\left( -\tfrac{5}{4}+\varepsilon \right)(\log t)^2
			\Bigr)\, \EF(t) \,\d t <\infty
\end{equation}
for some $\varepsilon>0$. Then, inequality \eqref{E:sup-improved} holds for
any constant $M>1$, or for any Young function $B$ obeying \eqref{BN}, according
to whether $u$ is subject to constraint \eqref{E:lgn-integral} or
\eqref{E:lgn-luxemburg}.
\end{theorem}

\begin{proof}
By Lemma~\ref{L:modular-implies-norm-to-M-is-B}, there exists a Young function
$B$ obeying \eqref{BN} such that condition \eqref{E:lgn-integral} implies
\eqref{E:lgn-luxemburg} for any $u\in\Wlgn\expLb$. Therefore, it suffices to
prove inequality \eqref{E:sup-improved} under the assumption that the supremum
is extended over all functions $u$ subject to constraint
\eqref{E:lgn-luxemburg}.

Corollary~\ref{C:sup-estimate-expb} yields
\begin{equation} \label{E:sup-estimate-expb-2}
	\sup_u\, \int_{\rn} \expb(\thetab|u|)\,\EF(|u|)\,\dgn
		\le 2 \int_{0}^{\frac{1}{2}}
			\expb\bigl(\thetab\,\GG_{B,C}(s)\bigr)\,\EF\bigl(\GG_{B,C}(s)\bigr)\,\d s,
\end{equation}
where the supremum is extended over all functions $u\in\WexpLb$ satisfying
conditions \eqref{E:mu} and \eqref{E:lgn-luxemburg}. Here, $\GG_{B,C}$ denotes
the function defined by \eqref{E:G-def} and $C$ is the constant given by
\eqref{E:C-mu}.  From Lemmas~\ref{L:second-term-expansion} and
\ref{L:Orlicz-norm} we infer that
\begin{equation} \label{E:G-ub}
	\GG_{B,C}(s) \le \frac{\beta}{2} \left( \log\is \right)^\ib
	\quad\text{for $s$ near zero}.
\end{equation}
Furthermore, Corollary~\ref{C:G-expansion} implies that, for any $\delta>0$,
\begin{equation} \label{E:G-ub-power}
	\bigl[ \thetab\,\GG_{B,C}(s) \bigr]^\beta
		\le \log\is +
		\begin{cases}
			\left(\delta - \frac{2+3\beta}{2-2\beta}\right) \log\log\is
				& \text{if $\beta\in(0,1)$}
				\\
			\left(\delta - \frac{5}{4}\right) (\log\log\is)^2
				& \text{if $\beta=1$}
		\end{cases}
	\quad\text{for $s$ near zero}.
\end{equation}
Assume first that $\beta\in(0,1)$. By equations \eqref{E:G-ub} and
\eqref{E:G-ub-power}, the integral on the right-hand side of
\eqref{E:sup-estimate-expb-2} converges if
\begin{equation*}
	\int_{0} \exp\left(
			\log\tis + \left( \delta - \tfrac{2+3\beta}{2-2\beta} \right)\log\log\tis
		\right) \EF\left( \tfrac{\beta}{2}\left( \log\tis \right)^\ib \right)\d s
		< \infty.
\end{equation*}
As shown by a change of variables, this condition is equivalent to
\begin{equation*}
	\int^\infty t^{-1-\frac{5\beta^2}{2-2\beta}+\beta\delta}\,\EF(t) \,\d t <\infty.
\end{equation*}
Thanks to assumption \eqref{E:improvement-conditions-1}, the latter condition
holds with $\delta=\varepsilon/\beta$.

Assume now that $\beta=1$. Owing to equations \eqref{E:G-ub} and
\eqref{E:G-ub-power} again, the integral on the right-hand side of
\eqref{E:sup-estimate-expb-2} converges provided that
\begin{equation*}
	\int_{0} \exp\left(
			\log\tis + \left( \delta - \tfrac{5}{4} \right)\left( \log\log\tis \right)^2
		\right) \EF\left( \tfrac{1}{2}\log\tis \right)\d s
		< \infty.
\end{equation*}
After changing  variables, this condition turns out to be equivalent to
\begin{equation*}
	\int^\infty \exp\left(
			\left( \delta-\tfrac{5}{4} \right)(\log t + \log 2)^2
		\right)\,\EF(t) \,\d t <\infty,
\end{equation*}
which, by assumption \eqref{E:improvement-conditions-2}, holds for any
$\delta<\varepsilon$.
\end{proof}

Besides Theorem~\ref{T:integral-form-improved}, a few (special cases) of
classical results of functional analysis are needed in the proof of Theorem
\ref{T:maximizers}. They are recalled below, for the reader's convenience.

\begin{theoremalph}[Equi-integrability (de la Vall\'ee-Poussin)]
\label{TA:delaVallePoussin}
Let $(\RR,\nu)$ be a probability space. Then a sequence $\{u_k\}\subset
L^1(\RR,\nu)$ is equi-integrable if and only if there exists a continuous
function $\psi\colon[0,\infty)\to[0,\infty)$, satisfying $\lim_{t\to\infty}
\psi(t)/t=\infty$, such that
\begin{equation*} 
	\sup_{k} \int_{\RR} \psi(|u_k|)\,\d\nu < \infty\,.
\end{equation*}
\end{theoremalph}

\begin{theoremalph}[Convergence in $L^1$ (Vitali)]
\label{TA:Vitali}
Let $(\RR,\nu)$ be a probability space, let $u\in L^1(\RR,\nu)$ and let
$\{u_k\}$ be a sequence in $L^1(\RR,\nu)$ such that $u_k\to u$ in measure.  If
the sequence $\{u_k\}$ is equi-integrable, then $u_k\to u$ in $L^1(\RR,\nu)$.
\end{theoremalph}

The semicontinuity theorem stated below is a consequence of
\citep[Theorem 8, Section 1.1]{Gia:98a}.

\begin{theoremalph}[Semicontinuity]
\label{TA:Serrin}
Let $g\colon\rn\to[0,\infty)$ be a convex function. Then the functional
defined as \begin{equation*} \int_{\rn} g(|u|) \,\dgn \end{equation*} is
sequentially lower semicontinuous with respect to the weak$^*$-convergence in
$L^1\RG$.
\end{theoremalph}

The next result is a special case of  the  Banach-Alaoglu theorem. With this
regard, note that the space $\expLb$ is the dual of the separable space
$L(\log L)^\ib\RG$.

\begin{theoremalph}[Weak$^{*}$ compactness in $\expLb$ (Banach-Alaoglu)]
\label{TA:Banach-Alaoglu}
Assume that $\{u_k\}$ is a bounded sequence in $\expLb$. Then there exist a
function $u\in\expLb$ and a subsequence of $\{u_k\}$, still denoted by
$\{u_k\}$, such that $u_{k}\rightharpoonup u$ in the weak$^*$-topology of
$\expLb$.
\end{theoremalph}

We conclude our preliminaries by stating a result proved in
\citep[Lemma~4.1]{Cia:20c} on the convergence of medians and mean values.

\begin{lemma} \label{L:mu-limit}
Let $u \in W^{1,1}\RG$ and let
$\{u_k\}$ be a sequence in $W^{1,1}\RG$
such that $u_k\to u$ in $L^1\RG$.
Then, there exists a subsequence of $\{u_k\}$, still denoted by $\{u_k\}$, such that
\begin{equation*} 
	\lim_{k\to\infty} \m(u_k) = \m(u).
\end{equation*}
Here, $\m(\cdot)$ stands for either mean value or median.
\end{lemma}

\begin{proof}[Proof of Theorem~\ref{T:maximizers}]
We treat Parts (i) and (ii) simultaneously, since  their proofs only require
minor variants. Let $M>1$, or let $B$ be a Young function satisfying equation
\eqref{BN}. Denote by $\{u_k\}$ a maximizing sequence for \eqref{E:int}, or
for \eqref{E:norm}, and by $S$ the supremum in \eqref{E:int} or
\eqref{E:norm}.  Set $f_k=\lgn u_k$ for $k\in\N$. Thus, $\m(u_k)=0$,
\begin{equation} \label{E:fk-int}
	\int_{\rn} \expb(|f_k|)\,\dgn
		\le M,
\end{equation}
or
\begin{equation} \label{E:fk-norm}
	\|f_k\|_{L^B\RG} \le 1
\end{equation}
for $k\in\N$, and
\begin{equation}\label{july21}
	\lim_{k\to\infty} \int_{\rn}
		\expb\left(\thetab |u_k|\right)\dgn
		= S,
\end{equation}
where $\thetab$ is defined by~\eqref{E:thetab}. Throughout this proof, indices will be preserved when passing to a
subsequence, for simplicity of notation. Owing to
Lemma~\ref{L:modular-implies-norm-to-M-is-B}, condition \eqref{E:fk-int}
implies \eqref{E:fk-norm} for some (possibly different) Young function $B$
fulfilling \eqref{BN}. We claim that
\begin{equation} \label{E:nabla-u2-both-k}
	\int_{\rn} |\nabla u_k|^2\,\dgn
		\le 2 \int_{0}^{\frac12}
			\left( \frac{1}{I(s)} \int_{0}^{s} f_k^*(r)\, \d r \right)^2\d s
	\quad \text{for $k\in\N$}.
\end{equation}
Here, $f_k^*$ denotes the decreasing rearrangement of $f_k$.
Since the operators $\nabla$ and $\lgn$ are invariant under additive
constants, it suffices to prove  our claim under the assumption that $\med
(u_k)=0$. Set $\mu_k(t)=\gamma_n(\{x\in\rn:u_k(x)>t\})$ for $k\in\N$. Thanks
to  \citep[Proof of Theorem~3.1]{Cia:20a},
\begin{equation} \label{E:nabla-u2-ineq}
	-\frac{\d}{\d t} \int_{\{u_k>t\}} |\nabla u_k|^2\,\dgn
		\le - \frac{\mu_k'(t)}{I\bigl(\mu_k(t)\bigr)^2}
					\left( \int_{0}^{\mu_k(t)} (f_k)^*_+ (r)\, \d r\right)^2
			\quad\text{for $t>0$},
\end{equation}
for $k\in\N$. Integrating both sides of inequality \eqref{E:nabla-u2-ineq}
over $(0,\infty)$ tells us that
\begin{equation} \label{E:nabla-u2-int}
	\int_{\{u_k>0\}} |\nabla u_k|^2\,\dgn
		\le \int_{0}^{\frac12}
			\left( \frac{1}{I(s)} \int_{0}^{s} (f_k)^*_+ (r)\, \d r\right)^2\d s
	\quad\text{for $k\in\N$}.
\end{equation}
Note that the derivation of inequality \eqref{E:nabla-u2-int} also rests upon
the change of variables $s=\mu_k(t)$ and on the inequality
$\mu_k(0)\le\tfrac12$. Combining inequality \eqref{E:nabla-u2-int} with an
analogous estimate on the set $\{u_k<0\}$, with $(f_k)_+$ replaced by
$(f_k)_-$, yields
\begin{equation}
\label{E:nabla-u2-both}
	\int_{\rn} |\nabla u_k|^2\,\dgn
		\le  \int_{0}^{\frac12}
			\left( \frac{1}{I(s)} \int_{0}^{s} \left(f_k \right)^*_+(r)\,\d r \right)^2\d s
			+
			\int_{0}^{\frac12}
			\left( \frac{1}{I(s)} \int_{0}^{s} \left(f_k\right)^*_-(r)\,\d r \right)^2\d s
\end{equation}
for $k\in\N$.
Inequality \eqref{E:nabla-u2-both-k} follows from \eqref{E:nabla-u2-both}.

By H\"older's inequality in the form \eqref{E:Orl-Holder}
and equation \eqref{E:Orl-char},
\begin{equation} \label{E:f-star-lgn-estimate}
	\int_{0}^{s}  f_k^*(r)\, \d r
		\le \|f_k^*\|_{L^B(0,1)} \onorm{\chi_{(0,s)}}_{L^{\tilde B}(0,1)}
		= \|f_k\|_{L^B\RG} s B^{-1}\left( \tis \right)
	\quad \text{for $s\in(0,1)$}.
\end{equation}
Coupling inequalities \eqref{E:nabla-u2-both-k} and
\eqref{E:f-star-lgn-estimate} yields
\begin{equation} \label{E:L2-norm-u}
	\|\nabla u_k\|_{L^2\RG}
		\le \|f_k\|_{L^B\RG}
		\left( 2\int_{0}^{\frac12}
			\left( \frac{sB^{-1}\left( \is \right)}{I(s)}
			\right)^2\d s
		\right)^{\frac{1}{2}}
	\quad\text{for $k\in\N$}.
\end{equation}
Owing to equations \eqref{BN} and \eqref{E:I-expansion},
the integral on the right-hand side of \eqref{E:L2-norm-u}
converges. Hence, there exists a constant $C$ such that
\begin{equation} \label{E:uk-delta}
	\|\nabla u_k\|_{L^2\RG} \le C
		\quad\text{for $k\in\N$}.
\end{equation}
Since the embedding $W^{1,2}\RG \to L^1\RG$ is compact - see \eg
\citep[Theorem~7.3]{Sla:15} - there exist a function $u\in L^1\RG$ and a
subsequence of $\{u_k\}$ such that $u_k\to u$ in $L^1\RG$ and
\begin{equation} \label{E:uk-converges-ae}
	u_k(x)\to u(x)
		\quad\text{for \ae $x\in\rn$}.
\end{equation}
Lemma~\ref{L:mu-limit} ensures that, on taking a subsequence, if necessary,
$\m(u_k)\to\m(u)$, whence $\m(u)=0$.  Next, by inequality
\eqref{E:uk-delta} and   the reflexivity of the space $L^2\RG$, there exist a
subsequence of $\{u_k\}$ and a function $V\colon\rn\to\rn$ such that $V\in
L^2\RG$ and $\nabla u_k\rightharpoonup V$ weakly in $L^2\RG$. Hence,
$u \in W^{1,2}\RG$ and $\nabla u=V$.  Also, by inequality
\eqref{E:fk-norm} and Theorem~\ref{TA:Banach-Alaoglu}, there exist a function
$f\in\expLb$ and a subsequence of $\{f_k\}$ such that $f_k\rightharpoonup f$ in
the weak$^*$-topology of $\expLb$. In particular, $f_k\rightharpoonup f$ weakly
in $L^2\RG$. By equation \eqref{E:weak-formulation}, given any function
$v\in W^{1,2}\RG$,  one has that
\begin{equation} \label{E:lgn-def-uk}
	\int_{\rn} \nabla u_k \cdot \nabla v\,\dgn
		= - \int_{\rn} u_k\, v\,\dgn
	\quad \text{for $k\in\N$}.
\end{equation}
Passing to the limit as $k\to\infty$ in equation \eqref{E:lgn-def-uk} tells
us that $\lgn u=f$. Furthermore, inasmuch as $f_k\rightharpoonup f$ weakly$^*$
in $\expLb$, we have that $f_k\rightharpoonup\lgn u$ in the weak$^*$-topology
of $L^1\RG$.  Therefore, if constraint \eqref{E:fk-int} is in force, then by
Theorem~\ref{TA:Serrin}
\begin{equation*}
	\int_{\rn} \expb(|\lgn u|)\,\dgn
		\le\liminf_{k\to\infty} \int_{\rn} \expb(|f_k|)\,\dgn
		\le M
	\quad\text{for $k\in\N$}.
\end{equation*}
On the other hand, under constraint \eqref{E:fk-norm}, by the weak$^*$ lower
semicontinuity of the norm in $L^B\RG$, \eqref{E:fk-norm},
\begin{equation*}
	\|\lgn u\|_{L^B\RG}
		\le \liminf_{k\to\infty} \|f_k\|_{L^B\RG}
		\le 1.
\end{equation*}

We conclude by showing that
\begin{equation} \label{E:u-attains-sup-int}
	\int_{\rn} \expb\left( \thetab|u| \right)\,\dgn
		= S.
\end{equation}
Thanks to Theorem~\ref{T:integral-form-improved}, there exists a continuous
increasing function $\EF\colon[0,\infty)\to[0,\infty)$ such that
$\lim_{t\to\infty}\EF(t)=\infty$ and
\begin{equation} \label{sep3}
	\int_{\rn} \expb\left( \thetab|u_k| \right)
		\EF\left( |u_k| \right)\dgn
		\le C \quad \text{for $k\in \N$,}
\end{equation}
for some constant $C$. Consider the function
$\psi\colon[0,\infty)\to[0,\infty)$ defined as
\begin{equation*}
	\psi(t)
		= t \,\EF\left( \frac{1}{\thetab}(\log t)^\ib\right)
		\quad\text{for $t\ge 1$}
\end{equation*}
and $\psi(t)=0$ for $t\in[0,1)$. Then, $\lim_{t\to\infty}\psi(t)/t=\infty$.
Moreover, by inequality \eqref{sep3},
\begin{align*}
	\sup_{k\in\N} \int_{\rn} \psi \left(
		\expb\left( \thetab|u_k| \right)\right)\dgn
		= \sup_{k\in\N}
				\int_{\rn} \expb\left( \thetab|u_k| \right)
					\EF\left( |u_k| \right)\dgn
		\le C.
\end{align*}
As a consequence of Theorem~\ref{TA:delaVallePoussin}, the sequence of
functions $\{\expb\left( \thetab|u_k| \right)\}$ is equi-integrable in
$L^1\RG$. Also, by property \eqref{E:uk-converges-ae}, this sequence converges
\ae in $\rn$, and hence  in measure, to the function $\expb(\thetab|u|)$. Thus,
from equation \eqref{july21} and of Theorem~\ref{TA:Vitali}, one
infers that
\begin{equation*}
	S =	\lim_{k\to\infty}
				\int_{\rn} \expb\left( \thetab|u_k|\right)\dgn
		= \int_{\rn} \expb\left( \thetab|u|\right)\dgn,
\end{equation*}
whence equation \eqref{E:u-attains-sup-int} follows. Altogether, we have
shown that $u$ is a maximizer for \eqref{E:int} or \eqref{E:norm}.
\end{proof}


\paragraph*{Funding}

This research was partly funded by:

\begin{enumerate}
\item Research Project 2201758MTR2  of the Italian Ministry of University and
Research (MIUR) Prin 2017 ``Direct and inverse problems for partial differential equations: theoretical aspects and applications'';
\item GNAMPA of the Italian INdAM -- National Institute of High Mathematics
(grant number not available);
\item Grant P201-18-00580S of the Czech Science Foundation.
\end{enumerate}

\paragraph*{Declaration of competing interest}

The authors declare that they have no conflict of interest.

\paragraph*{Acknowledgement}  We wish to thank the referee for his/her careful reading of the paper and for his/her valuable comments and suggestions.

\begin{small}
\bibliographystyle{abbrv_doi}

\end{small}

\end{document}